\documentclass[a4paper, english]{amsart}

\usepackage{amsmath, amsthm, amsfonts, amssymb}
\usepackage{enumitem}
\usepackage{xcolor}
\usepackage{hyperref}
\hypersetup{
    colorlinks,
    linkcolor={red!50!black},
    citecolor={blue!50!black},
    urlcolor={blue!80!black}
}

\usepackage{tikz-cd}
\usepackage{stackengine}

\usepackage{overpic}
\usepackage[ansinew]{inputenc}
\usepackage{graphicx}
\usepackage[rightcaption]{sidecap}
\usepackage{caption}
\usepackage{subcaption}

\theoremstyle{plain}
\newtheorem{theorem}{Theorem}[section]
\newtheorem{corollary}[theorem]{Corollary}
\newtheorem{lemma}[theorem]{Lemma}
\newtheorem{proposition}[theorem]{Proposition}

\newtheorem{claim}[theorem]{Claim}
\newtheorem{fact}[theorem]{Fact}

\newtheorem{thmintro}{Theorem}

\theoremstyle{definition}
\newtheorem{definition}[theorem]{Definition}
\newtheorem{remark}[theorem]{Remark}

\makeatletter
\newcommand{\setword}[2]{%
  \phantomsection
  #1\def\@currentlabel{\unexpanded{#1}}\label{#2}%
}
\makeatother

\newcommand{\eps}{\varepsilon}

\newcommand{\wt}[1]{\widetilde{#1}}
\newcommand{\R}{\mathbb R}

\newcommand{\ZZ}{\mathbb Z}

\newcommand{\gbl}{g_{\beta,L_1}}

\newcommand{\cT}{\mathcal{T}}
\newcommand{\cW}{\mathcal{W}}
\newcommand{\cF}{\mathcal{F}}
\newcommand{\cL}{\mathcal{L}}
\newcommand{\cC}{\mathcal{C}}
\newcommand{\cS}{\mathcal{S}}

\newcommand{\lt}{\widetilde {\mathcal{T}}}
\newcommand{\mt}{\widetilde M}
\newcommand{\ft}{\widetilde f}

\newcommand{\TT}{\mathbb{T}}

\newcommand{\tf}{\widetilde f}
\newcommand{\fol}{\cF}

\newcommand{\fn}{\widetilde \cF}

\newcommand{\cG}{\mathcal G}

\newcommand{\wcs}{\widetilde \cW^{\mathrm{cs}}}
\newcommand{\wcu}{\widetilde \cW^{\mathrm{cu}}}

\newcommand{\wphi}{\widetilde \Phi}

\newcommand{\cs}{\mathrm{cs}}
\newcommand{\cu}{\mathrm{cu}}

\newcommand{\cfix}{ \mathrm{Fix}^{\mathrm{\mathrm{c}}}_{\widetilde f}}

\newcommand{\rquotient}[2]{{\left.\raisebox{.1em}{$#1$}\middle/\raisebox{-.1em}{$#2$}\right.}}

\title[Partial hyperbolicity in $3$-manifolds, Part I]{Partially hyperbolic diffeomorphisms homotopic to 
the identity in dimension $3$ \\ Part I: the dynamically coherent case}
\author[T. Barthelm\'e]{Thomas Barthelm\'e}
\address{Queen's University, Kingston, ON}
\email{thomas.barthelme@queensu.edu}
\urladdr{sites.google.com/site/thomasbarthelme}

\author[S.R. Fenley]{Sergio R.\ Fenley} 
\address{Florida State University, Tallahassee, FL 32306}
\email{fenley@math.fsu.edu}

\author[S. Frankel]{Steven Frankel} 
\address{Washington University in St.~Louis, St.~Louis, Mo}
\email{steven.frankel@wustl.edu}

\author[R. Potrie]{Rafael Potrie} 
\address{Centro de Matem\'atica, Universidad de la Rep\'ublica, Uruguay}
\email{rpotrie@cmat.edu.uy}
\urladdr{http://www.cmat.edu.uy/~rpotrie/}

\begin{document}
%  \frontmatter
 
 \begin{abstract}
We study 3-dimensional dynamically coherent partially hyperbolic diffeomorphisms that are homotopic to the identity, focusing on the transverse geometry and topology of the center-stable and center-unstable foliations, and the dynamics within their leaves. We find a structural dichotomy for these foliations, which we use to show that every such diffeomorphism on a hyperbolic or Seifert-fibered 3-manifold is leaf-conjugate to the time-one map of a (topological) Anosov flow. This proves a classification conjecture of Hertz--Hertz--Ures in hyperbolic 3-manifolds and in the homotopy class of the identity of Seifert manifolds. 
\end{abstract}
% 
% \begin{altabstract}
%  Nous \'etudions, en dimension trois, les diff\'eomorphismes partiellement hyperboliques qui sont dynamiquement coh\'erents et homotopes \`a l'identit\'e. Nous nous focalisons sur la g\'eom\'etrie et la topologie de leurs feuilletages centraux-stable et centraux-instable, ainsi que sur leur dynamique dans les feuilles. Nous obtenons ainsi que la structure de ces feuilletages doit satisfaire \`a une dichotomie. En utilisant cette dichotomie, nous montrons que, lorsque la 3-vari\'et\'e est hyperbolique ou de Seifert, les diff\'eomorphismes \'etudi\'es ont leur feuilletage central conjugu\'e \`a celui du temps un d'un flot d'Anosov topologique. Ceci prouve une conjecture de Hertz--Hertz--Ures pour les vari\'et\'es hyperboliques et dans la classe d'homotopie de l'identit\'e pour les vari\'et\'es de Seifert.
% \end{altabstract}


 
% \altkeywords{Hyperbolicit\'e partielle, topology des 3-vari\'et\'es, feuilletages, classification} 

\thanks{
T.~Barthelm\'e was partially supported by the NSERC (Funding reference number RGPIN-2017-04592).
S.~Fenley was partially supported by Simons Foundation grant number 280429.
S.~Frankel was partially supported by National Science Foundation grant number~DMS-1611768.
R.~Potrie was partially supported by CSIC~618 and ANII--FCE--135352.}
 \maketitle

%\tableofcontents

% \mainmatter

\section{Introduction}\label{s.introduction}

A diffeomorphism $f$ of a $3$-manifold $M$ is \emph{partially hyperbolic} if it preserves a splitting of the tangent bundle $TM$ into three $1$-dimensional sub-bundles
\[ TM = E^{\mathrm{s}} \oplus E^{\mathrm{c}} \oplus E^{\mathrm{u}}, \]
where the \emph{stable bundle} $E^{\mathrm{s}}$ is eventually contracted, the \emph{unstable bundle} $E^{\mathrm{u}}$ is eventually expanded, and the \emph{center bundle} $E^{\mathrm{c}}$ is distorted less than the stable and unstable bundles at each point.

From a dynamical perspective, the interest in partial hyperbolicity stems from its appearance as a generic consequence of certain dynamical conditions, such as stable ergodicity and robust transitivity. For example, a diffeomorphism is \emph{transitive} if it has a dense orbit, and \emph{robustly transitive} if this behavior persists under $C^1$-small deformations. Every robustly transitive diffeomorphism on a $3$-manifold is either Anosov or ``weakly'' partially hyperbolic \cite{DPU}. Analogous results hold for stable ergodicity and in higher dimensions \cite{BDV}.

From a geometric perspective, one can think of partial hyperbolicity as a generalization of the discrete behavior of Anosov flows, which feature prominently in the theory of $3$-manifolds. Recall that a flow $\Phi$ on a $3$-manifold $M$ is \emph{Anosov} if it preserves a splitting of the unit tangent bundle $TM$ into three $1$-dimensional sub-bundles
\[ TM = E^{\mathrm{s}} \oplus T\Phi \oplus E^{\mathrm{u}}, \]
where $T\Phi$ is the tangent direction to the flow, $E^{\mathrm{s}}$ is eventually exponentially contracted, and $E^{\mathrm{u}}$ is eventually exponentially expanded. After flowing for a fixed time, an Anosov flow generates a partially hyperbolic diffeomorphism of a particularly simple type, where the stable and unstable bundles are contracted uniformly, and the center direction, which corresponds to $T\Phi$, is left undistorted. More generally, one can construct partially hyperbolic diffeomorphisms of the form $f(x) = \Phi_{\tau(x)}(x)$ where $\Phi$ is a (topological) Anosov flow and $\tau\colon M \to \mathbb{R}_{>0}$ is a positive continuous function; the partially hyperbolic diffeomorphisms obtained in this way are called \emph{discretized Anosov flows}\footnote{Note that a discretized Anosov flow is not in general the time-$1$ map of a reparametrization of the Anosov flow, see Appendix \ref{ss.DAF}.}.

In this article and its sequel \cite{BFFP-sequel} we show that large classes of partially hyperbolic diffeomorphisms can be identified with discretized Anosov flows. This confirms a large part of the well-known conjecture by Pujals \cite{BW}, which attempts to classify $3$-dimensional partially hyperbolic diffeomorphisms by asserting that they are all either discretized Anosov flows or deformations of certain kinds of algebraic examples.

\subsection{Homotopy, integrability, and conjugacy}
There are two important obstructions to identifying a partially hyperbolic diffeomorphism with a discretized Anosov flow. The first comes from the fact that the latter are homotopic to the identity, while the former may be homotopically non-trivial. Homotopically non-trivial examples include Anosov diffeomorphisms on the $3$-torus with distinct eigenvalues, ``skew products,'' and the counterexamples to Pujals' conjecture constructed in \cite{BPP,BGP,BZ,BGHP}.

The second major obstruction comes from the integrability of the bundles in a partially hyperbolic splitting. In the context of an Anosov flow $\Phi$, the stable and unstable bundles $E^{\mathrm{s}}$ and $E^{\mathrm{u}}$ integrate uniquely into a pair of $1$-dimensional foliations, the \emph{stable} foliation $\cW^{\mathrm{s}}$ and \emph{unstable} foliation $\cW^{\mathrm{u}}$. In fact, even the \emph{weak stable} and \emph{weak unstable} bundles $E^{\mathrm{s}} \oplus T\Phi$ and $E^{\mathrm{u}} \oplus T\Phi$ integrate uniquely into a transverse pair of $\Phi$-invariant $2$-dimensional foliations, the \emph{weak stable} foliation $\cW^{ws}$ and \emph{weak unstable} foliation $\cW^{wu}$.

In the context of a partially hyperbolic diffeomorphism $f$, the stable and unstable bundles still integrate uniquely into stable and unstable foliations, $\cW^{\mathrm{s}}$ and $\cW^{\mathrm{u}}$ \cite{HP-survey}. However, the \emph{center-stable} and \emph{center-unstable} bundles $E^{\mathrm{c}} \oplus E^{\mathrm{s}}$ and $E^{\mathrm{c}} \oplus E^{\mathrm{u}}$ may fail to be uniquely integrable. In fact, there are examples where it is impossible to find \emph{any} $f$-invariant $2$-dimensional foliation tangent to the center-stable or center-unstable bundle \cite{HHU-noncoherent, BGHP}.

If one can find a pair of $f$-invariant foliations tangent to the center-stable and center-unstable bundles then $f$ is said to be \emph{dynamically coherent}. This condition is certainly satisfied if $f$ is a discretized Anosov flow (cf.~Appendix~\ref{ss.DAF}).

We take dynamical coherence as an assumption in the present article; Part II \cite{BFFP-sequel} works without this.

\subsection{Results}
Most of the existing progress towards classifying partially hyperbolic diffeomorphisms takes an outside-in approach, restricting attention to particular classes of manifolds, and comparing to an \emph{a priori} known model partially hyperbolic (see \cite{CHHU,HP-survey} for recent surveys). In particular, partially hyperbolic diffeomorphisms have been completely classified in manifolds with solvable or virtually solvable fundamental group \cite{HP-Nil,HP-Sol}. Here, by classification we mean both the description of the topology of manifolds and isotopy classes admitting such dynamics as well as the production of topological models that describe such systems. 

Ours is an inside-out approach, using the theory of foliations to understand the way the local structure that defines partial hyperbolicity is pieced together into a global picture. We then relate the dynamics of these foliations to the large-scale structure of the ambient manifold. An advantage of this method is that, since it does not rely on a model partially hyperbolic to compare to, we can consider any manifold, not just one where an Anosov flow is known to exist. Note that here it will be possible to construct a topological model even if the manifold is not known to admit an Anosov flow, nor if such flow is unique. 

The following two theorems are the main consequences of our work, applied to two of the major classes of $3$-manifolds. Note that the classification of partially hyperbolic diffeomorphisms is always considered up to finite lifts and iterates, since one can easily build infinitely many different but uninteresting examples by taking finite covers.

\begin{thmintro}[Seifert manifolds] \label{thmintro:Seifert}
	Let $f\colon M \to M$ be a dynamically coherent partially hyperbolic diffeomorphism on a closed Seifert-fibered $3$-manifold. If $f$ is homotopic to the identity, then some iterate is a discretized Anosov flow. 
\end{thmintro}

We eliminate the assumption of dynamical coherence in \cite{BFFP-sequel}; this resolves the Pujals' Conjecture for Seifered fibered manifolds\footnote{The conjecture is true for Seifert manifolds with fundamental group with polynomial growth \cite{HP-Nil} and false in Seifert-fibered manifolds when the isotopy class is not the identity as the examples in \cite{BGP,BGHP} are not homotopic to identity and so cannot be discretized Anosov flows.}. Note that Theorem \ref{thmintro:Seifert} does not use the classification of Anosov flows on Seifert-fibered 3-manifolds \cite{Ghys,Barbot96}.

\begin{thmintro}[Hyperbolic manifolds] \label{thmintro:Hyperbolic}
	Let $f\colon M \to M$ be a dynamically coherent partially hyperbolic diffeomorphism on a closed hyperbolic $3$-manifold. Then some iterate of $f$ is a discretized Anosov flow.
\end{thmintro}

This resolves a classification conjecture in \cite{CHHU} for hyperbolic 3-manifolds. We note here that hyperbolic 3-manifolds are well known to admit many partially hyperbolic diffeomorphisms as many Anosov flows have been constructed on them. The question of which hyperbolic 3-manifolds admit Anosov flows is still open, but our result implies that those which admit dynamically coherent partially hyperbolic diffeomorphisms must admit Anosov flows.

Note that this theorem does not assume that $f$ is homotopic to the identity, since any homeomorphism on a closed hyperbolic $3$-manifold has a finite power that is homotopic to the identity. It does, however, assume dynamical coherence.

Theorems~\ref{thmintro:Seifert} and \ref{thmintro:Hyperbolic} strengthen a more general statement which works in every 3-manifold and which requires some knowledge of taut foliations. See Appendix~\ref{ap.tautfol} for the relevant definitions.

Let $f\colon M \to M$ be a partially hyperbolic diffeomorphism on a closed $3$-manifold $M$. When $f$ is homotopic to the identity, we denote by $\ft$ a lift to the universal cover $\mt$ that is obtained by lifting such a homotopy. For a dynamically coherent partially hyperbolic diffeomorphism, we denote the center-stable and center-unstable foliations by $\cW^{\mathrm{cs}}$ and $\cW^{\mathrm{cu}}$, and their lifts to $\mt$ by $\wcs$ and $\wcu$. 

\begin{thmintro}\label{teo-main-coherent} 
	Let $f\colon M \to M$ be a dynamically coherent partially hyperbolic diffeomorphism on a closed $3$-manifold $M$ that is homotopic to 
	the identity. If $\cW^{\cs}$ and $\cW^{\cu}$ are $f$-minimal, then either 
	\begin{enumerate}[label=\rm{(\roman*)}]
		\item \label{item.teo-main-coherent-discretized} $f$ is a discretized Anosov flow, or 
		
		\item \label{item.teo-main-coherent-double_translation} $\cW^{\cs}$ and $\cW^{\cu}$ are $\R$-covered and uniform, and $\widetilde f$ acts as a translation on the leaf spaces of $\widetilde \cW^{\cs}$ and $\widetilde \cW^{\cu}$.   
	\end{enumerate}
\end{thmintro}

Here, \emph{$f$-minimal} means that the only closed sets that are both $f$-invariant and saturated by the foliation are the empty set and the whole manifold $M$. If $f$ is transitive or volume-preserving, then it is already known that $\cW^{\cs}$ and $\cW^{\cu}$ are $f$-minimal \cite{BW}. We will show that this holds as well when $M$ is hyperbolic or Seifert and the lift $\tf$ fixes a leaf in the universal cover (see Proposition \ref{p.hypSeifminimal}). We show that \ref{item.teo-main-coherent-double_translation} cannot occur in a hyperbolic manifold, and Theorem~\ref{thmintro:Hyperbolic} follows.  

It is likely that Theorem~\ref{thmintro:Hyperbolic} could be shown in the setting of 3-manifolds with atoroidal pieces in their JSJ decomposition, but we have not pursued this here as it would require proving results similar to \cite{Thurston2,Calegari00,Fen2002} in this setting.

The technique to eliminate the possibility \ref{item.teo-main-coherent-double_translation} in Theorem \ref{teo-main-coherent} is more widely applicable: In a companion article \cite{BFFP_companion}, we use the same ideas to show that a partially hyperbolic diffeomorphism on a Seifert manifold which acts as a pseudo-Anosov on (part of) the base is not dynamically coherent.

For Seifert manifolds, it is possible to show that, after taking an iterate, there is another lift $\tf$ that is still a bounded distance from the identity and which fixes a leaf of $\widetilde \cW^{\cs}$, and $f$-minimality follows. We show that \ref{item.teo-main-coherent-double_translation} implies leaf conjugacy of (possibly an iterate of) $f$ to a time-one map of an Anosov flow on a Seifert-fibered manifold, and Theorem~\ref{thmintro:Seifert} follows. We also completely classify the partially hyperbolic diffeomorphisms for which it is necessary to take an iterate, as opposed to $f$ itself, to get a discretized Anosov flow (see Remark \ref{rem.classification_weird_examples}).

\subsection{Remarks and references}
The definition of a partially hyperbolic diffeomorphism traces back to \cite{HPS} and \cite{BrinPesin}.

The classification problem for $3$-dimensional partially hyperbolic diffeomorphisms has attracted significant attention since the pioneering work of \cite{BW,BBI1,BI}, which was partially motivated by Pujals' conjecture (see also \cite[\S 20]{PughShub}).
See also the surveys \cite{CHHU,HP-survey,PotrieICM}.

Besides its intrinsic interest, the classification problem for partially hyperbolic diffeomorphisms has dynamical implications. For example, several finer dynamical and ergodic properties have been studied under the assumption of having a discretized Anosov flow (while not using that terminology), for instance in \cite{AVW,BFT} (see also \cite{PotrieICM, Wilkinson}). Our results here and in \cite{BFFP-sequel} make that condition checkable.
Several of the techniques presented here also yield information about the dynamics along the center direction, which is so far poorly understood (see, e.g.~\cite{FenleyPotrie}).

In addition, this article contains several new results of independent interest. Indeed, important steps in our study do not use partial hyperbolicity, but instead only use the more general setting of foliation preserving diffeomorphisms. Thus some results (see \S\ref{s.dicho} and \S\ref{s.coarse_dynamics_translations} for instance) are much more widely applicable. In particular, in \S\ref{s.coarse_dynamics_translations} we use regulating pseudo-Anosov flows to understand the dynamics of a diffeomorphism that translates the leaves of an $\R$-covered foliation, showing that any such diffeomorphism has ``invariant cores'' that shadow the closed orbits of the corresponding flow.

\subsection{Acknowledgments}
We would like to thank Christian Bonatti, Andrey Gogolev, and Andy Hammerlindl for many helpful comments and discussions.

%%%%%%%%%%%%%%%%%%%%%%%%%%%%%%%%%%%%%%%%%%%
\section{Outline and discussion} \label{s.plan_of_proof}

In this section we will set some basic terminology, outline our major arguments, and detail the organization of this paper.

\begin{definition}
	A $C^1$-diffeomorphism $f \colon M \to M$ on a $3$-manifold $M$ is \emph{partially hyperbolic} if there is a $Df$-invariant splitting of the tangent bundle $TM$ into three $1$-dimensional bundles
	\[ TM = E^{\mathrm{s}} \oplus E^{\mathrm{c}} \oplus E^{\mathrm{u}} \]
	such that for some $n>0$, one has
	\begin{align*}
	 \|Df^n|_{E^{\mathrm{s}}(x)}\| &< 1, \\
	 \|Df^n|_{E^{\mathrm{u}}(x)}\| &> 1, \text{ and}\\
	 \|Df^n|_{E^{\mathrm{s}}(x)}\| < \|Df^n|_{E^{\mathrm{c}}(x)}\| &< \|Df^n|_{E^{\mathrm{u}}(x)}\|,
	\end{align*}
	for all $ x\in M$.
\end{definition}

See Appendix~\ref{app.partial_hyperbolicity} for more details. Our major goal is to show that large classes of partially hyperbolic diffeomorphisms are discretized Anosov flows:

\begin{definition}
	A \emph{discretized Anosov flow} is a partially hyperbolic diffeomorphism $g\colon M \to M$ on a $3$-manifold $M$ that is of the form $g(p) = \Phi_{t(p)}(p)$ for a topological Anosov flow $\Phi$ and a map $t \colon M \rightarrow (0,\infty)$.
\end{definition}

The precise definition of a topological Anosov flow is given in Appendix~\ref{ss.DAF}, where we also explain the relationship between discretized Anosov flows and the more common notion of partially hyperbolic diffeomorphisms that are leaf-conjugate to time-$1$ maps of Anosov flows. 

Consider a discretized Anosov flow $g\colon M \to M$ on a closed $3$-manifold $M$. We will see (Proposition~\ref{prop.equivDALC}) that $g$ is dynamically coherent, and that the center leaves of $g$ are exactly the orbits of the underlying flow. This means that $g$ fixes each leaf of the center foliation. Moreover, it has a natural lift $\widetilde{g}\colon \mt \to \mt$ to the universal cover that fixes the lift of each center leaf, but fixes no point in $\mt$.  Indeed, such a lift may be obtained by flowing points along lifted orbits. That is, $\widetilde{g}(p) = \widetilde{\Phi}_{t(\pi(p))}(p)$, where $\widetilde{\Phi}$ is the lifted flow and $\pi\colon \mt \to M$ is the covering map.

In fact, to show that a partially hyperbolic diffeomorphism $f\colon M \to M$ is a discretized Anosov flow, it will suffice to find a lift $\ft\colon \mt \to \mt$ with this property, i.e., that fixes the leaves of the lifted center foliation, but fixes no point in $\mt$. This argument is essentially given in \cite[Section 3.5]{BW} --- see Section~\ref{ss.topoAnosov}.

\subsection{Setup}
We will now set some basic definitions and outline our major arguments. We will assume some familiarity with $3$-manifold topology, taut foliations, and leaf spaces; see Appendices~\ref{app.3_manifold_topology} and \ref{ap.tautfol} for an outline of the necessary background.

In this paper, $M$ will be a closed $3$-manifold, and $f\colon M \to M$ will be a dynamically coherent partially hyperbolic diffeomorphism that is homotopic to the identity.

The center-stable, center-unstable, stable, unstable, and center foliations on $M$ are denoted by $\cW^{\cs}$, $\cW^{\cu}$, $\cW^{\mathrm{s}}$, $\cW^{\mathrm{u}}$, and $\cW^{\mathrm{c}}$. These lift by the universal covering map $\pi\colon \mt \to M$ to foliations on $\mt$ which we denote by $\wt\cW^{\cs}$, $\wt\cW^{\cu}$, $\wt\cW^{\mathrm{s}}$, $\wt\cW^{\mathrm{u}}$, and $\wt\cW^{\mathrm{c}}$.

\vspace{.1cm}
\fbox{\begin{minipage}{.95\textwidth}
        \textbf{Convention:}  Throughout this paper we will assume that $\pi_1(M)$ is not virtually solvable.
       \end{minipage}}
\vspace{.1cm}
       
This assumption implies that there is no closed surface tangent to either $E^{\cs}$ or $E^{\cu}$ (Theorem~\ref{thm-HHUtori}), a fact that we will use often.

The classification of partially hyperbolic diffeomorphisms on manifolds with virtually solvable fundamental group is complete \cite{HP-Nil,HP-Sol}, and our assumption does not affect our main results (see Theorem~\ref{teo.solv}).  

\subsubsection{Good lifts}\label{sss.goodlift}
Since $f$ is homotopic to the identity, we can lift such a homotopy to $\mt$, and obtain a lift $\ft\colon \mt \to \mt$ that is \emph{good}:

\begin{definition}\label{d.goodlift}
	A lift $\tf \colon \mt \to \mt$ of a homeomorphism $f \colon M \to M$ is called a \emph{good lift} if
	\begin{enumerate}[label=(\roman*)]
		\item\label{item.goodlift_bounded} it moves each point a uniformly bounded distance (i.e., there exists $K>0$ such that $d_{\mt}(x, \ft(x)) < K$ for all $x\in \mt$), and
		\item\label{item.goodlift_commutes} it commutes with every deck transformation.
	\end{enumerate}
\end{definition}

Throughout this paper we will always take $\ft$ to be a good lift of $f$.

\begin{remark}
	In fact, it is easy to show that \ref{item.goodlift_bounded} follows from \ref{item.goodlift_commutes} on a closed manifold.  
	
	A homeomorphism may have more than one good lift. Indeed, composing a good lift with a deck transformation in the center of the fundamental group yields another good lift. Conversely, the existence of more than one good lift implies that the fundamental group has non-trivial center. By the Seifert-fibered space conjecture (see, e.g., \cite{Calegari:book}), this implies that the manifold is Seifert-fibered with orientable Seifert fibration.
\end{remark}

\subsection{Outline of the paper}

Recall that there are no closed surfaces tangent to the center-stable or center-unstable bundles. In particular, $\cW^{\cs}$ and $\cW^{\cu}$ have no closed leaves, which implies that they are taut. Furthermore, $\mt$ is homeomorphic to $\mathbb{R}^3$, and each leaf of $\wt\cW^{\cs}$ or $\wt\cW^{\cu}$ is a properly embedded plane that separates $\mt$ into two open balls (cf.~Theorem \ref{thm-tautfoliation}).

\subsubsection{Dichotomies for foliations}\label{sss.proof_program}
In \S\ref{s.dicho} we study the basic structure of the center-stable and center-unstable foliations, and the way that $\ft$ permutes their lifted leaves. Much of this section applies more generally to a homeomorphism that is homotopic to the identity and preserves a foliation.

The basic tool is Lemma~\ref{l-unboundedhalfspaces}, which says that the complementary components of a lifted center-stable or center-unstable leaf are ``large'' in the sense that they contains balls of arbitrary radius. Since $\ft$ moves points a uniformly bounded distance, this has immediate consequences for the way that it acts on the leaf spaces of $\wt\cW^{\cs}$ and $\wt\cW^{\cu}$.

In particular, we deduce that the set of center-stable leaves that are fixed by $\ft$ is closed in the leaf space $\cL^{\cs}$ of $\wt\cW^{\cs}$, each complementary component of this set is an open interval that is acted on by $\ft$ as a translation, and any two leaves in one of these ``translation regions'' are a finite Hausdorff distance apart (Proposition~\ref{p.dichotomy}). The same holds for the center-unstable foliation.

When $\cW^{\cs}$ is $f$-minimal, or $M$ is hyperbolic or Seifert-fibered, we use this to show that either:\\
 \begin{tabular}{cl}
 \setword{$(\star)$}{star_dichotomy}
&%
 \begin{minipage}{.85\textwidth}
 \vspace{.1cm}
  \begin{itemize}
		\item $\ft$ fixes every leaf of $\wt\cW^{\cs}$, or
		\item $\cW^{\cs}$ is $\R$-covered and uniform, and $\ft$ acts as a translation on the leaf space of $\wt\cW^{\cs}$.
	\end{itemize}
 \end{minipage}
 \end{tabular}
\vspace{.1cm}
 
Recall that \emph{$\R$-covered} means that the leaf space in the universal cover is homeomorphic to  $\R$, and \emph{uniform} means that any two leaves in the universal cover are a finite Hausdorff distance apart.

This dichotomy is easy to show under the assumption of $f$-minimality, where it does not use partial hyperbolicity (Corollary~\ref{c.minimalcase}). It takes significantly more work assuming instead that $M$ is hyperbolic or Seifert-fibered (Proposition~\ref{p.hypSeifminimal}).\footnote{This dichotomy holds even without the assumption of dynamical coherence, but the proof is substantially more difficult (see \cite{BFFP-sequel}).}

If $\cW^{\cs}$ and $\cW^{\cu}$ are $f$-minimal, or $M$ is hyperbolic or Seifert-fibered, we are left with three possibilities:
\begin{enumerate}[label=(\arabic*)]
\item\label{it.dinv} {\bf double invariance:} $\ft$ fixes every leaf of both $\widetilde \cW^{\cs}$ and $\widetilde \cW^{\cu}$;
\item\label{it.nomix} {\bf mixed behavior:} $\ft$ fixes every leaf of either $\widetilde \cW^{\cs}$ or $\wt\cW^{\cu}$, and acts as a translation on the leaf space of the other, which is $\R$-covered and uniform; or
\item\label{it.dtran} {\bf double translation:} $\ft$ acts as a translation on both $\widetilde \cW^{\cs}$ and $\widetilde \cW^{\cu}$, which are $\R$-covered and uniform. 
\end{enumerate}

The remainder of the argument is arranged around these three possibilities. We will see in \S\ref{s.nomix} that mixed behavior cannot happen. In \S\ref{s.doublyinvariant} we show that double invariance implies that $f$ is a discretized Anosov flow. The double translation case is ruled out for Seifert-fibered manifolds in \S\ref{s.proof_of_thmA_DCcase}, and for hyperbolic manifolds in \S\ref{s.coarse_dynamics_translations}--\S\ref{sec-thmB}.

\subsubsection{Center dynamics in fixed leaves} \label{sss.intro_fixed_center}
In \S\ref{s.perfectfits}, we work under the assumption that $\ft$ fixes every leaf of $\wt\cW^{\cs}$, and study the dynamics \emph{within} each center-stable leaf. In particular, we show (Proposition~\ref{p.fixedcenter}):\\
\begin{tabular}{cl}
	\setword{$(\star\star)$}{star_fixedcenters}
	&%
	\begin{minipage}{.85\textwidth}
		\vspace{.1cm}
		If $\ft$ fixes every leaf of $\wt\cW^{\cs}$, then any leaf of $\wt\cW^{\cs}$ that is fixed by a non-trivial deck transformation contains a center leaf that is fixed by $\ft$.
	\end{minipage}
\end{tabular}
\vspace{.1cm}

This immediately eliminates the possibility of mixed behavior (see \S\ref{s.nomix}). It will also be used in \S\ref{s.doublyinvariant} to show that double invariance implies that $f$ is a discretized Anosov flow. 

Consider a center-stable leaf $L$ that is fixed by a deck transformation $\gamma$. The proof of \ref{star_fixedcenters} comes down to understanding the topology of the stable foliation within $L$ ``in the direction of'' $\gamma$. The formal meaning of this is the \emph{axis} for the action of $\ft$ on the stable leaf space in $L$ (see Appendix~\ref{ss.axes}), but it can be understood intuitively as the set of all stable leaves that cross the core of the cylinder $\rquotient{\mt}{\langle\gamma\rangle}$ essentially. 

Suppose that there is an line's worth of stable leaves in this direction, which corresponds to circle's worth in $\rquotient{\mt}{\langle\gamma\rangle}$ as depicted (roughly) in the left half of Figure~\ref{f-fixedcenters}. Then one can find a curve representing $\gamma$ that is transverse to the stable foliation, and a ``graph transform argument'' finds a corresponding center leaf preserved by both $\gamma$ and $\ft$ (Lemma~\ref{l.Axs-line}).

\begin{figure}[h]
	\centering
	\begin{subfigure}[t]{0.45\textwidth}
		\centering
		\includegraphics{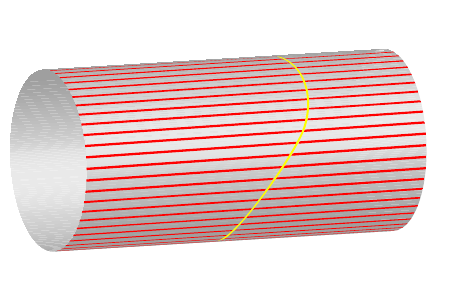} 
	\end{subfigure}%
	\hspace{.05\textwidth}
	\begin{subfigure}[t]{0.45\textwidth}
		\centering
		\includegraphics{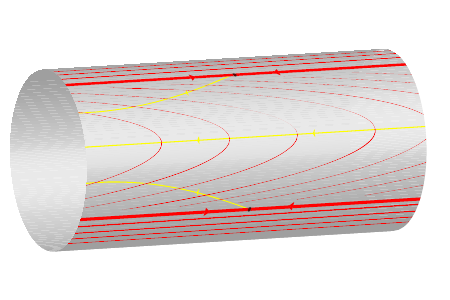} 
	\end{subfigure}
	\caption{Axes}\label{f-fixedcenters}
\end{figure}

The other possibility is that one finds gaps, which look roughly like Reeb components as in the right half of Figure~\ref{f-fixedcenters}. We eliminate the possibility of such gaps by combining the dynamics coming from partial hyperbolicity with two conflicting forces:
\begin{enumerate}
	\item[(i)] On one hand, the topology of the stable and center foliations within $L$ forces the existence of a center ray within this gap that is expanded by $f$ (Lemma~\ref{l.expandedcenterDC}).
	
	\item[(ii)] On the other hand, we find from the geometry of $L$ that the entire gap, and any center leaf within it, must be coarsely contracted (Lemma~\ref{l.gapcontractionDC}).
\end{enumerate}
These conclusions are contradictory, so there can be no gaps, completing the proof of \ref{star_fixedcenters}.

The existence of the expanded center ray (i) is delicate, and depends crucially on dynamical coherence (see Remark~\ref{r.problemndc} and Figure~\ref{f.pfDC}). The coarse contraction of gaps (ii) is more robust, and will be used again in the dynamically incoherent case in \cite{BFFP-sequel}.

\subsubsection{Double invariance}\label{sss.double_invariance}
In the doubly invariant case~\ref{it.dinv}, one would like to show that $\ft$ fixes each center leaf. Since by assumption it fixes each center-stable and center-unstable leaf, it fixes the intersection between any two such leaves. Each component of this intersection is a center leaf, but there is no a priori reason for it to have a single component. In \S\ref{ss.fixingcenter}, we show that $\ft$ fixes either every center leaf or no center leaf, and so by \ref{star_fixedcenters} it fixes every center leaf.

Once we know that $\ft$ fixes every center leaf, we can use the arguments of Bonatti--Wilkinson \cite{BW} to show that the center foliation is the orbit foliation of a topological Anosov flow, and hence that $f$ is a discretized Anosov flow. This is done in \S\ref{ss.topoAnosov}, completing the proof of Theorem \ref{teo-main-coherent}.

\subsubsection{Double translation in Seifert-fibered manifolds}
The double translation case \ref{it.dtran} turns out to be the trickiest. The preceding results work with either topological conditions ($M$ being Seifert-fibered or hyperbolic) or dynamical conditions (minimality or $f$-minimality). To handle double translations we will need topological restrictions.

Part of the difficulty is that double translation \emph{do} in fact exist (Remark \ref{rem.classification_weird_examples})! However, these examples live in Seifert-fibered manifolds, and have iterates that are discretized Anosov flows. Using a similar idea (see section 6.2 of \cite{BFFP_Announcement}), one can build homeomorphisms that act as a double translation on any manifold which admits an $\R$-covered Anosov flow which is not a suspension, but our techniques do show that they cannot be dynamically coherent partially hyperbolic (even in a topological sense) when the ambient manifold is hyperbolic.

Eliminating double translations when $M$ is a Seifert manifold relies on a trick: 
Since there are many good lifts, we show in \S\ref{s.proof_of_thmA_DCcase} that some good lifts (of a power of $f$) must fix the leaf of at least one foliation. This completes the proof of Theorem~\ref{thmintro:Seifert}.

\subsubsection{Double translation in hyperbolic manifolds}
We are left to treat the case of double translations in hyperbolic manifolds, which we do in \S\ref{s.coarse_dynamics_translations} and \S\ref{sec-thmB}.

In \S\ref{s.coarse_dynamics_translations}, we prove a result about $\R$-covered foliations that is of general interest. In a hyperbolic $3$-manifold, an $\R$-covered foliation admits a transverse regulating pseudo-Anosov flow (see Appendix \ref{app.regulatingpA}). We will use this flow to understand the dynamics of any homeomorphism that acts as a translation on its leaf space (Proposition \ref{p.homeotranslation}):\\

\begin{tabular}{cl}
	\setword{$(\star\star\star)$}{star_cores}
	&%
	\begin{minipage}{.85\textwidth}
		\vspace{.1cm}
		Let $f\colon M \to M$ be a homeomorphism on a closed hyperbolic $3$-manifold that is homotopic to the identity and preserves a taut, $\R$-covered foliation $\cT$. Suppose that a good lift acts as a translation on $\wt\cT$.
		
		Then for each periodic orbit $\gamma$ of the regulating pseudo-Anosov flow $\Phi$, there is a corresponding invariant ``core'' $T_{\gamma}$ for $f$. Moreover, the dynamics of $f$ at $T_{\gamma}$ is coarsely identical to the dynamics of $\Phi$ at $\gamma$ (in the sense that they have the same Lefschetz index).
	\end{minipage}
\end{tabular}
\vspace{.1cm}

There is a little lie in this description, as the core $T_{\gamma}$ is in fact in the cover $\rquotient{\mt}{\langle\gamma\rangle}$ and is invariant under the appropriate lift of $f$ to that cover. Also, having a hyperbolic manifold is not essential  --- we use similar techniques on Seifert-fibered manifolds in \cite{BFFP_companion}.

The result \ref{star_cores} is the main ingredient in \S\ref{sec-thmB}, where we show that double translations cannot occur in hyperbolic manifolds and complete the proof of Theorem \ref{thmintro:Hyperbolic}. 
The rough idea is that \ref{star_cores} gives a circle invariant by (a lift of) $f$ and with at least one fixed point, but partial hyperbolicity implies that any fixed point must be, say, repulsive. But the devil is in the details, and while one can make this rough idea precise in the case of a foliation, it does not lead to a contradiction for branching foliation. This is why Theorem \ref{thmintro:Hyperbolic} requires dynamical coherence. In the non-dynamically coherent case, treated in \cite{BFFP-sequel}, while we do not prove or disprove its existence, we obtain a detailed description of the permitted structure in the translation case that in particular produces some topological obstructions for the existence of partially hyperbolic diffeomorphisms in hyperbolic 3-manifolds.

%%%%%%%%%%%%%%%%%%%%%%%%%%%%%%%%%%%%%%%%%%%%%%
%
%
%         END INTRO AND DISCUSSIONS    
%
%
%%%%%%%%%%%%%%%%%%%%%%%%%%%%
\section{Foliations and good lifts} \label{s.dicho}

In this section we will study the way that a good lift $\ft$ of a dynamically coherent partially hyperbolic diffeomorphism $f\colon M \to M$ that is homotopic to the identity permutes the leaves of the lifted center-stable and center-unstable foliations.

Most of the arguments in this section apply to any homeomorphism of a $3$-manifold that preserves an appropriate foliation and is homotopic to the identity, so we will work for a while in this more general setting. At the end, we obtain the following results for our setting.

\begin{proposition}\label{prop.endgoal_section_foliation}
	Let $f\colon M \to M$ be a dynamically coherent partially hyperbolic diffeomorphism on a closed $3$-manifold that is homotopic to the identity, and let $\ft$ be a good lift of $f$. If $\cW^{\cs}$ is $f$-minimal, or $M$ is hyperbolic or Seifert-fibered, then either
	\begin{enumerate}
		\item $\cW^{\cs}$ is $\mathbb{R}$-covered and uniform, and $\tf$ acts on the leaf space of the lifted foliations as a translation, or
		\item $\tf$ fixes each leaf of the lifted foliation.
	\end{enumerate}

	The same holds for the center-unstable foliation $\cW^{\cu}$.
\end{proposition}

\subsection{General homeomorphisms}\label{ss.general_homeomorphisms}
Let $\cT$ be a taut foliation on a closed $3$-manifold $M$. Recall our standing assumption that $M$ does not have virtually solvable fundamental group. This implies that the universal cover $\mt$ is homeomorphic to $\mathbb{R}^3$, and that each leaf of the lifted foliation $\wt\cT$ is a properly embedded plane (see Theorem \ref{thm-tautfoliation}). 

Fix a homeomorphism $f\colon M \to M$ that preserves $\cT$ and is homotopic to the identity, and a good lift $\ft\colon \mt \to \mt$ (Definition~\ref{d.goodlift}).

\subsubsection{Complementary regions} \label{sss.complementary_regions}
Being a properly embedded plane, each leaf $K \in \wt\cT$ separates $\mt$ into two open balls. We will call these two components of $\mt \setminus K$ the \emph{complementary regions} of $K$. The closure of such a complementary region $U$ is called a \emph{side} of $K$ and is simply $\overline{U} = U \cup K$.

If $K, L \in \wt\cT$ are distinct leaves, then $K \cup L$ separates $\mt$ into three open complementary regions, which can be built from the complementary regions of $K$ and $L$: Let $U$, $U'$ be the complementary regions of $K$, labeled so that $L \subset U'$, and let $W$, $W'$ be the complementary regions of $L$, labeled so that $K \subset W'$. Then the complementary regions of $K \cup L$ are $U$, $V = U' \cap W'$, and $W$. See Figure~\ref{f-regionbetween}. We call $V$ the \emph{(open) region between $K$ and $L$}. Its closure, which is simply $\overline{V} = K \cup V \cup L$, is called the \emph{closed region between $K$ and $L$}.

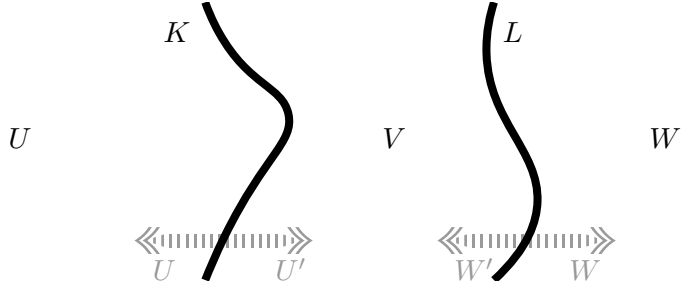
\begin{figure}[h!]
	\begin{center}
		\includegraphics{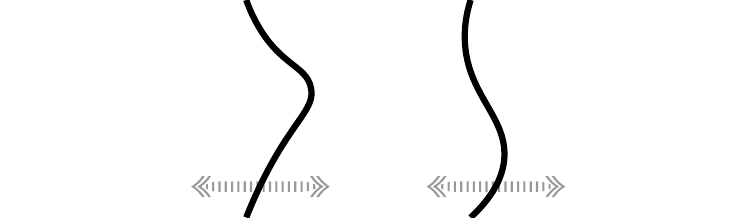}
		\begin{picture}(0,0)
		\put(-262, 90){$K$}
		\put(-135, 90){$L$}
		
		\put(-320,50){$U$}
		\put(-180,50){$V$}
		\put(-80,50){$W$}

		\put(-266,0){\textcolor[gray]{0.6}{$U$}}
		\put(-220,0){\textcolor[gray]{0.6}{$U'$}}
		
		\put(-153,0){\textcolor[gray]{0.6}{$W'$}}
		\put(-110,0){\textcolor[gray]{0.6}{$W$}}
		
		\end{picture}
	\end{center}
	\caption{The region between two leaves}\label{f-regionbetween}
\end{figure}

Since $\mt$ is simply connected, the lifted foliation $\wt\cT$ is coorientable. A coorientation determines a labeling of the complementary regions of each leaf $L \in \lt$ as a \emph{positive} complementary region denoted $L^\oplus$ and a \emph{negative} complementary region denoted $L^\ominus$. The corresponding positive and negative sides are denoted by $L^+ = L^\oplus \cup L$ and $L^- = L^\ominus \cup L$.

\begin{remark}
	We stress that a priori, some deck transformations or lifts of $\cT$-preserving homeomorphisms may exchange the coorientations of $\lt$.
\end{remark}

\subsubsection{The big half-space lemma}\label{ss.geometry_lifted_leaves}
The following lemmas will be used to understand the way that $\ft$ can act on the leaf space of $\lt$.

\begin{lemma}\label{l-unboundedhalfspaces}
For every leaf $L \in \lt$, and every $R>0$, there is a ball of radius $R$ contained in each of the complementary regions of $L$.
\end{lemma}

\begin{proof}
	If necessary, we pass to a double cover of $M$ for which $\cT$ is coorientable, and choose such a coorientation. Then every deck transformation preserves the corresponding coorientation on $\wt\cT$ and orientation on the leaf space $\cL = \cL_{\lt}$.
	
	Fix a ball $B \subset \mt$ of arbitrary radius, and a leaf $L \in \lt$. We will find a deck transformation $g$ that takes $B$ into $L^{\oplus}$; a similar argument would find a deck transformation that takes $B$ into $L^{\ominus}$, completing the proof.
	
	Since $B$ is compact, we can find a leaf $F \in \mt$ such that $B \subset F^{\oplus}$. Indeed, the quotient map $\nu\colon \mt \to \cL$ takes $B$ to a compact subset $\nu(B)$ of the leaf space, which can be covered by a finite collection of open intervals $I_1, I_2, \cdots, I_n$. We may assume that $\nu(B)$ intersects every one of these intervals. Since the leaf space is simply connected, we claim that
at least one of these intervals has an initial point (with respect to the orientation on $\cL$) that is not contained in any other interval.
To prove this we use induction.
Start with any interval, denote it by $I_{i_1}$ and let $p_1$ be
its initial point. If $p_1$ is not contained in 
the interior of any of the $I_i's$ then we proved the property we want.
Otherwise $p_1$ is contained in the interior of another interval,
which we denote by $I_{i_2}$. Let $p_2$ be the initial point
of $I_{i_2}$. Then, since $\cL$ is simply connected, $p_2$ is not in $I_{i_1}$ (and, obviously, nor is it in $I_{i_2}$). Notice that $I_{i_2}$ is distinct from $I_{i_1}$ as $p_1$ is in $I_{i_2}$ but
not in $I_{i_1}$. In this way, we inductively
obtain distinct intervals $I_{i_n}$ with 
initial points $p_n$ of $I_{i_n}$ not contained
in $\cup_{1 \leq j \leq n} I_{i_j}$. Since the collection
is finite the process terminates at some $n$, and we obtain
$I_{i_n}$ so that its initial point is not contained in the interior of any $I_i$.

The lowest endpoint of $I_{i_n}$ 
is therefore disjoint from $\nu(B)$. Then $B$ is contained in the positive complementary region of the leaf $F$ corresponding to this initial point.
	
	Let us now find a deck transformation $g$ that takes $F^{\oplus}$, and hence $B \subset F^{\oplus}$, into $L^{\oplus}$. Since $\cT$ is taut, we can find a positively oriented closed transversal $\gamma\colon  [0, 1] \to M$, based at a point in $\pi(F)$, that passes through $\pi(L)$. Let $\wt\gamma$ be a the lift of $\gamma$ based at a point in $F$, which passes through some lift $L'$ of $\pi(L)$. Then we can take $g = h' \circ h$, where $h$ is the deck transformation that takes $\wt{\gamma}(0)$ to $\wt{\gamma}(1)$, and $h'$ takes $L'$ to $L$. The oriented transversal $\wt\gamma$ certifies that $h(F^{\oplus}) \subset L'^{\oplus}$, and $h'(L'^{\oplus}) \subset L^{\oplus}$ because our deck transformations preserve coorientation.
\end{proof}

It follows that $\ft$ can never take a complementary region of a leaf off of itself: This would mean that it takes an arbitrarily large ball off of itself, which contradicts the fact that $\ft$ moves points a uniformly bounded distance. This has important consequences for the way that $\ft$ behaves with respect to each leaf.

In particular, if $\ft$ fixes a leaf, then it cannot interchange its complementary components, and we have:

\begin{corollary}
	If $L \in \lt$ is fixed by $\ft$, then $\ft$ preserves coorientations at $L$.
\end{corollary}

\subsubsection{Translated leaves}

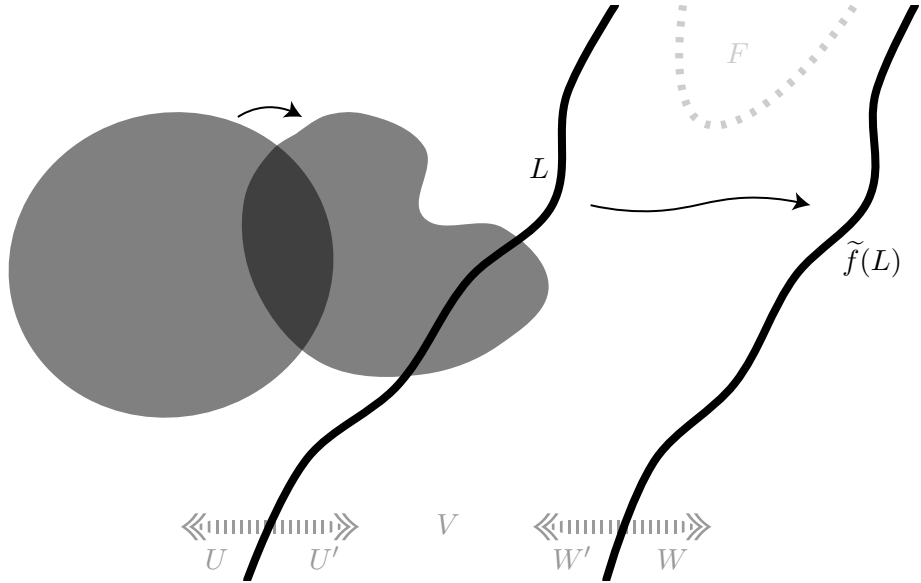
\begin{figure}[h!]
	\begin{center}
	\begin{overpic}{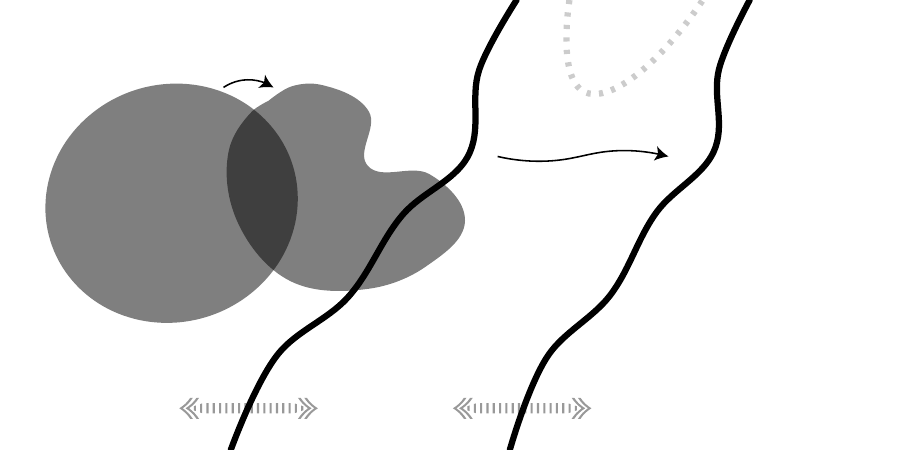}
		\put(67, 45){$\textcolor{black!20}{F}$}
		
		\put(50, 35){$L$}
		\put(77, 27){$\ft(L)$}
		
		\put(42,4){$\textcolor{black!40}{V}$}
		
		\put(22,1){$\textcolor{black!40}{U}$}
		\put(31,1){$\textcolor{black!40}{U'}$}
		
		\put(52,1){$\textcolor{black!40}{W'}$}
		\put(61,1){$\textcolor{black!40}{W}$}
	\end{overpic}
	\end{center}
	\caption{Translation-like behavior}\label{f-translationlike}
\end{figure}

If $\ft$ moves some leaf, then it does so in a ``translation-like'' manner as is illustrated in Figure~\ref{f-translationlike}. In fact, something a bit stronger is true:

\begin{proposition} \label{p.translatedleaves}
	Let $L \in \wt\cT$ be a leaf that is not fixed by $\ft$, then
	\begin{enumerate}
		\item the closed region between $L$ and $\ft(L)$ is foliated as a product,
		\item $\ft$ takes each coorientation at $L$ to the corresponding coorientation at $\ft(L)$, and
		\item the closed region between $L$ and $\ft(L)$ is contained in the closed $2R$-neighborhood of $L$, where $R = \max_{y \in \mt} d(y,\ft(y))$.
	\end{enumerate}
\end{proposition}
\begin{proof}
	As in Figure~\ref{f-translationlike}, let $U, U'$ be the complementary components of $L$, labeled so that $\ft(L) \subset U'$, and let $W, W'$ be the complementary components of $\ft(L)$, labeled so that $L \subset W'$. Then $V = U' \cap W'$ is the open region between $L$ and $f(L)$.
	
	 Note that $\ft$ must take $U$ to either $W$ or $W'$. But $W$ is disjoint from $U$, so we cannot have $\ft(U) = W$ by Lemma~\ref{l-unboundedhalfspaces}. Thus $\ft$ takes $U$ to $W'$, and $U'$ to $W$. This is what is meant formally by the aforementioned ``translation-like'' behavior. Denote $K=\ft(L)$.

	\begin{enumerate}
		\item It follows, in particular, that $\ft$ takes $V$ off of itself and into $W$. To see that $\overline{V} = K \cup V \cup L$ is foliated as a product, it suffices to show that every leaf that lies in $V$ separates $K$ from $L$. Suppose that some leaf $F \subset V$ does not separate $K$ from $L$. Then $K$ and $L$ are contained in the same complementary region of $F$, so the other complementary region is contained entirely in the open region $V$ between $K$ and $L$. But this means that $V$ contains balls of arbitrary radius, which contradicts the fact that $\ft$ takes $V$ off of itself. Thus every leaf that lies in $V$ separates $K$ from $L$, and $\overline{V}$ is foliated as a product.
		
		\item Since $\overline{V} = K \cup V \cup L$ is foliated as a product, it follows that a coorientation taking $L^{\oplus} = U'$ will take $f(L)^{\oplus} = W$. We have already seen that $\ft(U') = W$, so (2) follows.
		
		\item Suppose for a contradiction that there is a point $p \in \overline{V}$ with $d(p, L) = 2R + \varepsilon$ for some $\varepsilon > 0$. Then since $d(L, \ft(L)) \leq R$, it follows from the triangle inequality that $d(p, \ft(L)) \geq R + \varepsilon$. This means that the open ball $B_{R + \varepsilon}(p)$ at $p$ of radius $R + \varepsilon$ is contained in $V$. But we have already seen that $\ft$ takes $V$ off of itself, so this implies that $d(p, \ft(p)) \geq R + \varepsilon > R$, a contradiction. \qedhere
	\end{enumerate} 
\end{proof}

It follows that if $L \in \wt\cT$ is not fixed by $\ft$, then one can string together the $\ft$-translates of the closed region $\overline{V}$ between $L$ and $\ft(L)$ to see that their union
\[
U = \cdots \cup \ft^{-2}(\overline{V}) \cup \ft^{-1}(\overline{V}) \cup \overline{V} \cup \ft^{1}(\overline{V}) \cup \ft^{2}(\overline{V}) \cup \cdots
\]
is an open, product-foliated set that is preserved by $\ft$. This corresponds to an open interval in the leaf space on which $\ft$ acts as a translation.

Let $X \subset \mt$ be the union of all leaves of $\wt\cT$ that are fixed by $\ft$. Then $U$ is contained in a connected component of $\mt \setminus X$. In fact, the following lemma says that $U$ is exactly a connected component of $\mt \setminus X$.

\begin{lemma}\label{l.borderoftranslationdomain}
	Let $L$ be a leaf of $\wt\cT$ that is not fixed by $\ft$, and let $U = \bigcup_{i = -\infty}^{\infty} f^i(\overline{V})$, where $\overline{V}$ is the closed region between $L$ and $\ft(L)$. Then each leaf in $\partial U = \overline{U} \setminus U$ is fixed by $\ft$.
\end{lemma}
\begin{proof}
	The frontier $\partial U$ can be broken into ``forwards'' and ``backwards'' frontiers
	\[
		\partial_\omega U = \limsup_{i \to \infty} f^i(L) \text{ and } \partial_\alpha U= \limsup_{i \to -\infty} f^i(L),
	\]
	each of which is preserved by $f$.
	
	Let $K$ be a leaf in $\partial_\omega U$, and suppose that $\ft(K) \neq K$. Then the closed region between $K$ and $\ft(K)$ would be product foliated, and it follows that either $K$ separates $U$ from $\ft(K)$ or $\ft(K)$ separates $U$ from $K$. This contradicts the fact that $K, \ft(K) \subset \partial_\omega U$, so we must have $\ft(K) = K$. A similar argument shows that every leaf in $\partial_\alpha U$ is fixed by $\ft$.
\end{proof}

\subsubsection{The dichotomy} \label{ss.general_dichotomy}
We summarize the preceding discussion in terms of the leaf space:

\begin{proposition}\label{p.dichotomy} 
	Let $M$ be a closed $3$-manifold, $f\colon M \to M$ a homeomorphism homotopic to the identity that preserves a taut foliation $\cT$, and $\ft$ a good lift.
	
	The set $\Lambda \subset \cL_{\wt\cT}$ of leaves that are fixed by $\ft$ is closed and $\pi_1(M)$-invariant. Moreover, each connected component $I$ of $\cL_{\wt\cT} \setminus \Lambda$ is an open interval that $\ft$ preserves and acts on as a translation, and every pair of leaves in $I$ are a finite Hausdorff distance apart.
\end{proposition}
\begin{proof} 
	The only detail that needs to be pointed out is that $\Lambda$ is $\pi_1(M)$-invariant, which follows from the fact that $\ft$ commutes with every deck transformation.
\end{proof}

In particular, one may have $\Lambda = \emptyset$:

\begin{corollary}
	Let $M$ be a closed $3$-manifold, $f\colon M \to M$ a homeomorphism homotopic to the identity that preserves a taut foliation $\cT$, and $\ft$ a good lift.
	
	If $\ft$ fixes no leaf of $\wt\cT$, then $\cT$ is $\R$-covered and uniform, and $\ft$ acts on $\cL_{\wt\cT} \simeq \R$ as a translation.
\end{corollary}

This leads to a simple dichotomy when the foliation is $f$-minimal. Recall: 
\begin{definition}\label{def.f_minimal} A foliation $\cT$ that is preserved by a map $f\colon M \to M$ is said to be \emph{$f$-minimal} if the only closed sets that are both $f$-invariant and saturated are $M$ and $\emptyset$.
\end{definition}

\begin{corollary} \label{c.minimalcase}
	Let $M$ be a closed $3$-manifold, $f\colon M \to M$ a homeomorphism homotopic to the identity that preserves a taut foliation $\cT$, and $\ft$ a good lift.
	
	If $\cT$ is $f$-minimal, then either
	\begin{enumerate}
		\item $\ft$ fixes every leaf of $\lt$, or
		\item $\cT$ is $\R$-covered and uniform, and $\ft$ acts as a translation on the leaf space of $\lt$.
	\end{enumerate}
\end{corollary}
\begin{proof}
	Since $\ft$ commutes with each deck transformation, each deck transformation preserves the set $\Lambda \subset \cL$ of fixed leaves. In particular, if $I$ is a component of $\cL \setminus \Lambda$ and $g \in \pi_1(M)$ then one has either $g(I) = I$ or $g(I) \cap I = \emptyset$. If $\Lambda = \emptyset$ then we are in case (2) by the preceding corollary.
	
	Suppose instead that $\Lambda \neq \emptyset$. If $\Lambda \neq \cL$, then it corresponds to a closed, $\cT$-saturated subset of $M$ that is preserved by $f$. Furthermore, this subset is not all of $M$ since it cannot accumulate on a leaf lying in the interior of a complementary interval to $\Lambda$. This contradicts $f$-minimality, so we have $\Lambda = \cL$ and are in case (1).
\end{proof}

\subsubsection{Bounded movement inside leaves}
We end this section by showing that a good lift that fixes every leaf will be within a bounded distance of the identity not only in $\mt$ but also in each leaf.

\begin{lemma}\label{lema-boundedinfixcsleaf}
	Let $M$ be a closed $3$-manifold,
$f\colon M \to M$ a homeomorphism homotopic to the identity that preserves a taut foliation $\cT$, and $\ft$ a good lift.
	
	If $\ft$ fixes every leaf of $\lt$, then there is a uniform bound $K>0$ such that for any leaf $L \in \lt$ one has 
	\[
	d_L(x, \ft(x) ) <K \text{ for all } x \in L,
	\]
	where $d_L$ is the path metric on $L$.
\end{lemma}
\begin{proof}
	Suppose for a contradiction that there is a sequence of points $x_i \in \mt$ for which $d_{L_i} (x_i, \ft(x_i) )$ tends to infinity, where $d_{L_i}$ is the path metric on the leaf $L_i \in \wt\cT$ containing $x_i$.
	
	Since $M$ is compact, we can pass to a subsequence and find a sequence of deck transformations $\gamma_i$ such that $\gamma_i(x_i)$ converges to a point $x_{\infty}$. Since $\ft$ commutes with $\gamma_i$, we have that $\gamma_i ( \ft(x_i) ) = \ft( \gamma_i(x_i))$ converges to $\ft(x_{\infty} )$.
	
	Now, since $d_{L_i} ( \gamma_i (\ft(x_i) ), \gamma_i(x_i)) = d_{L_i} ( \ft (x_i) , x_i )$ goes to infinity, the points $\ft (x_{\infty} )$ and $x_{\infty}$ must be in different leaves of $\lt$. This contradicts the fact that $\ft$ fixes each leaf of $\lt$.
\end{proof}

\begin{remark}\label{r.lema-boundedinfixcsleaf_sublamination}
	This lemma applies as well to a leaf of a closed sublamination of $\cT$ whose lift is leafwise fixed by $\ft$. 
	\end{remark}

\subsection{Consequences for partially hyperbolic systems} \label{ss.direct_consequences}
Let us now specialize, and fix a closed $3$-manifold $M$ whose fundamental group is not virtually solvable, a dynamically coherent partially hyperbolic diffeomorphism $f \colon M \to M$ that is homotopic to the identity, and a good lift $\ft\colon \mt \to \mt$. 

We denote by $\cW^{\cs}$, $\cW^{\cu}$, $\cW^{\mathrm{s}}$, $\cW^{\mathrm{u}}$, and $\cW^{\mathrm{c}}$ the center-stable, center-unstable, stable, unstable, and center foliations.

\subsubsection{Fixed points and the topology of leaves}

\begin{lemma}\label{l.nofixedpoints} 
	Let $L \in \widetilde \cW^{\cs}$ be a leaf that is fixed by $\ft$. If there is a sequence of leaves $L_i \in \widetilde \cW^{\cs}$ that are fixed by $\ft$ and accumulate on $L$, then there are no points in $L$ fixed by non-trivial power of $\ft$. 
\end{lemma}
\begin{proof}
Suppose that $\ft^n$, $n>0$, fixes some point $x \in L$. Then it fixes the unstable leaf $\widetilde \cW^u (x)$ through that point. When $i$ is sufficiently large, $\widetilde \cW^u(x)$ intersects $L_i$ at a single point $x_i$, which must therefore also be fixed by $\ft^n$. This contradicts the fact that $\ft^n$ expands unstable leaves. 
\end{proof}

\begin{proposition}\label{p.planeannuliorfixedpoints} 
	Let $L \in \widetilde \cW^{\cs}$ be a leaf that is fixed by $\ft$. If $\ft$ fixes no point in $L$, then $A = \pi(L)$ has cyclic fundamental group (and is therefore a plane, cylinder, or M\"{o}bius band).
\end{proposition}
\begin{proof}
	Let $\cL$ be the leaf space of the stable foliation within $L$.	
	Since $\ft$ fixes no point in $L$, it cannot fix any stable leaf in $L$, since a stable leaf that is fixed by $\ft$ would contain a fixed point. In other words, $\ft$ acts freely on $\cL$, and hence has an axis $A_f$ by Proposition~\ref{proposition-axes}.
	
	Consider two elements $\gamma_1, \gamma_2 \in \pi_1(M)$ that fix $L$. Since the stable foliation can have no circular leaves, neither of these elements may fix a stable leaf. Hence each $\gamma_i$ acts freely on $\cL$ with an axis $A_i$. 
	
	As $\ft$ commutes with both $\gamma_1$ and $\gamma_2$ it follows from Proposition~\ref{proposition-axes} that in fact these axes are the same, i.e., $A_1 = A_f = A_2$. Proposition~\ref{proposition-axes} further implies that the subgroup generated by $\gamma_1$ and $\gamma_2$ is abelian. Since there are no compact leaves in $\cW^{\cs}$, it follows that this subgroup is cyclic, and hence $\gamma_1^n = \gamma_2^m$ for some $n,m$.
\end{proof}

\subsubsection{Minimality in hyperbolic or Seifert manifolds}\label{ss.minimalsandh} 

The following proposition implies that the dichotomy in Corollary~\ref{c.minimalcase} holds, without the assumption of $f$-minimality, when $M$ is hyperbolic or Seifert-fibered.

\begin{proposition}\label{p.hypSeifminimal}
Let $M$ be a closed $3$-manifold that is hyperbolic or Seifert-fibered, $f\colon M \to M$ a dynamically coherent partially hyperbolic diffeomorphism that is homotopic to the identity, and $\ft$ a good lift.

If $\ft$ fixes one leaf of $\widetilde \cW^{\cs}$, then $\cW^{\cs}$ is a minimal foliation, and $\ft$ fixes every leaf of $\widetilde \cW^{\cs}$. 
The same statement holds for $\cW^{\cu}$. 
\end{proposition}
\begin{proof}
	Without loss of generality, we may assume, by passing to a finite cover of $M$ and power of $f$, that $\cW^{\cs}$ is orientable and coorientable, and $f$ preserves all orientations and coorientations, and
in addition that $M$ is orientable.

	Let $X \subset \mt$ be the union of all leaves of $\wt\cW^{\cs}$ that are fixed by $\ft$. This set is non-empty by hypothesis, $\pi_1(M)$-invariant as $\ft$ is a good lift and closed by Proposition~\ref{p.dichotomy}. It follows that $\pi(X) \subset M$ is compact and non-empty. By Zorn's lemma, we can find a minimal compact, non-empty, $\cW^{\cs}$-saturated subset $\Lambda \subset \pi(X)$. We will show that $\Lambda = M$, which implies both that $\cW^{\cs}$ is minimal and that $\ft$ fixes every leaf.
	
	Note that $\Lambda$ cannot contain any isolated leaves. Indeed, it cannot consist solely of isolated leaves since then these leaves would be compact, and $\cW^{\cs}$ has no compact leaves. Deleting an isolated leaf from $\Lambda$ still leaves a closed, saturated subset, so the existence of an isolated leaf would contradict our minimality assumption.
	
	Let $\wt\Lambda$ be the preimage of $\Lambda$ in $\mt$. Since no leaf in $\Lambda$ is isolated, every leaf in $\wt\Lambda$ is accumulated on by a sequence of leaves in $\wt\Lambda$. Since these leaves are all fixed by $\ft$, Lemma~\ref{l.nofixedpoints} implies that $\ft$ has no fixed points in $\wt\Lambda$. It follows from Proposition~\ref{p.planeannuliorfixedpoints} that each leaf of $\Lambda$ is either a cylinder or a plane. 

	Assume for a contradiction that $\Lambda \neq M$, and hence $\wt\Lambda \neq \mt$. Then we can choose a non-trivial connected component $V$ of $\mt \setminus \wt\Lambda$.

	\begin{claim} \label{claim_finitely_many_leaves}
		The projection $\pi(\partial V)$ to $M$ consists of finitely many leaves.
	\end{claim}
	This is a standard fact in the theory of foliations \cite[Lemma 5.2.5]{CandelConlonI}.

	For each $x \in \partial V$, let $u_x$ be the maximal connected unstable segment that starts at $x$ and is contained in $\overline{V}$, which is either a closed interval or a ray. That is, $u_x$ is the component of $\wt\cW^{\mathrm{u}}(x) \cap \overline{V}$ that contains $x$. Given $r > 0$, and a leaf $L \subset \partial V$ define
	\[ 
	A^r_L = \{ x \in L \mid \ell(u_x) \geq r \}.
	\]
	
	\begin{claim}
		For any leaf $L \subset \partial V$, and any $r > 0$, the set $\pi(A^r_L)$ is compact as a subset of $\pi(L)$.
	\end{claim}
	\begin{proof}
		It is straightforward to see that $\pi(A^r_L)$ is closed as a subset of $\pi(L)$. If it is not compact, then one can find sequence of points $x_i \subset \pi(A^r_L)$ that escapes every compact subset of $\pi(L)$. After taking a subsequence we can assume that the $x_i$ converges in $M$ to some point $x$. Take a chart around $x$ of the form $D^2 \times (0, 1)$ where each $D^2 \times \{y\}$ is a plaque of $\cW^{\cs}$, and each $\{p\} \times (0, 1)$ is an oriented plaque of $\cW^u$. Since the $x_i$ escape every compact subset of $\pi(L)$, we can pass to a subsequence such that each $x_i$ is contained in a different plaque. Then it is easy to see that the lengths of the unstable segments at $x_i$ that stay in $\pi(\overline{V})$ must go to $0$, a contradiction.
	\end{proof}
	
	\begin{claim} \label{claim_piLannulus}
		Each leaf $L \subset \partial V$ corresponds to an annulus $\pi(L)$ in $M$. 
	\end{claim}
	\begin{proof}
		Fix a leaf $L \subset \partial V$ and an $r > 0$ for which $A_L := A^r_L$ is non-empty.	As $\pi(L)$ is either a plane or an annulus, we assume for a contradiction that it is a plane. Then the covering map $\pi$ restricts to a homeomorphism on $L$, so the fact that $\pi(A_L)$ is compact means that $A_L$ is compact. Let $D$ be a disk in $L$ containing $A_L$ in its interior. 
		
		Since the leaves of $\partial V$ are fixed by $\ft$, and a positive iterate of $\ft$ expands the lengths of unstable arcs, we can find an $n \geq 1$ for which $\ft^n(D) \subset A_L \subset D$. Then Brouwer's fixed point theorem implies that $\ft^n$ has a fixed point in $L$, which contradicts Lemma \ref{l.nofixedpoints}. So $L$ must be an annulus.
	\end{proof}
	
	Now we can complete the proof of Proposition \ref{p.hypSeifminimal}. Let $L_1, \ldots, L_k$ be a finite collection of leaves in $\partial V$ that cover $\pi(\partial V)$, and fix $r > 0$ such that each $A_i := A^r_{L_i}$ is non-empty. Choose a compact annulus $C_i$ in each $\pi(L_i)$ that contains $\pi(A_i)$.	Since $f$ preserves orientations and coorientations, we can join each $C_i$ to an adjoining $C_j$ with an annulus built out of unstable segments $u_x$ for points $x \in \partial C_i$. Iterating this procedure, we obtain a torus $T$ that consists of alternating annuli 
	contained in leaves of $\cW^{\cs}$ and annuli transverse	to $\cW^{\cs}$ inside $W = \pi(V)$.  
Notice that $T$ is a torus and not a Klein bottle, since $T$
is two sided and $M$ is orientable.
	
	We will now (for the first time) use the assumption that $M$ is hyperbolic or Seifert-fibered to see that $T$ bounds a solid torus.
	If $M$ is hyperbolic, then $T$ either bounds a solid torus or is contained in a $3$-ball (Lemma~\ref{l.torus_inball_or_bounds_solid_torus}).	If $T$ is contained in a $3$-ball, then the annuli $C_i$ are contained in that ball, so the $\cW^{\cs}$ leaf containing $C_i$ is compressible. This contradicts the fact that $\cW^{\cs}$ is a taut foliation (see Theorem \ref{thm-tautfoliation}), so $T$ bounds a closed solid torus $U$.
Notice that $V$ is not $\pi_1(M)$ invariant, but it is precisely
invariant, that is, if $\gamma(V), V$ intersect for $\gamma \in \pi_1(M)$
then they are equal. The set $\pi(V)$ is just the projection
of $V$ to $M$, and it is a complementary region of $\Lambda$.
	
	If $M$ is Seifert-fibered, then $\cW^{\cs}$ is a horizontal foliation. That is, one can isotope $\cW^{\cs}$ so that all leaves are transverse to the Seifert fibers of $M$ (Theorem \ref{thm-horizontalseif}). It follows that the complementary regions of the lamination $\Lambda$ are horizontal. In particular, the region $\pi(V \cup \partial V)$ is a product, which means that the torus $T$ is made up of two horizontal $C_i$ and two transverse annuli, and hence bounds a closed solid torus $U$.
	
	We will now use a ``volume vs.~length'' argument to get a contradiction. 
Roughly the argument is that volume grows linearly with iteration,
but unstable length grows exponentially, leading to a contradiction.
We refer to \cite[Proposition 5.2]{HaPS} for a detailed proof and give only a sketch:
	Consider an unstable arc inside $U$ from a point in some $\pi(A_i)$ to some $C_j$. Fix some $\eps>0$, and call $u$ the non-empty part of that unstable segment that is at distance $\geq\eps$ from both $C_i$ and $C_j$. Up to taking $\eps>0$ smaller if necessary, we can then assume that $u$ is at distance at least $\eps>0$ from $T$. Consider a lift $\tilde u \subset V$ of $u$, and note that for any positive $n$, $\tf^n (\tilde u)$ stays a bounded distance $a_0 > 0$ away from the corresponding lift $\widetilde T$ of $T$. This is the reason for taking $u$ which is $\eps$ away from $T$,
the value of $\eps$ is not important.
Notice also that $\tilde f(U) \subset U$.
The length of $\tf^n (\tilde u)$ will grow exponentially in $n$, while the volume of its maximal tubular neighborhood of size $a_0$
can only grow linearly, as $\ft$ is at bounded distance from the identity and the fundamental group of $T$ is $\mathbb{Z}$. This means that for
$n$ arbitrarily big, $\tf^n (\tilde u)$ returns very near to itself in $\mt$, contradicting the fact that it is transverse to $\wt\cW^{\cs}$. Thus $\Lambda = M$ as desired.
\end{proof}

\begin{remark}
We point out here that the hypothesis of $M$ being hyperbolic or Seifert-fibered is used in a single place, but it is crucial. To see this, it is enough to consider the time-one map of Franks-Williams intransitive
Anosov flow \cite{Fr-Wi} (or any other non-transitive Anosov flow),
for which neither the center-stable nor the center-unstable foliations are minimal. 
\end{remark}

\subsection{Gromov-hyperbolicity of leaves}

In this section we show that Candel's Theorem (Theorem \ref{thm-leafuniformisation}) applies under the assumptions that $f$ is partially hyperbolic and that $\ft$ fixes the leaves of the center-stable foliation. It is known that the assumption for Candel's Theorem is always satisfied for hyperbolic 3-manifolds (see e.g., \cite{Calegari:book}), as well as for horizontal foliations in Seifert-fibered manifolds with exponential growth of fundamental group (which is automatic in our case thanks to Theorem \ref{thm-horizontalseif}). However, in order to deal with other 3-manifolds, we need a more general version.

\begin{lemma}\label{Gromovhyp}
Let $M$ be a closed $3$-manifold, $f\colon M \to M$ a dynamically coherent partially hyperbolic diffeomorphism homotopic to the identity, and $\ft$ a good lift.
 
Suppose $\cW^{\cs}$ has no compact leaves, and $\ft$ fixes every leaf of $\widetilde{\cW}^{\cs}$. Then every leaf of ${\mathcal W}^{\cs}$ is Gromov-hyperbolic. Moreover, there is a metric on $M$ which restricts to a metric of constant negative curvature on each leaf of $\cW^{\cs}$. 
\end{lemma}

\begin{proof}
Thanks to Candel's Theorem (Theorem \ref{thm-leafuniformisation}), all we have to show is that $\cW^{\cs}$ does not admit a holonomy invariant transverse measure. So we suppose that there is an invariant transverse measure $\mu$ to $\cW^{\cs}$. Let $S$ be its support. First notice that, as there are no compact leaves in ${\mathcal W}^{\cs}$, $\mu$ has no atoms, so there are no isolated leaves in $S$. 

Let $\widetilde \mu$ be the lift of $\mu$ to $\mt$. The fact that $\ft$ fixes every leaf of 
$\widetilde \cW^{\cs}$ implies that the measure $\mu$ is $f$-invariant. To see this, consider $\tau$ a small transversal to ${\mathcal W}^{\cs}$, and $\wt\tau$ its lift to $\mt$. Then, since $\ft$ fixes every leaf of $\widetilde \cW^{\cs}$, the transversals $\wt\tau$ and $\ft(\wt\tau)$ intersect the same set of leaves of $\cW^{\cs}$. Hence, $\wt\tau$ and $\ft(\wt\tau)$ have the same $\wt\mu$-measure (because $\mu$ is an invariant transverse measure), thus $\mu(\tau)=\mu(f(\tau))$ as desired.

Now let $\tau$ be a closed segment on an unstable leaf and call $x$ one of its endpoints.
Note that $\tau$ is a transversal to ${\mathcal W}^{\cs}$, and, up to taking a different unstable segment, we assume that $\tau$ is chosen so that $\mu(\tau)>0$.
We can also choose a sequence $(n_i)$ of negative integers converging to $-\infty$ such that $(f^{n_i}(x))$ converges to some $y\in M$.

Then, as the $n_i$ are negative integers, $f^{n_i}$ contracts the unstable length, so the sequence of segments  $(f^{n_i}(\tau))$ also converges to $y$. Now, since $\mu$ is $f$-invariant, it implies that $\mu\left(f^{n_i}(\tau)\right) = \mu(\tau) >0$, for all $n_i$. By taking the limit, we get that the $\cW^{\cs}$ leaf containing $y$ must be an atom of $\mu$, in contradiction with the fact recalled earlier that $\mu$ has no atoms.
Thus $\cW^{\cs}$ does not admit an invariant transverse measure and Candel's Theorem yields the conclusion of our lemma.
\end{proof}

We will use the metric given by this lemma on $M$ 
in the specific situations where a hyperbolic metric
makes the proof less technical. But all such results
only need a Gromov-hyperbolic metric in the center-stable
or center-unstable leaves.

\subsection{Summary}

For convenience, we summarize the results obtained in Section \ref{s.dicho}.

\begin{corollary}\label{coro-sumarizedichotomy} Let $f \colon M \to M$ be a partially hyperbolic, dynamically coherent, diffeomorphism of a $3$-manifold $M$ that is homotopic to the identity. Suppose that $\cW^{\cs}$ is 
$f$-minimal, or that $M$ is hyperbolic or Seifert-fibered. Let $\ft\colon \mt \to \mt$ be any good lift of $f$.

Then, $\ft$ has no fixed points and either
\begin{enumerate} 
\item the foliation $\cW^{\cs}$ is 
$\R$-covered and uniform, and $\ft$ acts as a translation on the
leaf space of $\widetilde \cW^{\cs}$; or,
\item the map $\ft$ leaves every leaf of 
$\widetilde \cW^{\cs}$ fixed and every leaf of 
$\cW^{\cs}$ is a plane, an annulus or a M\"{o}bius band. Moreover, there is a metric on $M$ that restricted to each leaf has constant negative curvature
$-1$.
\end{enumerate}
\end{corollary}

\begin{proof}
By Proposition \ref{p.dichotomy}, either the foliation $\cW^{\cs}$ is 
$\R$-covered and uniform, and $\ft$ acts as a translation on the
leaf space of $\widetilde \cW^{\cs}$, i.e.~we are in case $1.$, or, if $\ft$ does not act as a translation, then it must fix at least one leaf. 

Then, by assumption, or by Proposition \ref{p.hypSeifminimal} if $M$ is hyperbolic or Seifert-fibered, we have that $\cW^{\cs}$ is $f$-minimal.
Hence, if $\ft$ does not act as a translation, then we can apply Corollary \ref{c.minimalcase} 
to get that $\ft$ must fix every leaf of $\widetilde \cW^{\cs}$.

Now, if no leaf of $\widetilde \cW^{\cs}$ is fixed by 
$\ft$ then $\ft$ cannot have fixed points. On the other hand, if all leaves 
of $\widetilde \cW^{\cs}$ are fixed, then we can apply Lemma \ref{l.nofixedpoints} to deduce that $\ft$ still does not fix points. 
Finally, Proposition \ref{p.planeannuliorfixedpoints} implies that when 
all leaves of $\widetilde \cW^{\cs}$ are fixed then every leaf 
is a plane, an annulus, or a M\"{o}bius band.
The existence of the claimed metric follows from Lemma \ref{Gromovhyp}.
\end{proof} 

The same statement holds for the foliation $\cW^{\cu}$. Notice that, in principle, the behavior of each foliation is independent. The goal of the next few sections is to show that the behavior of one of the foliations forces the same behavior in the other foliation.

%%%%%%%%%%%%%%%%%%%%%%%%%

%%%%%%%%%%%%%%%%%%%%%%%%%

\section{Center dynamics in fixed leaves}\label{s.perfectfits}

In this section we will study the dynamics \emph{within} center-stable leaves. The main result is Proposition~\ref{p.fixedcenter}, which will be used to understand the doubly invariant and mixed cases (see \S\ref{sss.proof_program}).

\subsection{Perfect fits}\label{ss.perfectfit}
Much of this section will be concerned with transverse pairs of foliations of a plane --- in particular, the stable and center foliations within a center-stable leaf. We begin by introducing some basic tools, in particular the idea of ``perfect fits'' first used by Barbot and the second author \cite{Fenley94,BarbotFeuill}.

Let $L$ be a complete plane equipped with a transverse pair of one-dimensional foliations $\cS$ and $\cC$. We denote by $\cL^s:= \rquotient{L}{\cS}$ and $\cL^c:= \rquotient{L}{\cC}$ their respective leaf spaces. These are simply-connected, separable $1$-manifolds which may not be Hausdorff (see e.g., \cite{Bar98,Calegari:book}). 

\begin{definition} \label{def.perfectfit}
	A leaf $c \in \cC$ and leaf $s \in \cS$ are said to make a $\cC\cS$-\emph{perfect fit}, if they do not intersect, but there is a local transversal $\tau$ to $\cC$ through $c$, such that every leaf $c' \in \cC$ that intersects $\tau$ on one side of $c$ must intersect $s$. 
	
	On the other hand, if there exists $\tau'$ a local transversal to $s \in \cS$, such that every leaf $s' \in \cS$ that intersect $\tau'$ on one side of $s$ has to intersect $c$, we will say that $s$ and $c$ make a $\cS\cC$-\emph{perfect fit}.
	
	If $c$ and $s$ make both a $\cC\cS$-perfect fit and a $\cS\cC$-perfect fit, we say that they make a \emph{perfect fit}.
\end{definition}

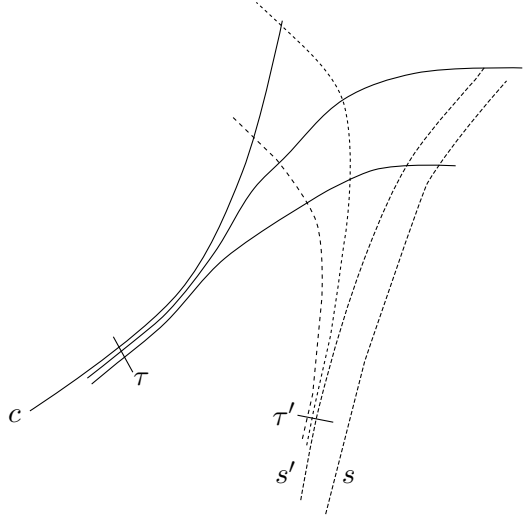
\begin{figure}[ht]
	\begin{center}
		\includegraphics[scale=0.5]{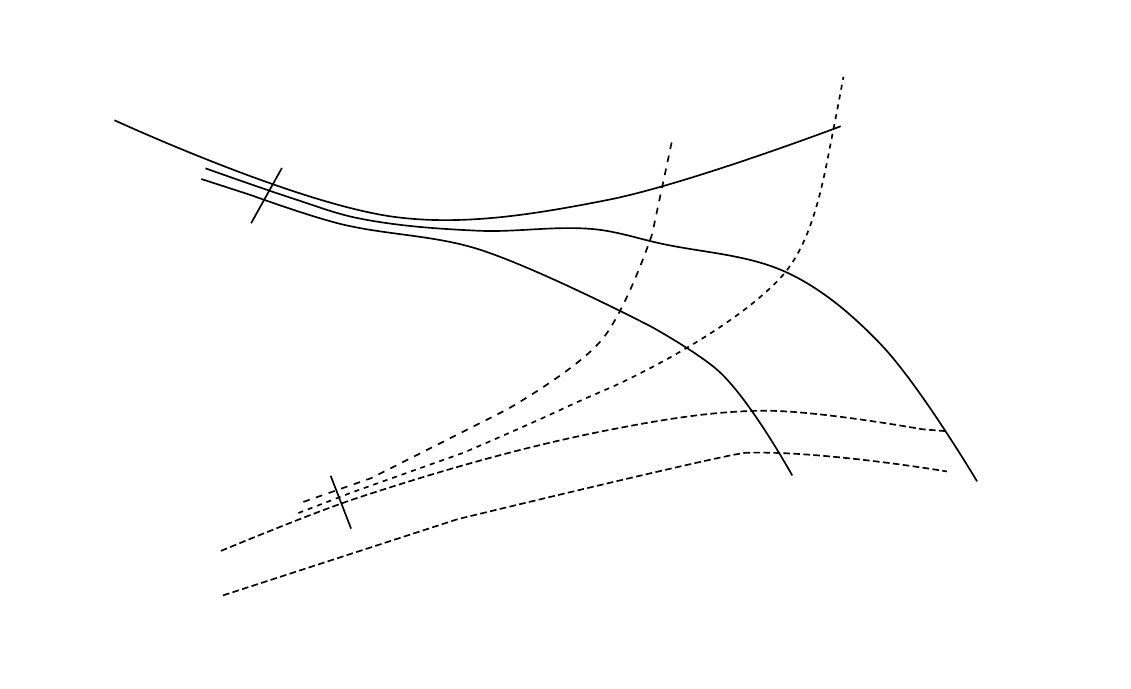}
		\begin{picture}(0,0)
		\put(-230,22){$s$}
		\put(-222,38){$s'$}
		\put(-224,110){$\tau$}
		\put(-205,47){$\tau'$}
		\put(-250,140){$c$}
		\end{picture}
	\end{center}
	\vspace{-0.5cm}
	\caption{{\small The leaves $c$ and $s$ make a $\cC\cS$-perfect fit,
			but not a $\cS\cC$-perfect fit. The leaves $c$ and $s'$ make a perfect fit.}}\label{f.perfectfit}
\end{figure}

\begin{lemma}
	\label{l-gettingperfectfit}
	If two leaves $c \in \cC$ and $s \in \cS$ make a $\cC\cS$-perfect fit, then there exists $s' \in \cS$, possibly distinct from $s$ such that $c$ and $s'$ make a perfect fit. The symmetric statement holds for $\cS\cC$-perfect fits. 
\end{lemma}
\begin{proof}
	Fix a small transversal $\tau$ to $c$. Let $c'$ near enough $c$
	which also intersects $s$. Let $p = c' \cap \tau$ and $q = c' \cap s$.
	For any $x$ in $c'$ between $p$ and $q$ and near enough $p$, the stable
	leaf of $x$ intersects $c$. Let $y$ in $c'$ between $p$ and $q$ be
	the first point such that the stable leaf of $y$ does not 
	intersect $c$. Let $s'$ be this stable leaf. Then $c, s'$ form a perfect
	fit.
\end{proof}

A straightforward argument shows the following --- see, e.g., \cite[Claim in Theorem~3.5]{Fen99}. Note that the argument in \cite{Fen99}
generally apply to any plane with two transverse foliations by lines. 
\begin{lemma}\label{l-nonsepperfectfit}
	If two leaves $s, s' \in \cS$ are non-separated in the leaf space $\cL^s$, then there is a unique leaf $c \in \cC$ that separates $s$ from $s'$ and makes a perfect fit with $s$.
\end{lemma}

\subsection{Finding fixed center leaves} \label{ss.finding_fixed_center_leaves}
The following proposition is the main result of this section.

\begin{proposition}\label{p.fixedcenter}
Let $f \colon M \to M$ be a dynamically coherent partially hyperbolic diffeomorphism homotopic to the identity, and $\ft$ a good lift.

Suppose that $\ft$ fixes every leaf of $\widetilde \cW^{\cs}$. Then any leaf of $\wt\cW^{\cs}$ that is fixed by a non-trivial element of $\pi_1(M)$ contains a center leaf that is fixed by $\ft$. 
\end{proposition}

The proof of this will span the rest of this section. Let us fix $M$, $f$, and $\ft$ as above, along with a leaf $L \in \widetilde \cW^{\cs}$ and a non-trivial element $\gamma \in \pi_1(M)$ that fixes $L$. Our goal is to find a center leaf $c \subset L$ that is fixed by $\ft$.

Let $\cL^s$ and $\cL^c$ be the leaf spaces of the foliations $\widetilde \cW^{\mathrm{s}}$ and $\widetilde \cW^c$ \emph{restricted to} $L$. These are simply connected, separable, $1$-manifolds that may not be Hausdorff.
Since $\ft$ fixes every center-stable leaf, Lemma~\ref{l.nofixedpoints} implies that $\ft$ has no fixed points. This means that $\ft$ cannot fix any stable leaf, since such a leaf would be contracted and hence contain a fixed point, so $\ft$ acts freely on $\cL^s$.

Since there are no circular stable leaves downstairs, $\gamma$ must also act freely on $\cL^s$. By Proposition~\ref{p.planeannuliorfixedpoints}, the stabilizer of $L$ is cyclic, so we may take $\gamma$ to be a generator.
Finally, since $\ft$ and $\gamma$ commute and act freely on $\cL^s$, they preserve an axis $A^s \subset \cL^s$, which is either a line or a $\ZZ$-union of intervals (see Proposition \ref{proposition-axes} and Remark \ref{r.axis}).

The following lemma completes the proof of Proposition~\ref{p.fixedcenter} when $A^s$ is a line. We remark here that a center leaf is an intersection of a center-stable and center-unstable leaf; but there could be curves tangent to the center which are not made this way if the center bundle is not uniquely integrable (which could be the case even if $f$ is dynamically coherent). We will call the latter curves `center curves' to distinguish them from 'center leaves'.

\begin{lemma}\label{l.Axs-line} 
	If $A^s \simeq \R$, then there exists a center leaf $c \subset L$ that is fixed by both $\ft$ and $\gamma$.
\end{lemma}
\begin{proof}
We will use the graph transform argument (Lemma \ref{l.grapht}). 
Since $A^s$ is homeomorphic to $\R$, one can find a bi-infinite curve $\eta$ in $L$ that is transverse to the stable foliation and invariant under $\gamma$. For instance, pick a point $x$ in $L$ and an arc $a$ from $x$ to $\gamma x$ transverse to the stable foliation. Concatenating the positive and negative iterates of $a$ by $\gamma$ gives such a curve $\eta$ (that can be smoothed if required).

In particular, $\eta$ represents the axis $A^s$ of $\gamma$, in the sense that a stable leaf is in $A^s$ if and only if it intersects $\eta$. 
Since $A^s$ is also the axis for the action of $\ft$ on $\cL^s$, every $\ft$-iterate of $\eta$ also represents $A^s$. In particular, $\ft(\eta)$ and $\eta$ intersect the exact same set of stable leaves.
So the curve $\eta$ satisfies the two hypothesis of Lemma \ref{l.grapht}, and we obtain a curve $\beta$ in $L$ that is tangent to $E^{\mathrm{c}}$ and invariant under both $\ft$ and $\gamma$. In particular $\beta$ is a center curve.
It remains to show that $\beta$ is in fact a center leaf.

Choose a point $x \in \beta$, and let $\beta'$ be the compact subsegment of $\beta$ running from $x$ to $\gamma x$. This is a fundamental domain for the action of $\gamma$ on $\beta$.
At each point $y \in \beta$, one can find a compact center segment $c_y$ through $y$ that intersects the same set of stable leaves as some compact subsegment $\beta_y$ of $\beta$, where the interior of $\beta_y$ contains $y$. A center segment is a segment contained in a center leaf.
By compactness of $\beta'$, one can find a finite collection $c_1, c_2, \cdots, c_k$ of these center segments such that the corresponding subsegments of $\beta$ covers $\beta'$.

Projecting to $M$, we have a finite union of center segments $\bigcup \pi(c_i)$ that intersects the same set of stable leaves as the closed curve $\pi(\beta)$. Since $f$ contracts stable segments, $f^n(\bigcup \pi(c_i))$ converges to $\pi(\beta)$ as $n \to \infty$. Since a sequence of center segments can only converge to a center segment, it follows that $\pi(\beta)$ is a center leaf, and so is $\beta$.
\end{proof}

\subsubsection{Gaps}
To complete the proof of Proposition~\ref{p.fixedcenter} we will show that $A^s$ is indeed a line.

Suppose that $A^s$ is not a line. Then it is a $\ZZ$-union of closed intervals
\[ A^s = \bigcup_{i\in \ZZ} [s^-_i, s^+_i]. \]
We will call each of the pairs $s^+_i, s^-_{i+1}$ a \emph{gap} in this axis.
The following lemma says that some positive power of $f$ fixes the image of each gap in $M$.

\begin{lemma}
	There are $m \neq 0$ and $n > 0$ such that $h = \gamma^m \circ \tf^n$ fixes every $s^\pm_i$.
\end{lemma}
\begin{proof}
	Since $\ft$ and $\gamma$ act freely on $A^s$, they act freely on the index set of the collection of intervals, which is $\ZZ$. It follows that some non-trivial element of the group generated by $\ft$ and $\gamma$ acts trivially on $\ZZ$. This element is of the form $h := \gamma^m \circ \ft^n$. Since both $\gamma$, and $\ft$ act freely on $A^s$, neither $n$ nor $m$ can be equal to zero, and we can take $n > 0$. Since $h$ fixes each interval, it fixes the endpoints of each interval as desired. 
\end{proof}

For the remainder of this section, we fix $h$ as in this lemma, and look at a single gap, setting $s^+ = s^+_i$ and $s^- = s^-_{i+1}$ for some fixed $i$. 

We highlight some features of this gap --- see Figure~\ref{f.pfDC}: 
First, Proposition~\ref{proposition-axes} says that $s^+$ is non-separated from $s^-$ in the leaf space $\cL^s$, so Lemma~\ref{l-nonsepperfectfit} provides a center leaf $c$ that makes a perfect fit with $s^+$ and separates $s^+$ from $s^-$. Since there is a unique such leaf, it follows that $h$ fixes $c$. 

Note that $h$ eventually contracts stable leaves; this is because $\gamma$ is an isometry, $\ft$ eventually contracts stable leaves, and $n > 0$. Up to an iterate we can assume that this contraction is immediate. It follows that $h$ fixes a single point $x$ within $s^+$. Let $c'$ be the center ray that starts at $x$ on the side of $c$. 

\begin{figure}[ht]
	\begin{center}
		\begin{overpic}[scale=0.70]{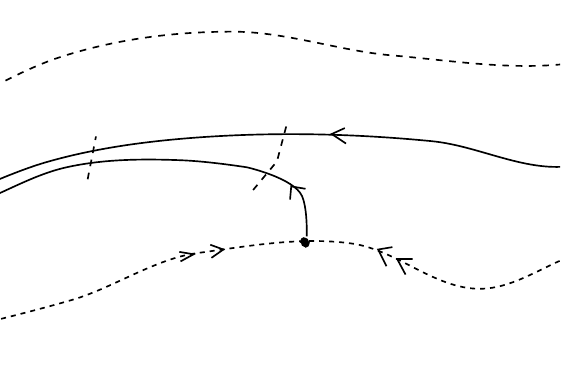}
		\put(52,18){$x$}
		\put(52,36){$s_y$}
		\put(36,30){$c'$}
		\put(6,5){$s^+$}
		\put(6,50){$s^-$}
		\put(40,43){$c$}
		\end{overpic}
		\hspace{.1cm}
        \begin{overpic}[scale=0.70]{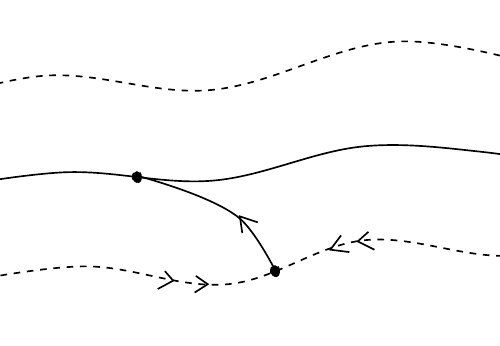}
		\put(46,36){$c$}
		\put(42,20){$c'$}
		\put(53,8){$x$}
		\end{overpic}
	\end{center}
	\vspace{-0.5cm}
	\caption{{\small A perfect fit forces expansion on a center ray (left), while, in the non-dynamically coherent case, $c'$ may land by merging with $c$ (right).}\label{f.pfDC}}
\end{figure} 

We will show that $h$ expands $c'$ in Lemma~\ref{l.expandedcenterDC}, and that it contracts $c'$ in Corollary~\ref{c.contractedcenterDC}. This is contradictory, so there are no gaps in $A^s$, i.e., it is a line, and the proof of Proposition~\ref{p.fixedcenter} is complete.

\subsubsection{Perfect fits and expanded center rays}\label{ss.perfect_fits_and_behavior_center_leaves}
In the following lemma we find that the topology of the stable and center foliations in $L$ forces a stable ray in our gap to expand.

\begin{lemma}\label{l.expandedcenterDC}
	$h$ acts as an expansion on $c'$ with unique fixed point $x$.
\end{lemma}
\begin{proof}
		
	Note that the stable leaf $s_y = \wt{\cW}^s(y)$ through any point $y \in c'$ that is sufficiently close to $x$ will intersect $c$. This is because $s^+$ and $c$ make a perfect fit, and $c'$ is a transversal to $\cS$ on the side of $c$. Given such a point, let $s'_y$ be the compact segment of $s_y$ that runs from $c'$ to $c$. Since the lengths of $h$-iterates of this segment go to zero, i.e., $\lim_{n \to \infty} \ell(h^n(s')) = 0$,  it follows that the $h$-iterates of $y$ eventually escape every compact set. Indeed, otherwise one would find that $c$ and $c'$ intersect at some accumulation point of $h^n(y)$. See Figure~\ref{f.pfDC} (left).

	The lemma follows since we can take $y \in c'$ arbitrarily close to $x$.
\end{proof}

\begin{remark}\label{r.problemndc}
	The proof of Lemma~\ref{l.expandedcenterDC} uses the structure of the transverse pair of foliations in an essential way. It does not hold when the center leaves are allowed to merge --- see Figure~\ref{f.pfDC} (right). This is exactly the type of behavior that arises in the examples of \cite{HHU-noncoherent}. 
\end{remark}

\subsubsection{Coarse contraction in stable gaps} \label{sss.using_perfect_fit_to_contradiction}  
In the following lemma we find that the geometry of the gap forces it to contract laterally. This contradicts the expansion found in Lemma~\ref{l.expandedcenterDC} --- see Corollary~\ref{c.contractedcenterDC}.

\begin{lemma}\label{l.gapcontractionDC}
	There is a rectangle $R$ bounded by segments of $s^+$ and $s^-$ that contain the fixed points, together with two arcs $\tau_1, \tau_2$, such that $h(R)$ is contained in the interior of $R$. See Figure~\ref{f.7M_firstuseoffigure}.
\end{lemma}

\begin{figure}[ht]
	\begin{center}
		\includegraphics[scale=0.68]{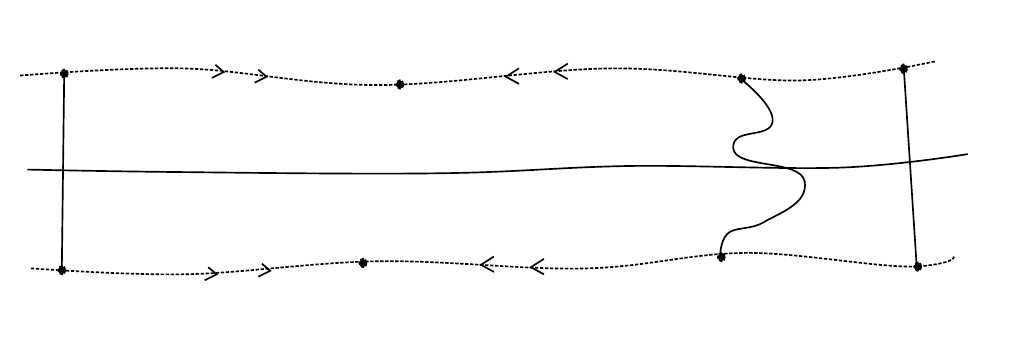} 
		\begin{picture}(0,0)
		\put(-225,18){$x$}
		\put(-325,18){$y_2$}
		\put(-45,18){$y_1$}
		\put(-121,72){$h(\tau_1)$}
		\put(-20,22){$s^+$}
		\put(-54,69){$\tau_1$}
		\put(-317,69){$\tau_2$}
		\put(-27,94){$s^-$}
		\put(-26,67){$c$}
		\put(-189,66){$R$}
		\end{picture}
	\end{center}
	\vspace{-0.5cm}
	\caption{The domain $R$ is mapped onto itself by $h$.}\label{f.7M_firstuseoffigure}
\end{figure}

We will need two lemmas. The first is that the gap is ``uniformly thin'':

\begin{lemma} \label{claim.boundeddistanceapart}
	The leaves $s^+$ and $s^-$ are a bounded Hausdorff distance apart with respect to the path metric on $L$.
\end{lemma}
\begin{proof}
	Since this gap is part of the axis $A^s = \bigcup_{k\in \ZZ} [s_k^-, s_k^+]$, it follows that $s^-$ separates $s^+$ from either $\ft(s^+)$ or $\ft^{-1}(s^+)$. Then Lemma~\ref{lema-boundedinfixcsleaf} implies that the Hausdorff distance between $s^+$ and $\ft^{\pm 1} (s^+)$ is uniformly bounded above, and the same holds for the Hausdorff distance between $s^+$ and $s^-$. 
\end{proof}

Recall that, since $\ft$ fixes all leaves of $\widetilde \cW^{\cs}$, Lemma \ref{Gromovhyp} (thanks to Candel's Theorem) implies that there is a metric $g$ on $M$ such that $\cW^{\cs}$ is leafwise hyperbolic. Let $d$ be the associated path metric on the leaf $L$. First, notice that since we are using a metric in $M$ provided by
Candel's Theorem \ref{thm-leafuniformisation}, it follows that
the path metric on leaves of $\cW^{\cs}$ is hyperbolic, hence
$L$ is isometric to the hyperbolic plane with this path metric.
In addition if $\gamma$ is a deck transformation leaving
$L$ invariant, then $\gamma$ is an isometry of this path
metric. This isometry has to be hyperbolic, as $M$ is a closed
manifold: There cannot be fixed points of $\gamma$ in $L$,
and the translation distance is strictly positive, because
$M$ is closed and there is a positive lower bound to the
injectivity radius, once a metric is fixed.
So $\gamma$ cannot be parabolic either.

\begin{lemma}\label{lema-pointsfar} 
  For any $K_0>0$, and any ray $r\subset s^+$, there exists $y\in r $ such that $d(y,h(y) ) > K_0$.
\end{lemma}
\begin{proof}
Let $r$ be a ray of $s^+$. Suppose for a contradiction that there exists $K_0$ such that for all $y$ in $r$ one has $d(y,h(y) ) < K_0$.

Recall that $h = \gamma^m \ft^n $, where $m$ and $n$ are fixed. By Lemma \ref{lema-boundedinfixcsleaf}, there exists a constant $K_1$ such that, for any $z$ in $L$,  $d(z,\ft^n(z)) < K_1$.
Thus, by assumption, for any $y\in r$, 
\begin{equation*}
 d(y,\gamma^m y) \leq d(y, \ft^n (\gamma^m y) ) + d(\gamma^m y, \ft^n (\gamma^m y) )  < K_0 + K_1.
\end{equation*}

Now, as explained above, $\gamma$ is an hyperbolic isometry for $d$. 
Hence, since $d(y, \gamma^m y)$ stays bounded for all $y$ in $r$, it implies that $r$ has to stay a bounded distance from the geodesic in $L$ that is the axis for the action of $\gamma$ on $L$.

So $\pi(r)$ stays a bounded distance away from the geodesic in $A= \pi(L)$ that lifts to the axis of $\gamma$.
Thus, Poincar\'e--Bendixson Theorem implies that $\pi(r)$ must accumulate onto a closed stable leaf in $M$, which is impossible (see Figure \ref{f.gromovhyp}).
\end{proof}

\begin{remark}\label{6.3nonhyp}
	Notice that this is the only place in the proof of Proposition~\ref{p.fixedcenter} that Candel's Theorem~\ref{thm-leafuniformisation} is used.
	
	In fact, the proof does not actually need $d$ to come from a Riemannian hyperbolic metric --- only that it is Gromov-hyperbolic: The only change one has to do in the proof is replace ``the'' geodesic realizing the axis of $\gamma$ by ``any'' geodesic. So Lemma \ref{lema-pointsfar} will hold as long as we know Gromov-hyperbolicity of the leafwise metric. We will need this in \cite{BFFP-sequel}.
\end{remark}

\begin{figure}[ht]
	\begin{center}
		\begin{overpic}[scale=0.6]{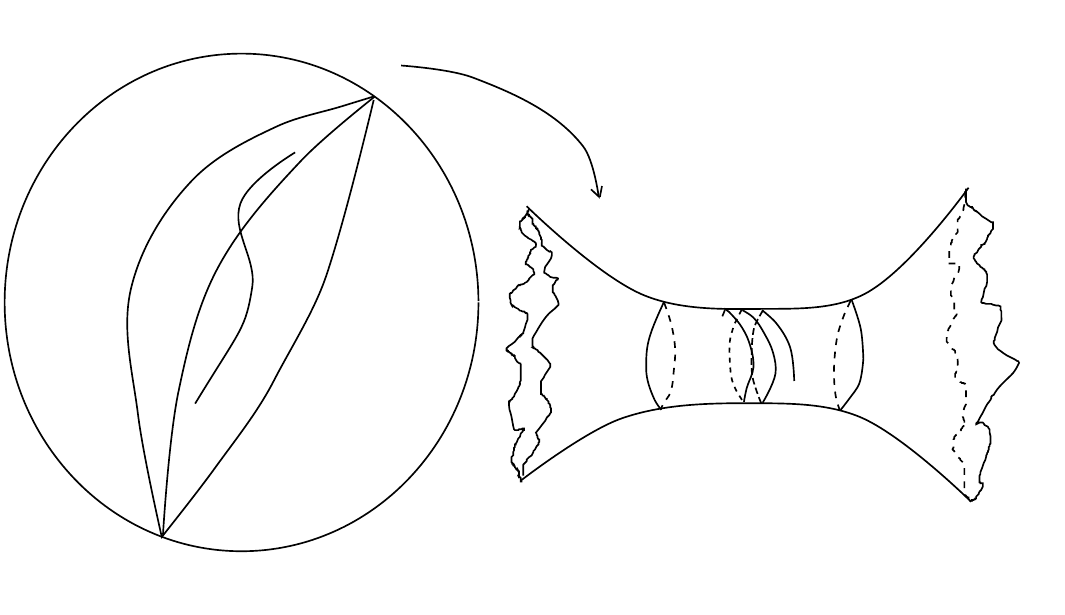}
		\put(45,10){$A=\pi(L)$}
		\put(50,49){$\pi$}
		\put(68,30){$\pi(r)$}
		\put(23,25){$r$}
		\put(7,7){$L$}
		\end{overpic}
	\end{center}
	\vspace{-0.5cm}
	\caption{If a stable ray in $L$ stays close to the axis of the deck transformation 
		$\gamma$ which is a hyperbolic isometry, then its projection
		in $M$ has to accumulate on a circle stable leaf.}\label{f.gromovhyp}
\end{figure}

\begin{proof}[Proof of Lemma~\ref{l.gapcontractionDC}]
	Let $y_1$ and $y_2$ be points in $s^+$ that lie on either side of, and far away from, the fixed point $x^+$, and let $\tau_1$ and $\tau_2$ be geodesic arcs from $y_1$ and $y_2$ to $s^-$. See Figure~\ref{f.7M_firstuseoffigure}. Note that the lengths of $\tau_i$, $i=1,2$, are uniformly bounded by Lemma~\ref{claim.boundeddistanceapart}, and since $f$ has bounded derivatives, the length of $h(\tau_i)$ is bounded as well. By Lemma~\ref{lema-pointsfar}, we can ensure that $h$ moves $y_i$ far enough to ensure that $h(\tau_i)$ is disjoint from $\tau_i$, and the lemma follows. 
\end{proof}

\begin{corollary}\label{c.contractedcenterDC}
	Some subsegment of $c'$ is contracted by $h$.
\end{corollary}
\begin{proof}
	This follows from Lemma~\ref{l.gapcontractionDC}, noting that $c'$ must intersect either $\tau_1$ or $\tau_2$.
\end{proof}

This completes the proof of Proposition~\ref{p.fixedcenter}.

\begin{remark}
	Note that Lemma~\ref{l.gapcontractionDC} also implies that the center leaf $c$ that separates $s^+$ from $s^-$ is ``coarsely contracted'' by $h$, in the sense that certain sufficiently large subsegments of $c$ are taken properly into themselves. To show this one needs to note that the set of fixed points of $h$ in $c$ is compact and this follows from the same argument as in Lemma~\ref{lema-pointsfar}. 
\end{remark}	
	This then generalizes as follows:

	\begin{lemma}\label{l.condition_for_coarse_contraction}
		Let $c$ be a center leaf in a center-stable leaf $L \subset \mt$. Suppose that $L$ is Gromov-hyperbolic, and fixed by $\ft$ and some non-trivial $\gamma \in \pi_1(M)$. Moreover, assume that there exist two stable leaves $s_1, s_2$ on $L$ such that:
		\begin{enumerate}
			 \item The center leaf $c$ is in the region between $s_1$ and $s_2$;
			 \item The leaves $s_1$ and $s_2$ are a bounded Hausdorff distance apart;
			 \item The leaves $c$, $s_1$ and $s_2$ are all fixed by $h = \gamma^n \circ \ft^m$, $m\neq 0$.
		\end{enumerate}
		Then, there exists a compact segment $I\subset c$, such that $h$ (if $m > 0$) or $h^{-1}$ (if $m < 0$) acts as a contraction on $c \smallsetminus \mathring{I}$.
	\end{lemma}

	This remains true without assuming dynamical coherence --- we will use it in \cite{BFFP-sequel}. The proof of this lemma is very similar to that of Lemma~\ref{l.gapcontractionDC}. Note that we do not need $c$ to make a perfect fit with $s_1$ or $s_2$, nor do we need that $c$ necessarily goes to both ends of the band determined by $s_1$ and $s_2$ as in Figure \ref{f.7M_firstuseoffigure}. All we need is that $c$ is between $s_1$ and $s_2$, and that both rays of $c$ escape every compact set in $L$. That last fact is true of any center leaf in $\mt$.

%%%%%%%%%%%%%%%%%%%%%%%%%%

%%%%%%%%%%%%%%%%%%%%%

\section{Mixed behavior} \label{s.nomix}

We can now eliminate mixed behavior in our cases of interest.

\begin{theorem}\label{mixedimpossible}
Let $f \colon M \to M$ be a dynamically coherent partially hyperbolic diffeomorphism homotopic to the identity. Assume that $\cW^{\cu}$ is $f$-minimal or that $M$ is hyperbolic or Seifert.

If a good lift $\ft$ fixes all the leaves of $\widetilde \cW^{\cs}$, then it also fixes all the leaves of $\widetilde \cW^{\cu}$.
\end{theorem}

\begin{proof}
Since $M$ is not $\mathbb{T}^3$ (recall that $\pi_1(M)$ is not virtually solvable), Proposition \ref{Ros} says that we can find a leaf in $\cW^{\cs}$ with non-trivial fundamental group. Let $L \in \widetilde \cW^{\cs}$ be a lift of such a leaf, which is invariant by some non-trivial $\gamma \in \pi_1(M)$.
By Proposition~\ref{p.fixedcenter}, $\ft$ fixes some center leaf $c$ in $L$, so it must fix the center-unstable leaf $K \in \widetilde \cW^{\cu}$ that contains $L$. From the dichotomy in Corollary \ref{coro-sumarizedichotomy}, it follows that $\ft$ fixes every leaf of $\widetilde \cW^{\cu}$ as desired.
\end{proof}

In particular, under the assumptions of this theorem, one rules out the mixed case (see item \ref{it.nomix} of section \ref{sss.proof_program}). 

%%%%%%%%%%%%%%%%%%%%%%%%%%

%%%%%%%%%%%%%%%%%%%%%%%%%%%
\section{Double invariance}\label{s.doublyinvariant}

In this section we show that, under the appropriate conditions, the \emph{doubly invariant case} (see item \ref{it.dinv} of section \ref{sss.proof_program}) leads to a discretized Anosov flow. 

\begin{theorem} \label{t.weakfixedimpliesAnosov}
  Let $f$ be a dynamically coherent partially hyperbolic diffeomorphism.  Assume that $M$ is hyperbolic or Seifert or that $\cW^{\cs}$ and $\cW^{\cu}$ are $f$-minimal. 
  If there exists a good lift $\ft$ which fixes a leaf of $\widetilde \cW^{\cs}$, then $f$ is a discretized Anosov flow. 
  \end{theorem}

Thanks to Proposition \ref{p.hypSeifminimal}, under any of the hypothesis of the Theorem, both $\cW^{\cs}$ and $\cW^{\cu}$ are $f$-minimal. 
Moreover, Theorem \ref{mixedimpossible} implies that $\ft$ must fix every leaf of both $\widetilde \cW^{\cs}$ and $\widetilde \cW^{\cu}$.

Notice that Theorem~\ref{t.weakfixedimpliesAnosov} together with the dichotomy of Corollary~\ref{c.minimalcase} proves Theorem~\ref{teo-main-coherent} from the introduction.

We will first show that connected components of the intersections of center-stable and center-unstable leaves are fixed by $\ft$ (i.e., that $\ft$ fixes leaves of the center foliation). To do that, we will prove that the set of connected components of intersections fixed by $\ft$ is both open and closed and then that it is non-empty.
Proving that $f$ is a discretized Anosov flow will then follow rather easily.

\subsection{Fixing center leaves}\label{ss.fixingcenter}

The main step in the proof of Theorem \ref{t.weakfixedimpliesAnosov} is the following proposition. Recall that the lift $\widetilde  \cW^c$ of the center foliation $\cW^{\mathrm{c}}$ consists of the connected components of the intersections between leaves of $\widetilde  \cW^{\cs}$ and $\widetilde  \cW^{\cu}$. 

\begin{proposition} \label{p.allcentralfixed}
 Let $f$ be a dynamically coherent partially hyperbolic diffeomorphism 
homotopic to the identity. Let $\ft$ be a 
good lift of $f$ which fixes every leaf
of $\widetilde \cW^{\cs}$ and $\widetilde \cW^{\cu}$.
Suppose that $\cW^{\cs}$ and $\cW^{\cu}$ are $f$-minimal in $M$. 
Then $\ft$ fixes every leaf of $\widetilde \cW^c$.
\end{proposition}

The key point in the proof of this proposition
is to show that either all leaves 
of $\widetilde  \cW^c$ are fixed by $\ft$,
or no leaf of $\widetilde \cW^c$ is fixed by $\ft$.
In the latter case we will use an argument similar to that of the analysis of the mixed behavior case, reaching a contradiction from the results of section \ref{s.perfectfits}.

\begin{lemma} \label{l.fixedcentralopen}
The set 
\[
 \cfix := \left\{ x \in \mt \mid \text{the center leaf through } x \text{ is fixed by } \ft\right\}
\]
is open in $\mt$.
In addition $\cfix$ is invariant under deck transformations.
\end{lemma}

\begin{proof}
 Let $c \in \widetilde \cW^c$ be such that $\ft(c) = c$. 
Let $L = \widetilde \cW^{\cs}(c)$ be the center-stable leaf containing $c$.
Let $\eps>0$ be small enough so that the center and stable foliations restricted to any ball of radius $\eps$ in $L$ is product (i.e., every stable and central leaf in the ball intersect each other).

Let $x\in c$.
By continuity of $f$, pick $\delta>0$ such that if $d(x,y)<\delta$ then $d(\ft(x), \ft(y)) < \eps$.
Up to taking $\delta$ smaller, and since $\ft(x)\in c$, we can assume that for any $y\in L$ such that $d(x,y)<\delta$, we have that $c(y)$, the central leaf through $y$, intersects $s(\ft(x))$, the stable leaf through $\ft(x)$.
This $\delta$ a priori depends on $x$.

Let $y\in L$ such that $d(x,y)<\delta$, then $c(y)\cap s(\ft(x))\neq \emptyset$. Moreover, since $d(\ft(x), \ft(y)) < \eps$, we have that $c(f(y))\cap s(\ft(x))\neq \emptyset$.
So the stable leaf $s(\ft(x))$ intersects both $c(y)$ and $c(\ft(y))$.

Now, as $\ft$ fixes the leaves of the central unstable foliations, we have that $\widetilde \cW^{\cu} (c(y)) = \widetilde \cW^{\cu}\left( c(\ft(y))\right)$. But $s(\ft(x))$ is transverse to $\widetilde \cW^{\cu}$, so it cannot intersect the same leaf of $\widetilde \cW^{\cu}$ more than once (see Theorem \ref{thm-tautfoliation}). Hence $c(\ft(y)) = c(y)$.

Thus, the set of center leaves fixed by $\ft$ in a center-stable leaf is open in that center-stable leaf. As the same argument applies
to the center-unstable leaves and it is uniform, we obtain that the union of points
in center leaves in $\cfix$ is open in $\mt$.
Finally, since $\ft$ commutes with every deck transformation, $\cfix$ is $\pi_1(M)$-invariant.
\end{proof}

We will think of $\cfix$ as both a subset of $\mt$ and
a collection of center leaves in $\widetilde \cW^c$.
Let $D := \pi(\cfix)$. By Lemma \ref{l.fixedcentralopen}, $D$ is open in $M$, and, obviously, $f$-invariant.

\begin{lemma} \label{l.allcentralfixedornone}
 Either $\cfix = \mt$ or $\cfix = \emptyset$.
\end{lemma}

\begin{proof}
Assume that $\cfix \neq \emptyset$, and thus, $D \not = \emptyset$.

We start by showing that every leaf of $\widetilde \cW^{\cs}$ has at least some fixed center leaves: Let $E$ be the $\widetilde \cW^{\cs}$-saturation of $\mathrm{Fix}^c_{\ft}$, and suppose for a contradiction that there exists $L$ a leaf of $\widetilde \cW^{\cs}$ such that $L\cap E =\emptyset$.
Since $\mathrm{Fix}^c_{\ft}$ is $\pi_1(M)$-invariant, so is $E$. Hence, for any $\gamma \in \pi_1(M)$, we have $\gamma L \cap \mathrm{Fix}^c_{\ft} = \emptyset$. Therefore, in $M$, we have
\[
\pi(L) \cap \pi(E) =\emptyset.
\]

So $\pi(L)$ is contained in the set $M \smallsetminus \pi(E)$, which is thus non-empty. But $\pi(E)$ is the $\cW^{\cs}$-saturation of $D$, hence open since $D$ is open. The set $\pi(E)$ is also $f$-invariant since $D$ is. Therefore, $M \smallsetminus \pi(E)$ is a non-empty, closed, $f$-invariant subset of $M$ saturated by $\cW^{\cs}$. The $f$-minimality of $\cW^{\cs}$ implies that $M \smallsetminus \pi(E) = M$, which is in contradiction with the fact that $\pi(E)$ is non-empty.
It follows that, for any center-stable leaf $L$, we have $L \cap \mathrm{Fix}^c_{\ft} \not = \emptyset$.

\medskip

Our next step is to prove that any center-stable leaf that has a non-trivial stabilizer in $\pi_1(M)$ is contained in $ \mathrm{Fix}^c_{\ft}$. 
Let $L$ be a leaf of $\widetilde \cW^{\cs}$ such that its projection $A = \pi(L)$ is not simply
connected (in which case it must be an annulus or a M\"{o}bius band according to Corollary \ref{coro-sumarizedichotomy}).
As we proved above, we know that $L\cap \cfix \neq \emptyset$. We now want to show that 
$L \subset \cfix$. Let us assume for a contradiction that $\cfix \cap L \neq L$. 

Recall that $\cfix$ is open (by Lemma \ref{l.fixedcentralopen}), thus so is $B= \cfix \cap L$ (for the relative topology on $L$). Notice that, since both $\cfix$ and $L$ are invariant by $\ft$, so is $B$, and in turn, so is its boundary $\partial B$.

Let $c_1$ be a center leaf in $\partial B$.
Then $\ft(c_1)\neq c_1$, but arbitrarily near $c_1$ there
are some fixed center leaves.
Since $c_1$ and $\ft(c_1)$ are both in $\partial B$, 
they are non-separated from each other in the leaf space of the center foliation in $L$.
Indeed, if one takes a sequence
$( c_n )$ of central leaves in $B$ that accumulates on $c_1$, then, since $\ft(c_n)=c_n$, the sequence also accumulates on $\ft(c_1)$.

As $c_1$ and $\ft(c_1)$ are not separated in the center leaf
space of $L$, it follows that there exists a stable leaf $s_1$ making a perfect fit with $c_1$, such that $s_1$ separates $c_1$ from $\ft(c_1)$.
 If some power of $\ft$ fixes $s_1$, then that power of
$\ft$ has a fixed point in 
$s_1$, 
contradicting Lemma \ref{l.nofixedpoints} (since $\ft$ fixes every leaves of $\widetilde \cW^{\cs}$).

It follows that the sequence 
$\left( \ft^i(s_1) \right)$ is infinite. Moreover, there exists $c \in \cfix$ that intersects all the leaves $\left( \ft^i(s_1) \right)$. Indeed, taking $c \in \cfix$ to be a central leaf close enough to $c_1$ so that $c\cap s_1\neq \emptyset$, then $c$ intersects every $\ft^i(s_1)$, because $\ft(c) =c$.
Furthermore, for all $i$, $\ft^i(s_1)$ separates $\ft^{i-1}(s_1)$ from $\ft^{i+1}(s_1)$,
because $\ft$ acts as a translation on $c$ (because $\ft$ cannot have a fixed point in $L$ by Lemma \ref{l.nofixedpoints}).

 As $\ft$ acts freely on the stable leaf space in $L$ 
(again thanks to Lemma \ref{l.nofixedpoints}),
then $\ft$ has an axis $A^s(\ft)$ for this action. By definition, all the leaves $\ft^i(s_1)$ are in this axis.
Since all the leaves $\ft^i(s_1)$ also intersect a common transversal $c$, we deduce that $A^s(\ft)$ is a line (see Figure \ref{f.S81}). 

Now recall that $C = \pi(L)$ is an annulus or a M\"{o}bius 
band (thanks to Corollary \ref{coro-sumarizedichotomy}).
Let $\gamma$ be the deck transformation associated
with the generator of $\pi_1(C)$, so  that $\gamma$ fixes $L$.

\begin{figure}[ht]
\begin{center}
\includegraphics[scale=0.54,angle=-90,origin=c]{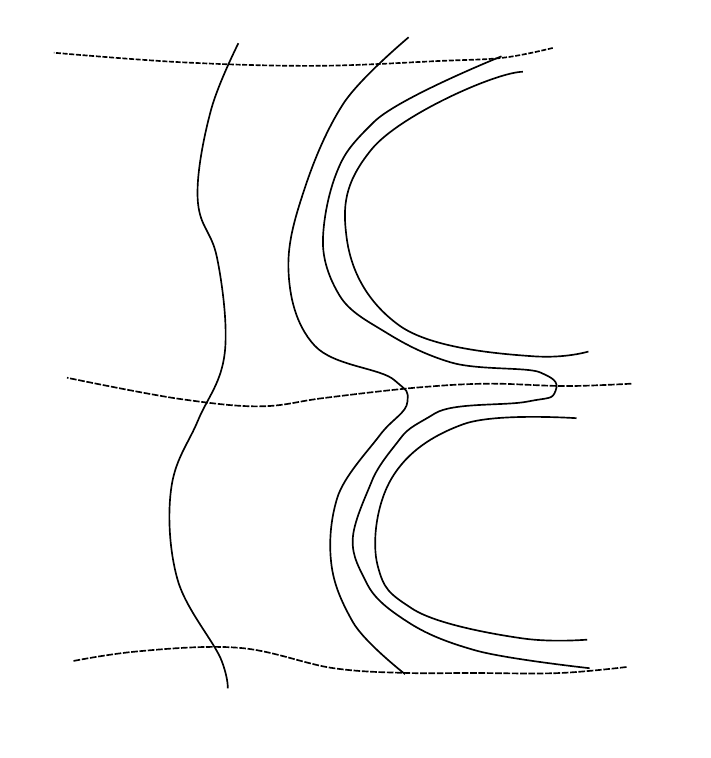}
\begin{picture}(0,0)
\put(-45,159){$B$}
\put(-139,136){$c_0$}
\put(-145,75){$\ft(c_1)$}
\put(-54,90){$c_1$}
\put(-45,110){$c$}
\put(-100,159){$s_{1}$}
\put(-170,159){$\ft(s_1)$}
\end{picture}
\end{center}
\vspace{-1.5cm}
\caption{{\small The combination of fixed and non-fixed center leaves allows to construct a center leaf 
intersecting $s_1$ and $\ft(s_1)$ in the axis $A^s(\ft)=A^s(\gamma)$.}\label{f.S81}}
\end{figure}

Recall also that, since there does not exist closed stable leaves in $M$, $\gamma$ must act freely on the stable leaf space in $L$. Thus $\gamma$ admits an axis $A^s(\gamma)$. Since $\ft$ and $\gamma$ commute, then
$A^s(\wt f) = A^s(\gamma)$ (see Proposition \ref{proposition-axes}). In particular, $A^s(\gamma)$ is a line.

Therefore there exists a $\gamma$-invariant curve in $L$ , that we call $\alpha$, such that $\alpha$ is transverse to the stable foliation,
and intersects each stable leaf in $A^s(\gamma) = A^s(\ft)$ exactly once. 
It follows that $\ft(\alpha)$ and $\alpha$ intersect exactly the same 
set of stable leaves in $L$. So we can use the Graph Transform argument (Lemma \ref{l.grapht}) on $\alpha$ and obtain that there exists a curve $c_0$ in $L$, tangent to the central direction\footnote{In fact, since there exists a central leaf that is transverse to the axis, an argument used in the proof of Lemma \ref{l.Axs-line} shows that $c_0$ is not just tangent to the central direction, but an actual central leaf. However, just having tangency to the central direction is enough to finish the proof.} $E^{\mathrm{c}}$, and invariant by both $\wt f$ and $\gamma$.

Since $c_0$ intersects $s_1$, and the leaves
$s_1$ and $c_1$ make a perfect fit, we deduce that there exists $s$, close to $s_1$, that intersects both  $c_0$  and $c_1$.
Let $x = c_0 \cap s$, $y = c_1 \cap s$ and $z = c_0 \cap s_1$. Up to choosing $s$ closer to $s_1$, we may assume that the distance between $x$ and $z$ is less than some fixed $K>0$, the length of the closed curve $\pi(c_0)$. Now, since $c_0$ is invariant by $\gamma$, we have that, for all $n$,
$$
d(\ft^n (x), \ft^n(z) ) \leq K.
$$
Moreover, since $\ft$ contracts stable length, we have that $d(\ft^n (x), \ft^n(y) )$ converges to $0$ as $n$ goes to $+\infty$.

Using the above, together with the invariance of $c_0$ by $\ft$ and the fact that $c_0$ is tangent to the central direction, we deduce that for some large enough $n$, the leaf $\ft^n(c_1)$ intersects $\ft^n(s_1)$, contradicting the fact that $s_1$ and $c_1$ do not intersect.

\begin{figure}[ht]
\begin{center}
\includegraphics[scale=0.94]{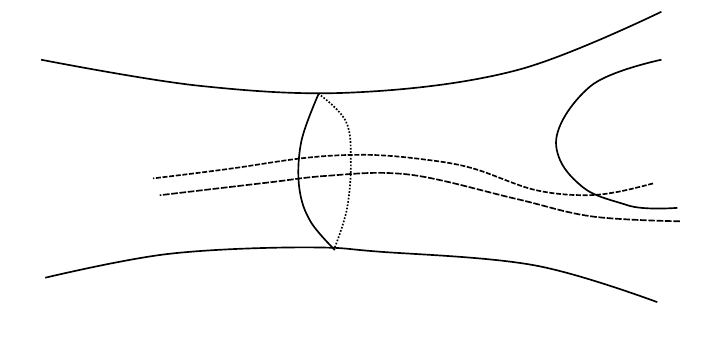}
\begin{picture}(0,0)

\put(-185,65){$\pi(z)$}
\put(-195,84){$\pi(x)$}

\put(-54,69){$\pi(y)$}
\put(-44,44){$\pi(s_1)$}
\put(-34,109){$\pi(c_1)$}
\put(-189,116){$\pi(c_0)$}
\put(-249,76){$\pi(s)$}
\end{picture}
\end{center}
\vspace{-0.5cm}
\caption{{\small The fixed center circle and the circles in the boundary of $\pi(B)$ are joined by a stable leaf.}\label{f.S82}}
\end{figure}

Hence, we proved thus far that for any $L$ a center-stable leaf with non-trivial stabilizer, we have $L \subset \cfix$. We can now finish the proof of Lemma \ref{l.allcentralfixedornone}.

Let $c$ be any center leaf in $\mt$. Let $x$ be a point in $c$, let $U$ be the center-unstable leaf containing $c$, and let $\tau$ be a small unstable segment in $\widetilde \cW^{\mathrm{u}}(x)$ that contains $x$ in its interior.
Recall (see Proposition \ref{Ros}) that there exists leaves in $\widetilde \cW^{\cs}$ with non-trivial stabilizer. Then $f$-minimality implies that such leaves are dense. Thus, we may assume that both endpoints of $\tau$ are on center-stable leaves with non-trivial stabilizer.
Call $c_1$ and $c_2$ the center leaves through the two endpoints of $\tau$. We proved above that both $c_1$ and $c_2$ are fixed by $\ft$ (since they are on center-stable leaves with non-trivial stabilizer).

Since
$c$ intersects $\tau$, an unstable segment from $c_1$ to $c_2$, it follows
that $c$ separates $c_1$ from $c_2$ in $U$. 
As $\ft$ fixes both $c_1$ and $c_2$ then $\ft(c)$ also
separates $c_1$ from $c_2$ in $U$. This implies that $\ft(c)$ 
also intersects $\tau$. 
As argued before, since 
$\ft$ fixes every center-stable leaves, $c$ and $\ft(c)$ must be in the same center-stable 
leaf, and, since they both intersect $\tau$, which is a transversal to the center-stable
foliation, we deduce that $c = \ft(c)$.
Therefore, we proved that $\ft$ fixes every center leaf, i.e., $\cfix = \mt$, as desired.
\end{proof}

We can now prove Proposition \ref{p.allcentralfixed}.

\begin{proof}[Proof of Proposition \ref{p.allcentralfixed}]
By assumption, $\ft$ fixes every leaf of $\wt\cW^{\cs}$, and, by Proposition \ref{Ros}, there exists some center-stable leaf with non-trivial stabilizer. Thus, Proposition \ref{p.fixedcenter} implies that there exists at least one fixed center leaf, i.e., $\cfix \neq \emptyset$. Lemma \ref{l.allcentralfixedornone}, then yields that $\cfix = \mt$, which is what we wanted to prove.
\end{proof}

\subsection{Showing that the map is a discretized Anosov flow}\label{ss.topoAnosov}

\begin{proposition}[Leaf conjugacy to a topological Anosov flow] \label{p.leafconjugacy}
 Let $f$ be a partially hyperbolic diffeomorphism on a $3$-manifold $M$. Suppose that there exists a lift $\ft$ to the universal cover $\mt$ such that $\ft$ moves points a bounded distance and $\wt f$ fixes every center leaf. Then the center foliation is the orbit foliation of a topological Anosov flow.
\end{proposition}

The proof is very similar to that given in \cite[Section 3.5]{BW}. We sketch the main points of the proof. We also refer to Appendix \ref{ss.DAF} for the precise definition of a topological Anosov flow, and more discussion about discretized Anosov flows.

\begin{proof}
Fix a metric on $M$ and consider $X^c$ a unit vector field in $E^{\mathrm{c}}$ which we first assume orientable. In the universal cover, using that $\ft$ fixes every center leaf, one can show that $\ft$ does not fix any point in $\mt$, that there is a uniform estimate for $d_c(x, \ft(x))$, and it is indeed continuous (see \cite[Lemma 3.4]{BW} for a proof with less hypothesis). In particular, we can assume that $[x,\ft(x)]^c$ is positively oriented with respect to $X^c$. 

Now, let $c_1,c_2$ be two center leaves in the same center-stable leaf such that for some $x\in c_1$ one has that 
$\widetilde \cW^s(x) \cap c_2 \neq \emptyset$. Then, letting $y$ be the point of intersection, we have that $d(\ft^n(x),\ft^n(y)) \to 0$ as $n \to \infty$. As the points are moving forward by
$\ft$ along the orbits of $X^c$ at bounded speed, this shows that the flow is locally contracted on center-stable manifolds. The symmetric arguments gives local contraction for the past in center-unstable manifolds. Notice that the fact that $\ft$ acts as a translation in all center leaves and that center leaves are fixed by $\ft$ implies that no deck transformation can reverse orientation of the center, this implies that our initial assumption is verified. 

This shows that the flow generated by $X^c$ is expansive. Moreover, it preserves the transverse foliations $\cW^{\cs}$ and $\cW^{\cu}$, which do not have singularities. Thus, the work of Paternain \cite{Paternain} implies that the flow generated by $X^c$ is a topological Anosov flow (see also Appendix \ref{ss.DAF}).
\end{proof}

Putting together Theorem \ref{mixedimpossible}, Proposition \ref{p.allcentralfixed}, Proposition \ref{p.leafconjugacy} and Proposition \ref{prop.equivDALC} one finishes the proof of Theorem \ref{t.weakfixedimpliesAnosov} and of Theorem \ref{teo-main-coherent}.

%%%%%%%%%%%%%%%%%%%%%%%%%%%

%%%%%%%%%%%%%%%%%%%%%%%%%%%%
\section{Proof of Theorem \ref{thmintro:Seifert}}\label{s.proof_of_thmA_DCcase}

We are now ready to finish the proof of Theorem \ref{thmintro:Seifert}.

We start by proving that, in a Seifert manifold, one can always choose a good lift in such a way that it fixes one center-stable leaf.

\begin{proposition}\label{p.liftfixleaf_DC}
Let $f\colon M \to M$ be a dynamically coherent partially hyperbolic diffeomorphism on a Seifert manifold. Suppose that $f$ is homotopic to the identity and that the Seifert fibration in $M$ is orientable.
Then there exists a good lift of an iterate of $f$ which fixes a leaf (and therefore every leaf) of $\wt\cW^{\cs}$. 
\end{proposition} 

\begin{proof}
To prove the result, we will need partial hyperbolicity for two things: To get that $M$ has non-zero Euler class (\cite[Theorem B]{HaPS}), and that the foliation is horizontal (\cite[Theorem 3.1]{HaPS}). 

First up to taking a finite lift we assume that $M$ is an orientable circle bundle over a higher genus (orientable) surface $\Sigma$. 

Consider the leaf space $\cL^{\cs}$ of the center-stable 
foliation and let $\delta$ be the deck transformation associated 
with the center of $M$. As the foliation is horizontal (see Theorem \ref{thm-horizontalseif}), 
it follows that the leaf space $\cL^{\cs}$ is homeomorphic
to the real line. 
In addition, $\cL^{\cs}/_{\langle \delta \rangle}$ 
is a circle that we will call $S^1_\delta$. 

Consider a good lift $\ft$ of $f$.
The map $\widetilde f$  induces a homeomorphism $\hat f\colon S^1_\delta \to S^1_\delta$. Moreover, $\hat f$ commutes with the image of the homeomorphisms $\hat \rho(\gamma)\colon S^1_\delta \to S^1_\delta$ which are defined for all $\gamma \in \pi_1(M)$. Note that 
$\hat \rho(\delta\gamma) = \hat \rho(\gamma)$, so $\hat \rho$ naturally induces a quotient representation $\rho\colon \pi_1(\Sigma) \to \mathrm{Homeo}_+(S^1_\delta)$ when using the identification $\pi_1(M)/_{\langle \delta \rangle} \cong \pi_1(\Sigma)$.  The Euler class of $M$ coincides with the one of the representation $\rho$ (see \cite[Chapter 4]{CandelConlonII}).

We first show that $\hat f$ has rational rotation number.
We proceed by contradiction: Assume that $\hat f$ has irrational rotation number.

Suppose first that $\hat f$ is minimal. It directly implies that $\hat f$ is conjugate to an irrational rotation by a homeomorphism $h\colon S^1_\delta \to S^1_\delta$. Conjugating the homeomorphisms $\rho(\gamma)$ (that is,
$h^{-1} \rho(\gamma) h$), since they commute with an irrational rotation they must commute with every rotation. 
Therefore the homeomorphisms $\rho(\gamma)$ are all conjugate by $h$ to rigid rotations. This implies that $\rho\colon \pi_1(\Sigma) \to \mathrm{Homeo}_+(S^1)$ is conjugate to a representation into $\operatorname{SO}(2,\R)$. This allows to construct a path to the trivial representation, because one can move 
freely along $\operatorname{SO}(2, \R)$ until one gets to the identity without
altering the relations. Therefore the representation 
has zero Euler class, see \cite[Sections 5.2 and 5.3]{Mann}, a contradiction. 

If $\hat f$ is not minimal, it is a Denjoy counterexample, one can see that the representation of $\pi_1(\Sigma)$ into $\mathrm{Homeo}(S^1)$ is semi-conjugate to a representation which commutes with a minimal homeomorphism, and so it also has to have zero-Euler class (see \cite[Section 5.2]{Mann}). This is a contradiction and proves that $\hat f$ has rational rotation number.

Now we go back to the original manifold. Since in the finite cover,
the corresponding map $\hat f$ had rational rotation number, the
same is true for $\hat f$ associated with the original manifold.
In particular, $\hat f$ has a periodic point, which means that for 
some $i \not = 0$, $\delta^n \widetilde f^i$ has a fixed point. So $\delta^n \widetilde f^i$ is the sought good lift (note that it is a good lift because the Seifert fibration is orientable, and thus $\delta$ is in  the center of $\pi_1(M)$)
This finishes the proof.
\end{proof}

Notice that the symmetric statement holds for $\wt \cW^{\cu}$ but a priori not for both simultaneously.

\begin{remark}\label{rem.same_proof_works_for_nonDC}
 In this proof, we did not really use dynamical coherence (see \cite{BFFP-sequel}). One could give a slightly simpler proof that uses dynamical coherence. However, since we will need this result in the non-dynamically coherent case in \cite{BFFP-sequel}, it is more efficient to give the general proof.
\end{remark}
  
So we can now prove Theorem \ref{thmintro:Seifert}.

 \begin{proof}[Proof of Theorem \ref{thmintro:Seifert}]
If the result holds in a finite cover of $M$, then, by projection, it also holds in $M$. So, by lifting to a double cover, we may assume that
$M$ has orientable Seifert fibration.
  Let $\ft$ be a good lift of some power $f^k$ given by Proposition \ref{p.liftfixleaf_DC}. Then $\ft$ does not act as a translation on both center-stable and center-unstable leaf spaces, so is not in case \ref{item.teo-main-coherent-double_translation} of Theorem \ref{teo-main-coherent}. Thus it is in case \ref{item.teo-main-coherent-discretized} of Theorem \ref{teo-main-coherent}, i.e., $f^k$ is a discretized Anosov flow.
 \end{proof}

 \begin{remark}\label{rem.classification_weird_examples}
  Note that, in Theorem \ref{thmintro:Seifert}, we need to take a power of $f$ to get a discretized Anosov flow, whereas Theorem \ref{teo-main-coherent} holds for the original $f$. This condition is necessary, i.e., there are some dynamically coherent partially hyperbolic diffeomorphisms homotopic to the identity on a Seifert manifold that are not discretized Anosov but such that a (non-trivial) iterate is. We will give such an example below and also classify all such examples.
  
  Consider $\Sigma$ a hyperbolic surface (or orbifold) and $g^t$ the geodesic flow on $T^1\Sigma$. Let $M$ be a $k$-fold cover of $T^1\Sigma$ obtained by unwrapping the fiber and $g^t_M \colon M \to M$ be a lift of $g^t$ to $M$. Call $s\colon M \to M$ the map obtained by lifting the ``rotation by $2\pi$'' along the fiber in $T^1\Sigma$.
  Then for any $i= 1, \dots, k-1$, the diffeomorphism $f_{k,i}:=g^1_M\circ s^i$ is a partially hyperbolic diffeomorphism, dynamically coherent, homotopic to the identity, and it is \emph{not} a dicretized Anosov flow (but $f_{k,i}^k$ is a discretized Anosov).
  Notice that the action of any good lift of $f_{k,i}$ on the center-stable and center-unstable leaf spaces is by translations.
  
  Now, suppose that $M$ is a Seifert manifold and $f$ is a dynamically coherent partially hyperbolic diffeomorphisms homotopic to the identity.
  Then, by Theorem \ref{thmintro:Seifert}, there exists $k$ such that $f^k$ is a discretized Anosov flow. Thus, by the classification of Anosov flows on Seifert manifolds (see \cite{Ghys,Barbot96}), $M$ is a finite lift of the unit tangent bundle of an orbifold $\Sigma$ and $f^k$ is leaf-conjugate to the time-one map of the lift of the geodesic flow. Then the action of (a good lift of) $f$ on both the central stable and central unstable leaf spaces is conjugated to the action of (a good lift of) a diffeomorphism $f_{k,i}$ as above. So $f$ and $f_{k,i}$ are leaf-conjugate.
 \end{remark}

%%%%%%%%%%%%%%%%%%%%%%%%%%%%

%%%%%%%%%%%%%%%%%%%%%%%%%%%%
\section{Coarse dynamics of translations}\label{s.coarse_dynamics_translations}

In this section, we consider a homeomorphism $f\colon M\to M$ of a hyperbolic $3$-manifold that preserves a uniform, $\R$-covered foliation $\fol$ and acts as a translation on its leaf space. We show that the dynamics of $f$ is comparable to the dynamics of the pseudo-Anosov flow $\Phi$ (given by Theorem \ref{teo-transversepA}) that regulates $\fol$. More precisely, for every periodic orbit of $\Phi$, we show that there exists a compact core (in a lift of $M$) invariant by $f$ that plays the role of the periodic orbit of $\Phi$. 

So in particular, the result of this section does not require $f$ to be partially hyperbolic and is of independent interest. The description of the dynamics of $f$ in periodic leaves of $\fol$ (if any) can be compared to the global shadowing for pseudo-Anosov homeomorphisms done in \cite{Handel}. 
We will use the results obtained here to complete the proof of Theorem \ref{thmintro:Hyperbolic} in the next section.

To make this comparison precise we will introduce some more objects.
Let $f\colon M \to M$ be a homeomorphism of a hyperbolic $3$-manifold. 
We assume that $f$ is homotopic to the identity, 
and preserves a foliation $\fol$. Furthermore, we suppose that $\fol$ is $\R$-covered and uniform and such that a good lift $\ft$ of $f$ acts as a translation on the leaf space of $\fn$.

Since $\ft$ commutes with any deck transformation and acts as a translation on the leaf space of $\fn$, it implies that the foliation $\fol$ is actually transversely orientable. Hence Theorem \ref{teo-transversepA} applies and we call $\Phi$ a transverse regulating pseudo-Anosov flow. We denote by $\wphi$ its lift to the universal cover $\mt$. We note that all periodic orbits of a pseudo-Anosov flow are homotopically non-trivial (in the case at hand this is even easier as the flow is transverse to a Reebless foliation). 

Let $\gamma \in \pi_1(M)$ be an element associated with a periodic orbit of $\Phi$ (i.e., such that there is a flow line of $\wphi$ invariant under $\gamma$). Let $M_\gamma:= \mt/_{<\gamma>}$ be the cover of $M$ associated with that deck transformation.
The foliation $\hat \fol_\gamma$ lifted from $\fol$ to $M_\gamma$ is a foliation by planes. Indeed, since $\Phi$ is regulating, each orbit of $\wphi$ can represent the leaf space $\cL_{\fn}$ of $\fn$. Thus $\gamma$, and all of its powers, act as a translation on $\cL_{\fn}$, so no leaf of $\fn$ can be fixed by a power of $\gamma$. Therefore, $\hat \fol_\gamma$ is a foliation by planes (see, e.g., \cite{Fen2002} for more details).
Since $\ft$ is a good lift of $f$ it induces a lift $\hat f_\gamma$ of $f$ in $M_\gamma$. 

For a periodic orbit $\alpha$ of $\Phi$, we call a (stable) \emph{half-leaf} of $\alpha$ any connected component of the complement of $\alpha$ in its stable leaf (so that a regular orbit has two half-leaves for each foliations, and a $p$-prong orbit has $p$ half-leaves).
We can now state precisely the main result of this section.

\begin{proposition}\label{p.homeotranslation} 
Let $M$, $f\colon M \to M$, $\fol$ and $\Phi$ be as above.

Then, for every $\gamma \in \pi_1(M)$ associated with a periodic orbit of $\Phi$, there is a compact $\hat f_{\gamma}$-invariant set $T_\gamma$ in $M_\gamma$ which intersects every leaf of $\hat \fol_{\gamma}$.

 Moreover, if an iterate $\hat f_{\gamma}^k$ of $\hat f_{\gamma}$ fixes a leaf $L$ of $\hat \fol_{\gamma}$, and $\gamma$ fixes all the half -leaves of the periodic
orbit associated with $\gamma$, then the fixed set of $\hat f_{\gamma}^k$ in $L$ is contained in $T_\gamma \cap L$ and has negative Lefschetz index. 
\end{proposition}

\begin{remark}\label{rem-indexpoints}
In fact, the proof will show that the total Lefschetz
index $I_{T_\gamma \cap L}(\hat f_{\gamma}^k|_{L})$ equals $-1$ if the periodic orbit of $\Phi$ is a regular periodic orbit, and equals $1-p$ if the periodic orbit is a $p$-prong, $p \geq 3$, assuming that
$\gamma$ fixes the half-leaves of the orbit.

We also remark that, by construction, the set $T_\gamma$ is \emph{essential} in the sense that any neighborhood of it contains a curve homotopic to (a power of) $\gamma$. 
\end{remark}

To prove this proposition, we first need to explore some properties of the pseudo-Anosov flow $\Phi$ and its interaction with the foliation $\fol$.

Let $\Lambda_s$ and $\Lambda_u$ be the weak stable and weak unstable (singular) foliations of the pseudo-Anosov flow $\Phi$. We denote by $\widetilde \Lambda_s$ and $\widetilde \Lambda_u$ their lift to the universal cover.
For any leaf $L$ of $\fn$, we denote by $\cG^s_L$ and $\cG^u_L$ the one-dimensional (singular) foliations obtained by intersecting the foliations $\widetilde \Lambda_s$ and $\widetilde \Lambda_u$ with $L$. 

\begin{fact}\label{fact-pA1}
The length along foliations $\cG_L^s$ and $\cG_L^u$ is uniformly efficient up to a multiplicative distortion at measuring distances in the leaves of $\fn$. That is, the rays of $\cG_L^s$ and $\cG_L^u$ are uniform quasi-geodesics for the path metric on $L$. 
\end{fact}

\begin{proof}
This fact is a consequence of the construction of the foliations $\widetilde \Lambda_s$ and $\widetilde \Lambda_u$. They are obtained by blowing down certain laminations that intersect the leaves of $\wt\fol$ along geodesics (with respect to the uniformization metric obtained via Candel's Theorem \ref{thm-leafuniformisation}). We refer to \cite{Fen2002} or \cite{Calegari:book} for the construction of these laminations.

In particular, there exists a uniform $K_1>1$ such that for every $L \in \fn$ and $y \in \cG^s_L(x)$ one has
$$ \ell ([x,y]^{\cG^s_L}) \leq K_1 d_L(x,y) + K_1 $$
where $\ell([x,y]^{\cG^s_L})$ denotes the length of the arc in $\cG^s_L$ joining $x$ and $y$. And similarly for $\cG^u_L(x)$.
\end{proof}

The flow $\wphi$ does not preserve the foliation $\fn$, but since it is transverse and regulating to the foliation, it makes sense to consider, given $L_1, L_2 \in \fn$ two leaves, the map $\tau_{12}\colon L_1 \to L_2$ consisting in flowing along $\wphi$ from one leaf to the other. By construction, the map $\tau_{12}$ is a homeomorphism. Notice that since $\fol$ is $\R$-covered and uniform, the Hausdorff distance between $L_1$ and $L_2$ is bounded multiplicatively with the flow distance between the leaves -- at least for
leaves which are sufficiently apart from each other. 

By convention, we will always assume that $L_2$ is taken to be \emph{above} $L_1$, in the sense that one has to follow the orbits of $\wt\Phi$ in the positive direction to go from $L_1$ to $L_2$.
Notice that invariance of $\widetilde\Lambda_s$ and $\widetilde \Lambda_u$ by $\wt \Phi$ imply that the homeomorphism $\tau_{12}$ maps the foliations $\cG^{\mathrm{s}}_{L_1}$ and $\cG^u_{L_1}$ into the the foliations $\cG^s_{L_2}$ and $\cG^u_{L_2}$ respectively. 

When the leaves $L_1, L_2$ are understood, we will omit them from the notation. It is a standard fact from the dynamics of pseudo-Anosov flows and the bounded comparison between flow distance and leaves\footnote{It is worth noting that the pseudo-Anosov property is invariant under reparametrizations of the flows.} that the following holds: 

\begin{fact}\label{fact-pA2} 
For any leaves $L_1$ and $L_2$ sufficiently far apart, the map $\tau_{12}$ expands lengths (and, equivalently distances) of $\cG^{\mathrm{u}}_{L_1}$ exponentially in terms of the Hausdorff distance between $L_1$ and $L_2$. 
That is, there exists a $\lambda>0$, independent of $L_1,L_2$, such that,  for any $x\in L_1$ and $y\in \cG^u_{L_1}(x)$, we have
\[
 d_{L_2}(\tau_{12}(x), \tau_{12}(y)) \geq e^{\lambda d_{\mathrm{Haus}}(L_1,L_2)} d_{L_1}(x,y),
\]
as long as $d_{\mathrm{Haus}}(L_1,L_2)$ is sufficiently big.
Similarly, $\tau_{12}^{-1}$ expands the lengths in $\cG^s$ exponentially in terms of the Hausdorff distance between $L_1$ and $L_2$. 
\end{fact}

The following simple result will be extremely useful for us. 
The leaf space is endowed with an orientation. If a deck transformation
$\beta$
acts freely on the leaf space,  we say that $\beta$ {\emph acts
decreasingly} if $\beta z < z$ with respect 
to the orientation for some (and hence all) $z$ in the leaf space.

\begin{lemma}\label{fact-pA3} 
Suppose that $\beta$ is a deck transformation
that acts freely and decreasingly on the leaf space
of $\fn$. 
Let $L_1$ be a 
leaf of $\widetilde \fol$. Let $\tau_{12}$ be the flow
along map from $L_1$ to $L_2:=\beta^{-1}(L_1)$. Define $\gbl :=
\beta \circ \tau_{12} \colon L_1 \to L_1$.
Fix a point $x_1$ in $L_1$.
Then, for every $K>0$, there exists $R>0$ such that for any $x$ in $L_1$ satisfying $d_{L_1}(x,x_1)> R$, then $d_{L_1}(x, g_{\beta,L_1}(x))> K$. 
\end{lemma}
\begin{remark}
 Notice that in this Lemma, we do not ask for $\beta$ to be associated with a periodic orbit of the pseudo-Anosov regulating flow.
\end{remark}

\begin{proof}
Suppose for a contradiction that there exists $K>0$ and a sequence $y_n$ escaping to infinity in $L_1$ and such that $d_{L_1}(y_n, \gbl(y_n))\leq K$ for all $n$. Then, up to taking a subsequence, there exists $\gamma_n \in \pi_1(M)$ such that $\gamma_n(y_n)$ converges to $y_0$ in $\mt$.

We define a map $\tau_{\beta}\colon \mt \rightarrow \mt$ as
follows: given $x$ in $\mt$,  it is in $L$ a leaf of 
$\widetilde \fol$, then we let $\tau_{\beta}(x)$ be the intersection
of the flow line of $\wt\Phi$ through $x$ with $\beta^{-1}(L)$. 
Notice that if $x\in L_1$, then $\tau_{\beta}(x) = \tau_{12}(x)$. In particular, for every $n$, $\tau_{\beta}(y_n) = \tau_{12}(y_n)$.

Since $\gamma_n(y_n)$ converges to $y_0$, and $\tau_{\beta}$ consists of flowing along $\wt\Phi$ a uniformly bounded amount, for $n$ big enough, we have that $d(\tau_{12}(y_n), \tau_{\beta}(\gamma_n^{-1}(y_0))$ is as small as we want. Hence, for $n$ big enough, we have 
$$
d(\beta \circ \tau_{12}(y_n), \beta \circ \tau_{\beta}(\gamma_n^{-1}(y_0))
< 1.
$$
This only uses that $\beta$ is an isometry of $\mt$.
Now, $\beta \tau_{12}(y_n) = \gbl(y_n)$ is at distance 
less than $K$ from $y_n$ in $L_1$ $-$ this is by hypothesis. 
Since path distance in the leaf is less than distance in $\mt$,
it follows that the same inequality is true for $d$.
The triangle inequality implies that
$d(y_n, \beta \circ \tau_{\beta}(\gamma^{-1}_n(y_0)) < 1 + K$ $-$ 
again $d$ is distance in $\mt$.
Thus, after applying $\gamma_n$, we get
$$
d(\gamma_n(y_n), \gamma_n \circ \beta \circ \tau_{\beta}(\gamma_n^{-1}(y_0)))
< 1 + K.$$
We are again using that $\gamma_n$ is an isometry of $\mt$.
Note that the map $\tau_{\beta}$ moves every point a bounded distance, the transformations $\gamma_n, \beta$ are isometries, and, for $n$ big enough, $d(\gamma_n(y_n),y_0)$
is very small. Therefore, 
$d(y_0, \gamma_n \circ \beta \circ \gamma_n^{-1}(y_0)) < K'$ for all $n$ big enough and a fixed constant $K'$. 

So we can extract a converging subsequence once more, and get that for any
$n, m$ the distance between $\gamma_n \beta \gamma_n^{-1}(y_0)$ and $\gamma_m \beta \gamma_m^{-1}(y_0)$ is smaller than the injectivity radius of $M$. It follows that
\[
 \gamma_m \beta \gamma_m^{-1} =
\gamma_n \beta \gamma_n^{-1},
\]
for all $n, m$.

Now we use that $M$ is hyperbolic. So $\beta$ is a hyperbolic
isometry of ${\mathbb H}^3 \cong \mt$. It has an axis with
ideal points $a, b$. 
Then, since $\gamma_n \beta \gamma_n^{-1} =
\gamma_1 \beta \gamma_1^{-1}$, we have that $\gamma_n(a) = \gamma_1(a)$ and $\gamma_n(b) = \gamma_1(b)$, for all $n$.
Let $c:=\gamma_1(a)$ and $d:= \gamma_1(b)$. Notice that (for all $n$) the axis of the isometry $\gamma_n \beta \gamma_n^{-1}$ has endpoints $c$ and $d$. 

Let $\alpha$ be the generator of
the group of deck transformations fixing $c, d$.
Then, for all $n$ there is an integer
$i_n$ so that $\gamma_n = \alpha^{i_n} \gamma_1$.
In addition since $y_n$ escapes compact sets in $L_1$ and
$L_1$ is properly embedded, it follows that $y_n$ escapes
compact sets in $\mt$. On the other hand 
$\gamma_n(y_n)$ converges to $y_0$, so
$|i_n|$ converges to infinity.

Notice that, since $\gamma_1$ sends the axis of $\beta$ to the axis of $\alpha$, a power of $\alpha$ is conjugated to a power
of $\beta$ by $\gamma_1$. Now, since $\beta$ acts freely on the
leaf space of $\widetilde \fol$, then so does $\alpha$.
Let $z_n := \gamma_1(y_n)$, $L := \gamma_1(L_1)$ and
$E$ the leaf of $\widetilde \cF$ through $y_0$.
Recall that $\gamma_n(y_n)$ converges to $y_0$. Hence

$$\alpha^{i_n}(z_n) \  = \  \alpha^{i_n} \gamma_1(y_n) \ = \ 
\gamma_n(y_n)$$
converges to $y_0$. But all $z_n$ are in the fixed leaf $L$.
It follows that $\alpha^{i_n}(L)$ converges to $E$,
and does not escape in the leaf space. This
contradicts the fact that $\alpha$ acts freely on the leaf 
space because
$|i_n| \to \infty$.
\end{proof}

We remark that this proof 
only uses geometry of $M$ and foliations.
That is, this proof 
works for 
any regulating flow transverse to a transversely
oriented,  $\R$-covered, uniform  foliation in a hyperbolic 
$3$-manifold.

The reason we will be able to compare the dynamics of $\ft$ and $\wt\Phi$ is thanks to the fact that they are a uniform bounded distance apart. That is, we have the following.

\begin{lemma}\label{lema-ftclose}
Let $f\colon M \to M$ be a homeomorphism of a hyperbolic 3-manifold $M$ preserving an $\R$-covered uniform  foliation $\fol$ and $\ft$ a good lift to $\mt$. There exists $R_1>0$ so that for every $L_1 \in \fn$, if $L_2 = \ft(L_1)$ and $x\in L_1$ then $d_{L_2}(\ft (x), \tau_{12}(x)) < R_1$. 
\end{lemma}

\begin{proof}
Since $\ft$ is a good lift it follows that one can join $x$ with $\ft(x)$ by an arc of bounded length. In particular, since the foliation $\fol$ is $\R$-covered and uniform, it follows that the Hausdorff distance between $L_1$ and $L_2 = \ft(L_1)$ is uniformly bounded above and below independently of $L_1 \in \fn$. Therefore, as explained before the statement of Fact \ref{fact-pA2} the amount of flowing needed to go from $L_1$ to $L_2$ is also uniformly bounded below and above. 
Thus $d_{\mt}(\ft(x),\tau_{12}(x))$ is uniformly bounded. Again
we use the fact that leaves of an $\R$-covered taut foliation are uniformly properly embedded in the universal cover (see \cite[Lemma 4.48]{Calegari:book}). The result follows.   
\end{proof}

Now we are ready to prove Proposition \ref{p.homeotranslation}.  

\begin{proof}[Proof of Proposition \ref{p.homeotranslation}] 

Let $\gamma \in \pi_1(M)$ be represented by a periodic orbit $\delta_0$ of $\Phi$ and take $\delta$ the unique lift of $\delta_0$ to $\mt$ fixed by $\gamma$.
 We will build the core $T_\gamma$ that we seek by considering a very large tubular neighborhood of $\delta$ and taking the intersection of this tubular neighborhood with all of its forward and backwards images under $\ft$ (see figure~\ref{fig-to1}). We will prove that this infinite intersection is non-empty, thus its projection to $M_\gamma$ will have the desired properties.
 
 Note that, if we build the core for a power $\gamma^{k_1}$ and $\ft^{k_2}$ instead, then taking its union with its images by $\gamma, \dots, \gamma^{k_1-1}$ and $\ft, \dots, \ft^{k_2-1}$ will produce a core which
is both $\gamma$ and $\ft$ invariant, as we want. 
 So, in this proof we may take any finite power of $\gamma$ or $\ft$.
 
 Thus, if $\delta_0$ is a $p$-prong, we replace $\gamma$ by a power if necessary, so that $\gamma$ fixes every half-leaf of $\delta$. Furthermore, we take a power of $f$ so that for any $L$, $\ft(L)$ is above $\gamma(L)$. For notations sake, we assume this is the original $f$.

\begin{figure}[ht]
\begin{center}
\includegraphics[scale=0.5]{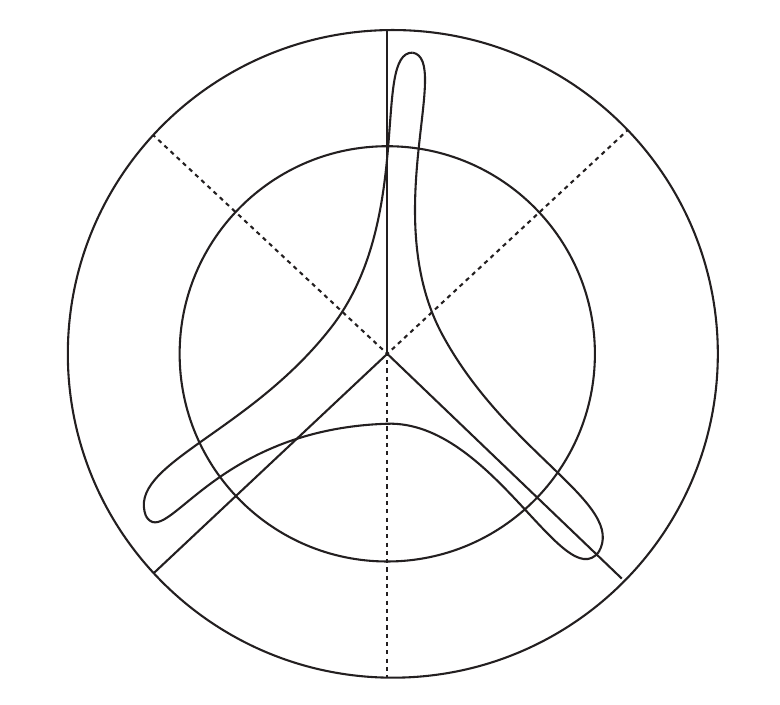}
\begin{picture}(0,0)
\put(-40,22){$a^1_L$}
\put(-162,22){$a^2_L$}
\put(-100,166){$a^3_L$}
\put(-40,140){$r^1_L$}
\put(-100,0){$r^2_L$}
\put(-162,140){$r^3_L$}
\end{picture}
\end{center}
\vspace{-0.5cm}
\caption{{\small The image of a large tubular neighborhood of the lift of the prong by $\ft$ in a given center-stable leaf.}}\label{fig-to1}
\end{figure}

For any leaf $L$ in $\fn$, we write $x_L$ to be the (unique) intersection of $\delta$ with $L$.
Let $a^i_L$, with $i=1, \dots, p$, be all the ideal points on the boundary at infinity of $L$ of the weak unstable leaf (of $\wt\Phi$) through $\delta$ intersected with $L$, where $p=2$ if $\delta$ is a regular orbit and otherwise $\delta$ is a $p$-prong orbit. Equivalently, $a^i_L$ is the ideal point determined by each ray of $\cG^u_L(x_L)$.
Similarly, we define $r^i_L$, $i=1, \dots, p$, to be the ideal ends of the rays of $\cG^s_L(x_L)$.

For every $L \in \fn$ and for every $i$, we choose $P^i_{L}$ and $N^i_{L}$ neighborhoods (in $L \cup \partial_{\infty} L$) of, respectively $a^i_{L}$ and $r^i_{L}$. We also choose these neighborhoods such that their boundary (in $L$) are geodesics for the path metric on $L$.
Furthermore, we choose these neighborhoods in such a way that they depend continuously on $L\in \fn$ and they are $\gamma$-invariant, i.e., $\gamma (P^i_{L}) = P^i_{\gamma (L)}$ and $\gamma (N^i_{L}) = N^i_{\gamma (L)}$.
Up to taking the neighborhoods smaller, we assume that for any $L$ and any $i,j$, $P^i_L \cap N^j_L = \emptyset$;  and for 
any $i \not = j$, $P^i_L \cap P^j_L = \emptyset$, $N^i_L \cap N^j_L = \emptyset$.

We define a map $\tau_f\colon \mt \to \mt$ in the following way:
For any $L$ in $\fn$ and any $x\in L$, $\tau_f(x)$ is the intersection of the orbit of $\wt\Phi$ through $x$ with $\ft(L)$.

Let $R_1$ be the constant given by Lemma \ref{lema-ftclose} (i.e., such that $\ft$ and $\tau_f$ are $R_1$-close). Using that leaves of $\cG^u_L(x_L)$ and $\cG^s_L(x_L)$ are quasigeodesics and that the flow in a given time expands and contracts their length by a factor different from one, one can construct neighborhoods $P^i_L$ and $N^i_L$ sufficiently small so that:

\begin{enumerate}[label=(\roman*)]
 \item For any $L$ and any $i$,
 \[
  \tau_{f}(P^i_L) \subset P^i_{\ft(L)} \text{ and } 
d_{\ft(L)}\left(\tau_{f}(P^i_L), 
\partial P^i_{\ft(L)}\right) > 10R_1.
 \]
 \item For any $L$ and any $i$, 
 \[
  \tau_{f}^{-1}(N^i_L) \subset N^i_{\ft^{-1} (L)} \text{ and } 
d_{\ft^{-1}(L)}\left(\tau_{f}^{-1}(N^i_L), 
\partial N^i_{\ft^{-1} (L)}\right) > 10R_1.
 \]
\end{enumerate}
To make sense of the distance between sets in $L \cup \partial_{\infty}L$ above, we decide that ideal points are at infinite distance from any other point.
A direct consequence of the conditions above is that
\begin{enumerate}
 \item\label{item_P} For any $L$ and any $i$, $\ft(P^i_L) \subset P^i_{\ft(L)}$
 \item\label{item_N} For any $L$ and any $i$, $\ft^{-1}(N^i_L) \subset N^i_{\ft^{-1} (L)}$
\end{enumerate}

Lemma \ref{lema-ftclose} shows that, for any $L$, the maps $\tau_f|_L$ and $\ft|_L$ are a finite distance from each other. Thus their extension to the circles at infinity $\partial_{\infty} L$ is the same. Now, recall that $\tau_f$ corresponds to flowing along the pseudo-Anosov flow $\wt\Phi$.
Hence, up to replacing $\ft$ by a very high power of $\ft$, we can moreover assume that:
\begin{enumerate}
\setcounter{enumi}{2}
 \item \label{item_no_escape_future} For any $L$ and $i$, if $\omega$ is an ideal endpoint of $\partial N_L^i$, then  $\ft(\omega) \in P_{\ft(L)}^j$, for some $j$ (where $j$ is the unique index such that $a^j_L$ is the first attractor on the side of $\omega$ from $r^i_L$);
 \item \label{item_no_escape_past} For any $L$ and $i$, if $\omega$ is an ideal endpoint of $\partial P_L^i$, then $\ft^{-1}(\omega) \in N_{\ft^{-1}(L)}^j$, for some $j$ (where $j$ is uniquely determined as above).
\end{enumerate}
 Note that conditions (\ref{item_P}) and (\ref{item_N}) still are satisfied  by our high power of $\ft$.

Now we choose a constant $R$ large enough so that for every $L$ and every $i$, the ball $D_L:= B(x_L, R)$, of radius $R$ around $x_L$, intersects every $P^i_L$. Moreover, we choose it to satisfy:
\[ D_L \supset \ft\left(\partial N_{\ft^{-1}(L)}^i\right)\smallsetminus \bigcup_j P_L^j \quad , \quad 
D_L \supset \ft^{-1}\left(\partial P_{\ft(L)}^i\right)\smallsetminus \bigcup_j N_L^j  \]
This is possible because the ideal points of 
$A = \ft(\partial N^i_{\ft^{-1}(L)})$ are contained in
the interior of the ideal boundary of the union
of the $P^j_L$. It follows that for each $L$ only a compact part
of $A \cap L$ is outside the union, and this
varies continuously with $L$. By choosing $R$ big
enough one satisfies the equations above.

Let 
\[
 V := \bigcup_{L\in \fn} D_L.
\]
We will show that the set $\bigcap_{n\in\ZZ} \ft^n(V)$ is non-empty and thus its projection to $M_{\gamma}$ is the core $T_{\gamma}$ that we seek.
The proof will be done by induction. In order to make that induction work, we need the following
\begin{claim}\label{claim_for_induction_future}
 Let $L$ be a leaf in $\fn$. Let $C \subset D_L$ be any compact and path-connected set that does not intersect any $N^i_{L}$.
 
 If there exists $i_1,i_2$ distinct such that $C$ intersect both $P^{i_1}_L$ and $P^{i_2}_L$, then there exists a path-connected component of $\ft(C)\cap D_{\ft(L)}$ that intersects $P^{i_1}_{\ft(L)}$ and $P^{j_2}_{\ft(L)}$, for some $j_2\neq i_1$ ($j_2$ is not necessarily $i_2$) and that does not intersect any $N^i_{\ft(L)}$.
\end{claim}

\begin{proof}
Since $C$ intersects $P^{i_1}_L$ and $P^{i_2}_L$, $\ft(C)$ also intersects both $P^{i_1}_{\ft(L)}$ and $P^{i_2}_{\ft(L)}$ (thanks to the condition (\ref{item_P})).
Now, since $C$ does not intersect any $N^i_L$, because of condition (\ref{item_no_escape_future}) and the choice of $D_L$, the intersections of $\ft(C)$ with $\partial D_{\ft(L)}$ are contained in the union
of the $P^i_{\ft(L)}$.

Thus, as claimed, $\ft(C)\cap D_{\ft(L)}$ contains a connected component that intersects $P^{i_1}_{\ft(L)}$ and $P^{j_2}_{\ft(L)}$, for some $j_2\neq i_1$ ($j_2$ is not necessarily $i_2$) and that does not intersect any $N^i_{\ft(L)}$.
\end{proof}

Figure \ref{fig-to2} shows a case where $j_2$ is not equal to $i_2$: It may be that $\ft(C)$ stretches well into $P^{j_2}(\ft(L))$ and out of $D_{\ft(L)}$. Thus, as in the figure, the intersection $\ft(C) \cap D_{\ft(L)}$ can have two components
$C_1$ and $C_2$, neither of which intersects both $P^{i_1}_{\ft(L)}$ and $P^{i_2}_{\ft(L)}$.

\begin{figure}[ht]
\begin{center}
\includegraphics[scale=0.6]{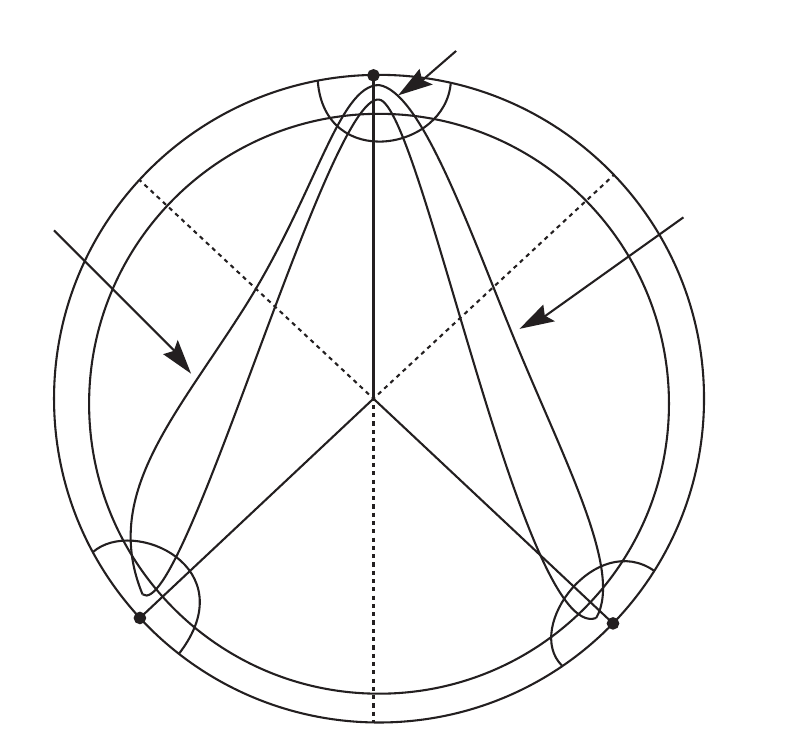}
\begin{picture}(0,0)
\put(-85,85){$C_2$}
\put(-182,85){$C_1$}
\put(-107,208){{\small $\ft(C)$}}
\put(-32,158){{\small $\ft(C)$}}
\put(-239,155){{\small $\ft(C)$}}
\put(-16,104){$\ft(L)$}
\put(-202,23){$P^{i_1}_{\ft(L)}$}
\put(-52,23){$P^{i_2}_{\ft(L)}$}
\put(-134,209){$P^{j_2}_{\ft(L)}$}
\put(-156,65){$D_{\ft(L)}$}
\end{picture}
\end{center}
\vspace{-0.5cm}
\caption{{\small The intersection $\ft(C)\cap D_{\ft(L)}$ may not have a connected set joining $P^{i_1}$ to $P^{i_2}$.}}\label{fig-to2}
\end{figure}

The same proof as above, using $\ft^{-1}$ instead (and the conditions (\ref{item_N}) and (\ref{item_no_escape_past})), gives

\begin{claim}\label{claim_for_induction_past}
 Let $L$ be a leaf in $\fn$. Let $C \subset D_L$ be any compact and path-connected set that does not intersect any $P^i_{\ft(L)}$. 
 
 If there exists $i_1,i_2$ distinct such that $C$ intersect both $N^{i_1}_L$ and $N^{i_2}_L$, then there exists a path-connected component of $\ft^{-1}(C)\cap D_{\ft^{-1}(L)}$ that intersects $N^{i_1}_{\ft^{-1}(L)}$ and $N^{j_2}_{\ft^{-1}(L)}$, for some $j_2\neq i_1$ and that does not intersect any $P^i_{\ft^{-1}(L)}$.
\end{claim}

For any leaf $L$ and any integer $n\geq 0$, define
\[
 R_L^n = \bigcap_{k=0}^n \ft^k(D_{\ft^{-k}(L)}) \quad \text{and} \quad Q_L^n = \bigcap_{k=0}^n \ft^{-k}(D_{\ft^{k}(L)}).
\]

\begin{claim}\label{claim_R_and_Q_good_property}
 For every $i$ and every $n$, $R_L^n$ contains a subset $C$, compact and path-connected that does not intersect any $N^j_{L}$ but does intersect $P^i_L$ and some $P^{i_2}_L$ (for some $i_2\neq i$).
 
 Similarly, for every $i$ and every $n$, $Q_L^n$ contains a subset $C$, compact and path-connected that does not intersect any $P^j_{L}$ but does intersect $N^i_L$ and some $N^{i_2}_L$ (for some $i_2\neq i$).
\end{claim}

\begin{proof}
We only do the proof for $R_L^n$, as the claim for $Q_L^n$ follows similarly.

First, since $R^0_L = D_L$, the claim is true for $n=0$ and any leaf $L$ (because $D_L$ clearly contains such a subset).
Let us assume that the claim holds for $R^{n-1}_L$ and for any $L$.
Then, Claim \ref{claim_for_induction_future} implies that (for any $L$) $\ft(R^{n-1}_L)\cap D_{\ft(L)}$ has a compact and path-connected subset that does not intersect any $N^j_{L}$ but does intersect $P^i_L$ and some $P^{i_2}_L$ (for some $i_2\neq i$).

But, by definition, we have 
\[
 R_L^n = \bigcap_{k=0}^n \ft^k(D_{\ft^{-k}(L)}) = \ft\left(R^{n-1}_{\ft^{-1}(L)}\right) \cap D_L.
\]
Thus the claim is proved.
\end{proof}

Now, since for any $L$, the ideal points $a^i_L$ and $r^i_L$ alternate, the properties of $R^n_L$ and $Q^n_L$ given by Claim \ref{claim_R_and_Q_good_property} imply that, for all $n$, $R^n_L\cap Q^n_L$ is a non-empty compact set.

Since $R^n_L$ and $Q^n_L$ are decreasing sets, the set 
\[
 T_L := \bigcap_{n\geq 0} \left(R^n_L\cap Q^n_L\right)
\]
is (for any $L$) non-empty and compact. Thus
\[
 T:= \bigcup_{L\in\fn} T_L
\]
is non-empty, and, by construction, $\ft$-invariant (note also that $T = \cap_{n\in\ZZ} \ft^n(V)$ as we claimed above).
Hence, the projection $T_\gamma$ of $T$ to $M_{\gamma}$ is non-empty, compact and $\hat{f}_\gamma$-invariant.

Once $T_{\gamma}$ is built, the second half of Proposition \ref{p.homeotranslation} follows directly from the homotopy invariance of Lefschetz index together with Lemma \ref{lema-ftclose}. So we finished the proof of Proposition \ref{p.homeotranslation}.
\end{proof}

In the proof of Proposition \ref{p.homeotranslation}, we obtained the following result which we state independently for future reference.

\begin{lemma}\label{fixedpointbounded} 
Let $f\colon M \to M$ be a homeomorphism of a hyperbolic 3-manifold, $f$
homotopic to the identity. Suppose that $f$ preserves an $\R$-covered uniform foliation $\cF$ and that a good lift $\ft$ of $f$ acts as a translation on the leaf space of $\fn$. Let $\gamma \in \pi_1(M)$ be a deck transformation.

If $h=\gamma \circ \ft^n$ fixes some leaf $L \in \fn$ (with $n\neq 0$) then the set of fixed points of $h$ in $L$ is contained in a compact subset of $L$.

Moreover, given $n > 0$ big enough, then 
for every $R>0$ there is a compact set $D \subset L$ such that if $y \notin D$ then $d_{L}(y,h(y))>R$.

Finally, let $P$ be the set of ideal points in the boundary at infinity $S^1(L)$ that are attracting and fixed under the map $\gamma \circ \tau_{12}$, where $\tau_{12} \colon L \to \ft^n(L)$ is the flow along $\Phi$ map.
Then, for any $y\in P$, there exists a  neighborhood $U$
of $y$ in $L \cup S^1(L)$ 
such that 
\begin{enumerate}
 \item $h(\overline U)$ is strictly contained in $\overline U$, and
 \item $\bigcap_{i \geq 0} h^i(U) = \{ y \}$.
\end{enumerate}
\end{lemma}

\begin{remark}
 Notice that the results of this section should be adaptable to the case of a homeomorphism acting as a translation on the leaf space of a manifold with one atoroidal piece. What would be required is some sort of analogue of Theorem \ref{teo-transversepA}. That is, we would need to know that there exists a transverse regulating flow such that any orbit that stays in the atoroidal piece is a hyperbolic $p$-prong ($p\geq 2$). Although that result seems likely to be true, it has not been proven. A similar case is dealt with in a companion paper \cite{BFFP_companion} where we study integrability for partially hyperbolic diffeomorphisms not homotopic to identity in Seifert manifolds. 
\end{remark}

%%%%%%%%%%%%%%%%%%%%%%%%%%%

%%%%%%%%%%%%%%%%%%%%%%%%%%%%
\section{Double translations in hyperbolic manifolds}\label{sec-thmB}

In this section we prove Theorem \ref{thmintro:Hyperbolic}. 

Let $f\colon M \to M$ be a dynamically coherent partially hyperbolic diffeomorphism of a hyperbolic 3-manifold $M$. 
Recall that we denote by $\cW^{\cs}$ and $\cW^{\cu}$ a pair
of $f$-invariant foliations tangent respectively to $E^{\cs}$ and $E^{\cu}$. Up to taking an iterate, one has that $f$ is homotopic to identity and therefore has a good lift $\ft$ to $M$. We fix that good lift.

We want to show that $\ft$ fixes the leaves of both foliations $\widetilde \cW^{\cs}$ and $\widetilde \cW^{\cu}$. By Theorem \ref{t.weakfixedimpliesAnosov} this is enough to prove Theorem \ref{thmintro:Hyperbolic}. Notice that by Corollary \ref{coro-sumarizedichotomy} and Theorem \ref{mixedimpossible} we can assume by contradiction that both foliations are $\R$-covered and uniform and that $\ft$ acts as translation on both leaf spaces.

We in fact will get a contradiction using just one of the translations thanks to Proposition \ref{p.homeotranslation}, together with the following result. Notice that we thus obtain an alternative proof, albeit much more complicated, of the fact that there cannot be a mixed behavior in a hyperbolic manifold. 

\begin{proposition}\label{fixedleaf_DC}
Assume that a good lift $\ft$ of $f$ acts as a translation on the foliation $\wt\cW^{\cs}$ and let $\Phi$ be a transverse regulating pseudo-Anosov flow for $\cW^{\cs}$. Then, for every $\gamma \in \pi_1(M)$ associated to the inverse of a periodic orbit $\gamma$ of $\Phi$ there is $n>0, m>0$ such that $h= \gamma^n \circ \ft^m$ fixes a leaf $L$ of $\wt\cW^{\cs}$. 
\end{proposition}

By symmetry, the same result holds if applied to $\cW^{\cu}$. Notice that once one knows that $h$ fixes a leaf $L$ of $\wt\cW^{\cs}$, the second part of Proposition \ref{p.homeotranslation} applies to $f$. 

\begin{proof}
Thanks to Proposition \ref{p.homeotranslation}, we can consider the cover $M_\gamma= \mt /<\gamma>$ and let $V$ be a compact solid torus in 
$M_{\gamma}$ such that $\bigcap_{n\in\ZZ} \hat f_{\gamma}^n(V)= T_\gamma$ is compact, non-empty, and far from $\partial V$.
Let $z \in T_\gamma$. Let $y\in T_{\gamma}$ be an accumulation point of $\left(\hat f_{\gamma}^n(z)\right)$.

 We take $i,j$ big enough, with $j$ much bigger than $i$, such that $\hat f_{\gamma}^i(z)$ and $\hat f_{\gamma}^j(z)$ are both very close to $y$.
 Now, consider $I$ a small closed unstable segment containing
$\hat f^i(z)$ in its interior.
 Since $\hat f^{j-i}_{\gamma}$ increases the unstable length, every leaf of $\hat\cW^{\cs}$ through $I$ intersects the interior of $\hat f^{j-i}_{\gamma}(I)$.
This set of $\hat\cW^{\cs}$ leaves is an interval.
This produces a fixed $\hat \cW^{\cs}$ leaf under $\hat f_{\gamma}^{j-i}$.
Lifting to $\mt$ proves the proposition.
\end{proof}

We can now finish the proof of Theorem \ref{thmintro:Hyperbolic}.
\begin{proof}[Proof of Theorem \ref{thmintro:Hyperbolic}]
 
 Let $\ft$ be a good lift of $f$ and let $L_0$ be a leaf fixed by $h:=\gamma \circ \ft^k$ for some $k>0$ and $\gamma \in \pi_1(M) \setminus \{\mathrm{id}\}$ given by Proposition \ref{fixedleaf_DC}.
 
 For any leaf $L$ fixed by $h$, the map $h|_{L}$ has negative Lefschetz index (according to Proposition \ref{p.homeotranslation}). We stress
that this is the two dimensional Lefschetz index
of $h|_L$ and not the of the map $h$ on a $3$-manifold.
Thus there exists a point $x_L\in L$ fixed by $h$. Now, $h$ is partially hyperbolic, so any fixed leaf $L$ is repelling along the unstable manifold through $x_L$.
 But this is impossible, as in the leaf space of $\wt\cW^{\cs}$, the closed interval between $L_0$ and $\gamma (L_0)$ is mapped to itself so cannot contain only repelling fixed points.

This contradiction implies that $\ft$ cannot act as a translation
on either leaf spaces. 
 It follows that $\ft$ has to fix every center-stable and center-unstable leaf. Therefore by Theorem \ref{t.weakfixedimpliesAnosov}, it is conjugate to a discretized Anosov flow.
This proves Theorem B.
\end{proof}

%%%%%%%%%%%%%%%%%%%%%%%%%%%%%%

%%%%%%%%%%%%
% Appendix %
%%%%%%%%%%%%
\begin{appendix}

\section{Some $3$-manifold topology} \label{app.3_manifold_topology}

We collect here some concepts from $3$-manifold topology that were
used in this article.   
We refer the reader to \cite{He,Hatcher} for more background. 

A 3-manifold (which we always mean to be a smooth manifold) is \emph{irreducible}, if every 
smoothly embedded sphere bounds a ball. It is well known that closed 3-manifolds admitting taut foliations are irreducible (see, e.g., \cite{Calegari:book}). 

An irreducible compact, 3-manifold $M$ is said to be \emph{homotopically atoroidal} if every $\pi_1$-injective map of a torus in $M$ is homotopic to a map into the boundary of $M$.
The manifold is \emph{geometrically atoroidal} if every $\pi_1$-injective, embedded smooth torus is homotopic to the boundary of $M$.

If a manifold with exponential growth of fundamental group is 
homotopically atoroidal, then by the geometrization theorem it is \emph{hyperbolic}, i.e., 
the interior of $M$ admits a complete,
Riemannian metric of constant negative curvature.

Notice that when $M$ is homotopically
atoroidal, $\pi_1(M)$ does not contain any subgroup isomorphic to $\ZZ^2$.

A $3$-manifold is called a \emph{Seifert manifold} if it admits a partition by distinct circles such that a tubular neighborhood of each fiber is homeomorphic by a fiber-preserving homeomorphism to either:
\begin{itemize}
 \item A fibered solid torus of type $(p,q)$. This is a torus obtained from $\mathbb{D}^2 \times [0,1]$ by identifying $\mathbb{D}^2 \times \{0\}$ to $\mathbb{D}^2 \times \{1\}$ via the map $(z,0) \mapsto (z\exp(2\pi i p/q), 1)$. The fiber $\{0\} \times S^1$ is called \emph{regular} if $p=0$ and \emph{exceptional} otherwise. Or,
 \item A fibered solid Klein bottle, obtained from $\mathbb{D}^2 \times [0,1]$ by identifying $\mathbb{D}^2 \times \{0\}$ to $\mathbb{D}^2 \times \{1\}$ via the map $(z,0) \mapsto (\bar{z}, 1)$. The fibers $\{z\}\times S^1$, $z\in \R$, are also called \emph{exceptional}.
\end{itemize}

The quotient of a Seifert manifold by the Seifert fibration, called the \emph{base}, $B$, has a structure of a $2$-orbifold (without corner reflectors). The exceptional fibers separate into two sets: The axis of the fibered solid torus projected to isolated points in the interior of $B$ (called \emph{conical points}), while the exceptional fibers coming from fibered solid Klein bottles projects to a closed $1$-submanifold of the boundary of $B$ (and each connected component is called a \emph{reflector curve}).

\begin{remark}
 The definition above is not the one originally taken by Seifert. Indeed, the fibered solid Klein bottles neighborhood were not allowed in the original definition. However, it is now more common to use this definition. In particular, with this definition, works of Epstein and Tollefson imply that all $3$-manifolds foliated by circles are Seifert (see, e.g., \cite{Sco83}).
 
 Note that both the original definition and the one chosen here agree when the manifold is assumed orientable.
\end{remark}
If a Seifert manifold has fundamental group
with exponential growth, then it is finitely
covered by a circle bundle over a surface of genus $\geq 2$. 
In particular, thanks to the classification of Seifert manifolds (see \cite[Theorem 3.8]{Sco83}), the Seifert fibration is unique in this case.

If a $3$-manifold $M$ is geometrically atoroidal but not 
homotopically atoroidal then the proof of the Seifert 
fibered conjecture (see, e.g., \cite{Calegari:book}) implies that 
$M$ is closed and Seifert.
The base surface has to be a sphere with $3$-singular fibers.
Unless the 
difference between geometric and homotopic atoroidal
is essential we only refer to it as atoroidal.

The JSJ decomposition theorem implies that compact,
irreducible,  and orientable
3-manifolds admit a decomposition into finitely many
pieces, which are either geometrically
atoroidal or Seifert 
\cite{He,Hatcher}.

The following lemma was used 
when establishing minimality of foliations in Seifert and hyperbolic 3-manifolds:

\begin{lemma}\label{l.torus_inball_or_bounds_solid_torus}
 If $T$ is an embedded torus inside an orientable closed hyperbolic $3$-manifold $M$, then $T$ either bounds a solid torus or is contained in a $3$-dimensional ball.
\end{lemma}

\begin{proof}
 This is standard result in $3$-manifold topology, so we only sketch the proof.
Since $M$ is orientable and hyperbolic, $T$ is two sided, and not $\pi_1$-injective. Since $T$ is furthermore embedded, Dehn's lemma \cite{He} implies
that there is a compressing disk $D$. That is, $D$ is 
embedded and $D\cap T = \partial D$. Cutting
$T$ along $D$ and capping with two copies of $D$ produces a sphere.
Since $M$ is hyperbolic, it is irreducible, so it follows 
that the sphere bounds a ball. This implies that either $T$ bounds
a solid torus or $T$ is contained in a $3$-ball
\cite{He}.
\end{proof}

We also use the following consequence of Mostow rigidity

\begin{proposition}\label{p.mostow}
If $M$ is a hyperbolic 3-manifold and $f\colon M \to M$ a homeomorphism, then it has an iterate which is homotopic to identity. 
\end{proposition}

\begin{proof}
Mostow rigidity (see, e.g., \cite{Be-Pe}) implies that every homeomorphism is homotopic to an isometry. 
By Theorem \cite[Theorem C.5.6]{Be-Pe} isometries of a closed
hyperbolic $3$-manifold are torsion elements, implying the result.
\end{proof}

\section{Taut foliations in 3-manifolds}\label{ap.tautfol}

All the foliations considered in this article are continuous foliations, with $C^1$ leaves, tangent to a continuous distribution of a $3$-manifold (so they are foliations of regularity $C^{0,1+}$ in the terminology of \cite{CandelConlonI}). In this appendix, all foliations are $2$-dimensional.

A foliation on $M$ is called \emph{taut} if it admits a closed transversal that intersects every leaf of $\cT$.

Let $\lt$ denote the lift of the foliation $\cT$ to $\mt$. The \emph{leaf space} of $\cT$ is defined as the set $\cL_{\lt}:= \mt/_{\lt}$ equipped with the quotient topology.

The following theorem gathers some known properties of taut foliations (see, for instance, \cite[Chapter 4]{Calegari:book} for the proofs) and relies particularly on the celebrated theorems by Novikov and Palmeira.
\begin{theorem}\label{thm-tautfoliation} 
A foliation without compact leaves in a 3-manifold $M$ is taut.

If $M$ is a 3-manifold that is not finitely covered by $S^2 \times S^1$ and admitting a taut foliation $\cT$\footnote{Note that since $M$ is not finitely covered by $S^2 \times S^1$, no leaves of $\cT$ can be a sphere or a projective plane.} then $\mt$ is homeomorphic to $\R^3$. Moreover, every leaf of $\cT$ lifts to a plane $L \in \lt$ which is properly tamely embedded in $\mt$ and separates $\mt$ in two half spaces.

The leaf space $\cL_{\lt}$ is a one-dimensional  (non necessarily Hausdorff), simply connected (separable) manifold.
Furthermore, every point in $\cL_{\lt}$ is contained in the interior
of an interval in $\cL_{\lt}$.

In particular, if $\beta$ is a transersal to $\lt$, then $\beta$ intersects a leaf of $\lt$ at most once.
\end{theorem}

When $\mathcal L_{\lt}$ is Hausdorff, then it is homeomorphic
to the real numbers $\R$. In this case, the foliation $\cT$ is called $\R$-\emph{covered}.

Since $\mt$ is simply connected, $\lt$ is transversely orientable (but deck transformations of $\mt$ may flip this transverse orientation).

For reference, we cite the following result that we used in this article (for the $C^0$ version of the arguments used, see \cite[\S 9]{CandelConlonI}, or \cite[Lemma 7.21]{Calegari:book}).

\begin{proposition}[Rosenberg]\label{Ros}
Let $M$ be a closed $3$-manifold which is not homeomorphic to ${\TT}^3$,
and let $\cT$ be a foliation on $M$. Then some leaf of $\cT$ is not a plane. 
\end{proposition}

\section{Uniformization of leaves}

The following result is very helpful
to understand the action of deck transformations inside 
leaves of the foliations of a partially hyperbolic
diffeomorphism. See, e.g., \cite{Calegari:book} for a proof.

\begin{theorem}[Candel]\label{thm-leafuniformisation} 
Let $\cF$ be a taut foliation of a 3-manifold $M$ and assume that it has no transverse invariant measure. Then, there is a metric in $M$
which restricts to a ($2$-dimensional) hyperbolic metric in each leaf of $\cF$.
\end{theorem}

A transverse invariant measure is an assignment of a non-negative
number to each arc transverse to $\cF$, such that it satisfies
the properties of measures under countable union and restriction.
In addition the measure is unchanged if we homotope the
transverse arcs keeping each point in its respective leaf.
The statement of Candel's Theorem gives further properties on the transverse invariant measure (also called holonomy invariant transverse measure), 
stating that it has zero Euler characteristic, but we will avoid defining this (see \cite{Calegari:book} for a detailed treatment).

It is well known that every taut foliation in a hyperbolic 3-manifold (see e.g., \cite{Calegari:book}) or the horizontal foliations in a Seifert manifold used
in this article satisfy the conclusion of Theorem \ref{thm-leafuniformisation}. We remark that it is possible to show that any minimal foliation on a manifold with non-virtually solvable fundamental group satisfies the hypothesis of Theorem \ref{thm-leafuniformisation} (see \cite[Section 5.1]{FenleyPotrie}).

\section{Uniform foliations and transverse pseudo-Anosov flows} \label{app.regulatingpA}

Uniform foliations were introduced by Thurston \cite{Thurston2},
and have been intensively studied, particularly when $M$ is a hyperbolic 3-manifold. They are intimately related to the notion of \emph{slitherings} (see \cite{Thurston2} or \cite[Chapter 9]{Calegari:book}). 

\begin{definition}\label{def.uniform}
 An  $\R$-covered foliation $\cT$ is called \emph{uniform} 
if the Hausdorff distance between any pair of leaves $L$, $L'$ of 
$\widetilde\cT$ is finite.
That is, there exists $K>0$ (depending on $L$ and $L'$) such that $L \subset B_{K}(L')$ and $L' \subset B_K (L)$ where $B_K(X)$ denotes the set of points at distance less than $K$ from $X \subset \mt$.
\end{definition}

Thurston build a special pseudo-Anosov flow associated with a $\R$-covered foliation in a hyperbolic manifold.

\begin{definition}
 Let $\fol$ be a foliation of a $3$-manifold. A flow $\Psi\colon M \to M$ is called \emph{regulating} for $\fol$ if every orbit of the lifted flow $\wt\Psi$ intersects every leaf of the lifted foliation $\wt\fol$ in the universal cover $\wt M$.
\end{definition}

\begin{theorem}[Thurston, Calegari, Fenley \cite{Thurston2,Calegari00,Fen2002}]\label{teo-transversepA}
A transversely oriented, $\R$-covered,
uniform foliation in a hyperbolic 3-manifold admits a regulating transverse pseudo-Anosov flow $\Phi$. 
Moreover, $\Phi$ can be chosen so that the singular foliations
have $C^1$ leaves outside the prongs.
\end{theorem}

Recall that a pseudo-Anosov flow $\Phi$ is a flow generated by a vector field $X$ which preserves two singular foliations $\Lambda^s$ and $\Lambda^u$ and such that, outside a finite number of singular orbits, the flow is locally modeled on a (topological) Anosov flow (see Appendix \ref{ss.DAF}). The foliations glue along the singularities forming $p$-prongs (with $p\geq 3$). We refer the reader to \cite{Calegari:book} for more details. We note also that every \emph{expansive} flow is orbit-equivalent to a pseudo-Anosov flow \cite{Paternain,InabaMatsumoto}.

Work of Barbot and the second author implies that Thurston's regulating flow is genuinely pseudo-Anosov:

\begin{proposition}\label{prop-transversepAprong}
If $\Phi$ is a  pseudo-Anosov flow regulating and transverse to a uniform foliation in a non-virtually solvable 3-manifold, then, $\Phi$ is not a topological Anosov flow. In particular, there are singular periodic orbits which are $p$-prongs with $p\geq 3$.  
\end{proposition}

\begin{proof}
This fact can be found in \cite{Fen2013}, but we recall the argument: every element of the fundamental group that represents a periodic orbit of 
$\Phi$ acts as a translation on the leaf space of $\wt\fol$. However, this is inconsistent with the fact that every Anosov flow on a $3$-manifold, except for suspensions of Anosov diffeomorphisms (which do not exist on non-virtually solvable manifolds), admits pairs of periodic orbit that are freely homotopic to the inverse of each other (this fact follows from work of Barbot and the second author, see \cite[Theorem 2.15]{BBGRH}).
\end{proof}

\section{Axes}\label{ss.axes}

Here, we recall some needed results from the theory of \emph{axes} for free actions on one-dimensional, non-Hausdorff, simply connected, manifolds. These results
extend similar results for trees. 
We refer the reader to \cite{Fen2003,Bar98} for a more detailed account.
All of the results we state are true for homeomorphisms of one-dimensional, non-Hausdorff, simply connected, separable manifolds. However, to keep the terminology close to the core of this article, we phrase our results in the setting of homeomorphisms preserving a foliation. 

Let $L$ be a complete plane. Let $\cC$ be a foliation such that its leaf space $\cL_\cC$ is a one-dimensional (not necessarily Hausdorff) simply connected manifold. 

\begin{definition}[Axis of a foliation-preserving homeomorphism]\label{d.axis}
 Let $g\colon L \to L$ be a homeomorphism which preserves $\cC$.
 The \emph{axis} (or  $\cC$-\emph{axis}) of $g$ is the set of leaves $c\in \cC$ such that $g(c)$ separates $c$ from $g^2(c)$. 
\end{definition}

For the statement below, we recall that a $\ZZ$-union of intervals means an ordered set consisting of countably many closed (possibly degenerate) intervals which are ordered according to $\ZZ$. 

\begin{proposition}\label{proposition-axes}
Let $g\colon L \to L$ be a homeomorphism that preserves $\cC$ without leaving any leaf of $\cC$ fixed. 
Then, the $\cC$-axis for the action of $g$ in $\cL_{\cC}$ is non-empty. 
In addition, the axis is either a line or an ordered $\ZZ$-union of intervals. 
In the second case, the axis is $\cup I_i$, where $I_i = [x_i,y_i]$ is
a closed interval and $y_i$ is not separated from $x_{i+1}$ in the leaf space of
$\cC$.

Moreover, suppose that $g,h\colon L \to L$ are two $\cC$-preserving homeomorphisms that do not fix any $\cC$-leaves, and that share the same axis. If the group generated by $g$ and $h$ acts freely\footnote{Recall that we say that a group acts \emph{freely} if no element different from the identity has a fixed point.} on this axis, then it is abelian.
\end{proposition}

\begin{proof}This is proven in section 3 of \cite{Fen2003}. See in particular Lemma 3.5, Theorem 3.8 and Proposition 3.10 there.
The last statement uses H\"older's Theorem (see e.g., \cite[Theorem 2.90]{Calegari:book})
to deduce that the group generated by $g$ and $h$ must be abelian. 
\end{proof}

\begin{remark}\label{r.axis}
Two commuting homeomorphisms that induce a free action of $\mathbb{Z}^2$ have the same axis (see \cite[Section 2]{Bar98} or \cite[Section 3]{Fen2003}).
We remark a couple of subtle points: 
\begin{enumerate}
\item The fact that $f$ acts freely
does not imply that any non-trivial power of $f$ acts freely, and in fact there
are easy counterexamples, 
\item  Unlike in the case of trees, $f$ acting
freely does not necessarily imply that the axis is properly
embedded in the leaf space. If the axis is a bi-infinite union of intervals, then
it is properly embedded (see Lemma \ref{l.union_intervals_axis_are_properly_embedded} below). If the axis is the reals, it may fail to be
properly embedded, even if all powers of $f$ act freely.
\end{enumerate}
\end{remark}

Notice that as a consequence we obtain:
\begin{corollary}\label{coro-abelianaxis}
Let $f,g,h\colon L \to L$ be three $\cC$-preserving homeomorphisms such that both $\langle f, g\rangle$ and $\langle f, h\rangle$ induce free actions of $\mathbb{Z}^2$ in the $\cC$-leaf space. Assume moreover that $f$ commutes with $g$ and with $h$. Then, the group generated by $f,g,h$ is abelian. 
\end{corollary}

\begin{proof}
Notice that as $f$ commutes with $g$ and $h$ then all three homeomorphisms of the leaf spaces share the same axis (\cite[Section 2]{Bar98}). Now the result follows from the previous proposition.
\end{proof}

Another useful fact about axes is the following:

\begin{lemma}\label{l.union_intervals_axis_are_properly_embedded}
 If $A$ is the axis of a $\cC$-preserving homeomorphism $f$ and $A$ is a $\ZZ$-union of intervals, then $A$ is properly embedded in the leaf space of $\cC$.
\end{lemma}

\begin{proof}
 Let $A= \cup_{\ZZ} I_i$, where $I_i = [x_i,y_i]$, with $y_i$ and $x_{i+1}$ not separated in the leaf space of $\cC$.
 
 If $A$ is not properly embedded, then there exists a leaf $c \in \cC$ such that $(x_i)$ and $(y_i)$ converges to $c$ as, say, $n$ goes to $+\infty$. Now, for any $i$, the interval $I_i$ separates $I_{i-1}$ from $I_{i+1}$. Thus, if $\tau$ is a transversal to $\cC$ through $c$, then $\tau$ intersects every $I_i$, for $i$ big enough. So in particular, for some $i$ big enough, $\tau$ intersects both $y_i$ and $x_{i+1}$. But this is impossible since these two leaves are not separated.
\end{proof}

\section{On partial hyperbolicity} \label{app.partial_hyperbolicity}

Here we state some facts about partial hyperbolicity that are used in the article but are well-known to the experts.

Recall that a $C^1$-diffeomorphism $f \colon M \to M$ of a $3$-manifold, is partially hyperbolic if there exists a $Df$-invariant splitting $TM = E^{\mathrm{s}} \oplus E^{\mathrm{c}} \oplus E^{\mathrm{u}}$ into $1$-dimensional bundles and an $n>0$ such that for every $x\in M$ we have

$$ \|Df^n|_{E^{\mathrm{s}}(x)}\| < \min\{1, \|Df^n|_{E^{\mathrm{c}}(x)}\|\} \leq \max \{ 1, \|Df^n|_{E^{\mathrm{c}}(x)}\|\} < \|Df^n|_{E^{\mathrm{u}}(x)}\|. $$

By changing the Riemannian metric, one can always assume that $n=1$ (see \cite{Gourmelon,CP}). 

A partially hyperbolic diffeomorphism is called \emph{dynamically coherent} if there exists $f$-invariant foliations $\cW^{\cs}$ and $\cW^{\cu}$ tangent to $E^{\cs}= E^{\mathrm{s}} \oplus E^{\mathrm{c}}$ and $E^{\cu}=E^{\mathrm{c}} \oplus E^{\mathrm{u}}$. Taking the intersection of $\cW^{\cs}$ and $\cW^{\cu}$ gives a one-dimensional foliation $\cW^{\mathrm{c}}$ tangent to $E^{\mathrm{c}}$ and $f$-invariant. Note that these foliations are not assumed to be unique in any sense (see \cite{BuW} for a discussion).

Partially hyperbolic diffeomorphisms need not be dynamically coherent, but when they are, the standard notion of equivalence (which goes back to \cite{HPS}) is that of \emph{leaf conjugacy}: Two dynamically coherent partially hyperbolic diffeomorphisms $f\colon M \to M$ and $g\colon N \to N$ are said to be \emph{leaf-conjugate} if there exists a homeomorphism $h\colon M \to N$ that maps the center foliation $\cW^c_f$ of $f$ to the center foliation $\cW^c_g$ of $g$. More precisely, $h$ is such that $h(\cW^c_f(f(x)))= \cW^c_g(g(h(x)))$ for all $x\in M$. We refer the reader to \cite{PotrieICM} and references therein for more discussions. 

We state the following result of Hertz, Hertz and Ures in a way that fits our particular needs.

\begin{theorem}[\cite{HHU-tori}]\label{thm-HHUtori} 
Let $f \colon M \to M$ be a partially hyperbolic diffeomorphism admitting a compact manifold\footnote{Notice that a compact manifold tangent to $E^{\cs}$ is necessarily a torus, see e.g., \cite{HHU-tori}.} tangent to $E^{\cs}$ (or $E^{\cu}$). Then, $M$ has solvable fundamental group (indeed, it is a torus bundle over the circle). 
\end{theorem}

In particular, if the fundamental group of $M$ is not virtually solvable, and $f$ is dynamically coherent, then the center-stable and center-unstable foliations are taut.

Hence, using the fundamental results of Burago and Ivanov \cite{BI} one gets:

\begin{corollary}[\cite{BI,Parwani}]\label{c.aesferical}
Let $M$ be a $3$-manifold with non-solvable fundamental group.
Suppose that $M$ admits a partially hyperbolic diffeomorphism $f$ such that the
bundles $E^{\mathrm{s}}, E^{\mathrm{c}}, E^{\mathrm{u}}$ are orientable and $Df$ preserves these orientations.
Then $M$ admits a taut foliation. 
In particular, $M$ is irreducible, and aspherical. 
\end{corollary}

We recall that when $\pi_1(M)$ is (virtually) solvable a complete classification of partially hyperbolic diffeomorphisms is known \cite{HP-Nil,HP-Sol,HP-Tori}. 

In the setting of Seifert manifolds we used the following result from \cite{HaPS} which is partially based on the study of horizontal and vertical laminations in Seifert manifolds \cite{Brit}. 

\begin{theorem}[\cite{HaPS}]\label{thm-horizontalseif} Let $f \colon M \to M$ be a dynamically coherent partially hyperbolic diffeomorphism on a Seifert manifold $M$ whose fundamental group is not (virtually) solvable. Then, $M$ is a finite cover of $T^1S$ where $S$ is a 2-dimensional hyperbolic orbifold, and the center-stable and center-unstable foliations of $f$ are \emph{horizontal}. That is, there exists a Seifert fibration $p\colon M \to \Sigma$ for which every leaf of the center-stable and center-unstable ($2$-dimensional) foliations is transverse to the ($1$-dimensional) fibers.
\end{theorem}

An invariant foliation is called \emph{$f$-minimal} if the only
non-empty, saturated, closed set invariant by $f$ is the whole manifold.  
The following result motivates asking for $f$-minimality of the foliations as a hypothesis as it covers the most important (from a dynamical standpoint) cases. 

\begin{proposition}[see Lemma 1.1 of \cite{BW}]\label{p.BWminimal}
Let $f \colon M \to M$ be a dynamically coherent partially hyperbolic diffeomorphism. If $f$ is either volume-preserving or transitive,\footnote{It in fact suffices that $f$ be \emph{chain-recurrent}, that is, if a non-empty open set $U$ verifies that $f(\overline{U})\subset U$ then $U=M$, see \cite{CP} for equivalences.} then the center-stable and center-unstable foliations are
$f$-minimal. 
\end{proposition}

\begin{proof} Assume that there is a compact, non-empty $f$-invariant set $\Lambda$ saturated by center-stable leaves. If $\Lambda \neq M, \emptyset$ then it must be a repeller, so $f$ cannot be transitive nor volume-preserving. 
\end{proof}

We remark that the property of
$f$-minimality of $\cW^{\cs}$ and $\cW^{\cu}$ is a strictly weaker hypothesis than (chain-)transitivity (as seen, for instance, in the examples of \cite{BonattiGuelman}).

Finally, we recall the classification of partially hyperbolic diffeomorphisms in manifolds with virtually solvable fundamental group under the assumption that $f$ is homotopic to the identity (see \cite{HP-Nil,HP-Sol,HP-Tori} for the general case): 

\begin{theorem}\label{teo.solv}
Let $f: M \to M$ be a partially hyperbolic diffeomorphism homotopic to identity in a 3-manifold with virtually solvable fundamental group. Then $M$ is not Seifert-fibered and if $f$ is dynamically coherent then it is a discretized Anosov flow. Moreover, if there are no tori tangent to $E^{\cs}$ or $E^{\cu}$ then $f$ is dynamically coherent. 
\end{theorem}

\section{Discretized Anosov flows}\label{ss.DAF}

Let $\varphi_t\colon M \to M$ be a continuous flow generated by a continuous vector field $X = \frac{\partial \varphi_t}{\partial t}|_{t=0}$. It is called a \emph{topological Anosov flow} if it preserves two topologically transverse codimension one continuous\footnote{We emphasize here that we do not require a priori the foliations to have $C^1$-leaves.} foliations $\cF^{ws}$ and $\cF^{wu}$ (called \emph{weak stable} and \emph{weak unstable}) such that:
\begin{enumerate}[label=(\roman*)]
 \item  For every pair of points $x,y \in \cF^{ws}$ (resp.~$x,y \in \cF^{wu}$), there exists an increasing continuous reparametrization $h\colon \R \to \R$ so that $d(\varphi_t(x), \varphi_{h(t)}(y)) \to 0$ as $t \to +\infty$ (resp.~as $t\to -\infty$);\label{item.topoAnosovForward}
 \item  There exists $\eps > 0$ such that for every $x,y \in \cF^{ws}$ (resp.~$x,y \in \cF^{wu}$) not on the same orbit, there exists $t\leq0$ (resp.~$t\geq 0$) such that $d(\varphi_t(x), \varphi_{t}(y)))>\eps$.\label{item.topoAnosovBackwards}
\end{enumerate}
As mentioned earlier in Appendix \ref{app.regulatingpA}, thanks to the work of Paternain \cite{Paternain} and Inaba and Matsumoto \cite{InabaMatsumoto}, the definition of topological Anosov flow can be replaced by asking for the flow to be expansive and to preserve two (non-singular, i.e., without prongs) foliations.
Note also that just condition \ref{item.topoAnosovForward} is not enough for a flow to be topological Anosov as condition \ref{item.topoAnosovForward} does not imply condition \ref{item.topoAnosovBackwards}. 

Conditions \ref{item.topoAnosovForward} and \ref{item.topoAnosovBackwards} allow one to obtain the same classical results as for Anosov flows (e.g., there are no closed $\cF^s$ or $\cF^u$ leaves; the foliations $\cF^{ws}$ and $\cF^{wu}$ are taut; the leaves are planes, annuli or M\"{o}bius bands --- these last two possibilities arising only when the leaves contain a periodic orbit; periodic points are dense in the non-wandering set, etc., see \cite{BarbotHDR} and references therein).

We say that a diffeomorphism $f\colon M \to M$ is a \emph{discretized Anosov flow} if there exists a topological Anosov flow $\varphi_t \colon M \to M$ and a continuous function $\tau\colon M \to \R_{>0}$ such that $f(x) = \varphi_{\tau(x)}(x)$. 

\begin{remark}
 Note that one should not expect a discretized Anosov flow to be the time-$1$ map of a reparametrization of the Anosov flow. To see this, one can for instance restrict to a periodic orbit of the flow: Then the discretized Anosov flow $f$ restricts to a diffeomorphism of the circle. But diffeomorphisms of the circle do not all come from the time-$1$ map of a non-singular flow on the circle.
\end{remark}

The following result relates the notion of discretized Anosov flows with the usual form of equivalence between partially hyperbolic systems.

\begin{proposition}\label{prop.equivDALC}
Let $f\colon M \to M$ be a partially hyperbolic diffeomorphism. The following are equivalent: 
\begin{enumerate}[label=\rm{(\arabic*)}]
\item \label{item.discretizedAF} $f$ is a discretized Anosov flow;
\item \label{item.leaf_conjugate} $f$ is dynamically coherent, the center leaves are fixed by $f$ and the center foliation is the flow line foliation of a topological Anosov flow.
\end{enumerate}
\end{proposition}

\begin{proof}
The fact that the second condition implies the first follows from arguments in \cite{BW}, as was done in section \ref{ss.topoAnosov}. 

The newer result is the other implication, which we now prove. 
Let $\varphi_t\colon M \to M$ be a topological Anosov flow and $\tau\colon M \to \R_{>0}$ be the positive continuous function such that $f(x)=\varphi_{\tau(x)}(x)$. 
Let $F$ be the distribution generated by the vector field $X$ generating $\varphi_t$. First, we claim that $F=E^{\mathrm{c}}$. To prove this we will first show that $F$ cannot be equal to $E^{\mathrm{s}}$ or $E^{\mathrm{u}}$ at any point and then deduce that $F$ has to be $E^{\mathrm{c}}$.

Suppose that 
there is $x\in M$ such that $F(x)=E^{\mathrm{s}}(x)$. Then, the invariance of the orbits of the flow $\varphi_t$ by $f$ together with the uniqueness of the stable manifold implies that there is an arc $I = \varphi_{[-\eps,\eps]}(x)$ of the orbit of $x$ by $\varphi_t$ which is tangent to $E^{\mathrm{s}}$. 
That fact is proven in \cite{CP}, we give here a brief explanation and the precise references.
Consider a small cone field
around $E^{\mathrm{s}}$. Let $\alpha$ be the orbit of $\varphi^t$ through $x$.
Since $F(x)=E^{\mathrm{s}}(x)$, it follows that, near $f^n(x)$, the curves $f^n(\alpha)$ are uniformly Lipschitz and their tangent are inside the cone field around $E^{\mathrm{s}}$.
Notice further that the family of curves $\{ f^n(\alpha) \}$ is also invariant under $f$.
The uniqueness of the stable manifold implies that Lipschitz
curves that are invariant and inside the cone have to be
the stable manifold near the point (this is done in \cite[Sections 4.2 and 4.3]{CP}). Hence $f^n(\alpha)$ must contain an open interval inside the stable manifold near $x$.

Iterating $I$ backwards, we get that the length of $f^{-n}(I)$ grows exponentially, contradicting the continuity of $\tau\colon M \to \R_{>0}$. 

Thus $F$ is never tangent to $E^{\mathrm{s}}$. The same argument shows that it is never tangent to $E^{\mathrm{u}}$. 

Now, suppose that there is a point $y$ such that $F$ is not inside $E^{\cs}$ at $y$. Then applying $Df^n$ to $F(y)$ will get $F(f^n(y))$ closer and closer to $E^{\mathrm{u}}(f^n(y))$. Hence, for any point $z$ in the $\omega$-limit set of $y$, one has that $F(z) = E^{\mathrm{u}}(z)$, contradicting the above. So $F$ is everywhere inside $E^{\cs}$ and, by the same argument, also inside $E^{\cu}$. Thus $F = E^{\mathrm{c}}$ everywhere.

The last step is to show that $f$ is dynamically coherent, for this, we use the fact that the strong stable saturation of a center curve is tangent to $E^{\mathrm{s}} \oplus E^{\mathrm{c}}$ (see \cite[Proposition 3.1]{BI}). We stress that a
center curve here means a curve whose tangent everywhere is in the
center bundle $E^{\mathrm{c}}$. Since $f$ is a discretized Anosov flow, the saturation of a flow line by strong stable leaves of $f$ is contained in the weak stable foliation of the Anosov flow $\varphi_t$ (cf. the proof of Proposition \ref{p.leafconjugacy}).
In particular this implies that the weak stable foliation $\cF^{ws}$ of $\varphi_t$ (which a priori could be only continuous) has $C^1$ leaves and is everywhere tangent to $E^{\mathrm{s}} \oplus E^{\mathrm{c}}$. This establishes dynamical coherence and completes the proof.  
\end{proof}

We remark that a long-standing conjecture (see \cite{BW}) states that every topological Anosov flow is orbit-equivalent to an Anosov flow. If this conjecture is true, then condition \ref{item.leaf_conjugate} above is equivalent to saying that $f$ is dynamically coherent and \emph{leaf-conjugate} to the time-one map of an Anosov flow. We remark here that it has been recently announced that this conjecture is true in the setting of transitive topological Anosov flows (Shannon \cite{Shannon}).

\section{The graph transform argument}\label{app.graph_transform}

We give here an application of the general graph transform technique to the particular case we needed it in.

We call \emph{center-stable plane} any embedded $C^1$-plane tangent to $E^{\mathrm{s}} \oplus E^{\mathrm{c}}$ in $\mt$. Notice that by unique integrability of $E^{\mathrm{s}}$ there is always a stable foliation inside a center-stable plane. 

\begin{lemma}[Graph transform lemma]\label{l.grapht}
Let $f$ be a partially hyperbolic diffeomorphism in $M$. Suppose
that $L \subset \mt$ is a center-stable plane which is fixed by
a lift $\hat f$ of $f$ to $\mt$, and by some $\gamma \in \pi_1(M) \smallsetminus \{id\}$.
Assume that there is a properly embedded $C^1$ curve $\eta$
transverse to the stable foliation in $L$ and such that
\[
 \gamma \eta = \eta \quad \text{and} \quad \hat f(\eta) \subset \bigcup_{z \in \eta} \widetilde{\cW}^s(z).
\]
Then in $L$ there is a curve
$\hat \eta$ which is fixed by both $\hat f$ and $\gamma$ and is everywhere tangent to $E^{\mathrm{c}}$. 
\end{lemma} 

Notice the subtlety in the conclusion of this lemma: The curve $\hat \eta$ produced is tangent to the center direction, however, it may \emph{not} be a center leaf since the bundle may not be uniquely integrable. 

\begin{remark}\label{r.condGT} 
The second hypothesis of the lemma is equivalent to saying that the union
$\bigcup_{z \in \eta} \widetilde{\cW}^s(z)$
is invariant by $\hat f$. In particular, all positive and negative images of $\eta$ by powers of $\hat f$ are contained in this union.
To see this, notice that, calling $\alpha$ the projection of $\eta$ to the cylinder $L /_{< \gamma >}$, the second condition implies that 
$f_1(\alpha)$ is freely homotopic to $\alpha$ (because $L /_{< \gamma >}$ is a cylinder), and $\alpha$ (or $f_1(\alpha)$) is not null homotopic in this cylinder.
Therefore $\bigcup_{z\in \eta} \widetilde{\cW}^s(z) = \bigcup_{z \in \hat f(\eta)} \widetilde{\cW}^s(z)$. 
Thus $\bigcup_{z\in \eta} \widetilde{\cW}^s(z)$ is $\hat f$ invariant. The converse is immediate.
\end{remark}

\begin{proof}  We work in the quotient $L/_{<\gamma>}$
which is an annulus on which $\eta$ projects to a closed $C^1$-circle transverse 
to the stable foliation, denoted by $\alpha$. 
Let $\pi_0\colon L \rightarrow L /_{< \gamma >}$ the quotient map.
Let $f_1$ be the induced diffeomorphism on $L/_{<\gamma>}$. 

Up to a small modification of $\alpha$ if necessary, we can assume that 
$\alpha$ is simple, that is, it goes around the
cylinder $L/_{<\gamma>}$ once.

We parametrize $\alpha$ in $L/_{<\gamma>}$ by arclength (for the leaf-wise path metric on $L /_{< \gamma >}$). Then, we parametrize $\bigcup_{z\in \alpha} \cW^s(z)$ as a cylinder $S^1 \times \mathbb{R}$ contained
in $L/_{<\gamma>}$,  where
$\alpha$ is the zero section and the stable leaves are parametrized by arclength. 

Since all $\hat f^n(\eta)$ are in $\bigcup_{z\in \eta} \cW^s(z)$, we can express $\eta_n := \hat f^n(\eta)$ as graphs of $C^1$-functions\footnote{In this specific case with one-dimensional center, one can assume that the stable foliation is $C^1$ inside center-stable leaves so that this makes sense, see \cite[Section 4.7]{CP}. If the stable foliation is less regular then one can go through with the proof by taking a smooth approximating foliation instead, and the arguments would be essentially the same.} from $S^1$ to $\mathbb{R}$ 

We want to show that $\hat f$ acts as a contraction on curves transverse to the stable foliation in (at least a compact part of) $\bigcup_{z\in \alpha} \cW^s(z) \simeq S^1\times \R$. 	
First we show that all the $\eta_n$ stay in a compact subset of $\bigcup_{z\in \eta} \cW^s(z)$ (and thus also of $L/_{<\gamma>}$).

Our assumptions imply that there exists some $a_0>0$, such that $f_1(\alpha)$ is contained
in an $a_0$ stable neighborhood of $\eta$. That is, in the union of stable segments of length $2 a_0$ centered at $\eta$.

Let $\lambda < 1$ be the smallest contraction factor for $f_1$ along
stable leaves. It follows that $f^2_1(\eta)$ is contained
in the stable neighborhood of size $a_0 + \lambda a_0$ around $\alpha$,
and so on.
Thus, we immediately get that, for all $n$, $f_1^n(\alpha)$ is contained in a compact subset of the annulus.

Now that we know that all the $\eta_n$ curves are contained in a compact subset, we can 
use the fact that $f_1$ contracts stable leaves more than centers to prove the following:

There exists some constant $a_1$ such that $f_1$ globally preserves the space of uniformly bounded (for some appropriately large bound) Lipschitz functions from $S^1$ to $\mathbb{R}$ with Lipschitz constant less than $a_1$. By standard computations, one can see that this acts as a contraction on this complete metric space (this is usually called the \emph{graph transform technique} see e.g., \cite{HPS} or \cite[Section 4.2]{CP} for a more detailed study of this technique and the reason for considering Lipschitz functions).

Therefore, one obtains that there is a unique fixed point of this 
action which corresponds to the graph of a Lipschitz function from $S^1$ to $\mathbb{R}$ which is the 
unique invariant Lipschitz graph under $f_1$. It is also standard to show that the tangent cones at each point must actually be degenerate (see \cite[Section 4.2]{CP}), i.e., the invariant curve is $C^1$.
Moreover, since $E^{\mathrm{c}}$ is the only invariant bundle transverse to $E^{\mathrm{s}}$, the curve must be everywhere tangent to $E^{\mathrm{c}}$. 
The lift of this curve to $L$ is the curve we sought. 
\end{proof}

Under the assumptions of the graph transform lemma (Lemma \ref{l.grapht}), another thing we easily deduce is that there must be a periodic center leaf of $f$ in the projection of the leaf $L$:
\begin{lemma}\label{l.if_graph_transform_then_periodic_center}
 Let $f$ be a partially hyperbolic diffeomorphism in $M$. Suppose
that $L \subset \mt$ is a center-stable plane which is fixed by
a lift $\hat f$ of $f$ to $\mt$, and by some $\gamma \in \pi_1(M) \smallsetminus \{id\}$.

Assume that there exists a curve $\hat\eta$ that is fixed by both $\hat f$ and $\gamma$. Then there exists a center leaf $c$ in $L$ and two integers $n,m$, with $m\neq 0$, such that
$c = \gamma^n \hat{f}^m c$.
\end{lemma}

To prove this lemma, we need to use the center leaf space on $L$. When the foliations are branching (i.e. in the non-dynamically coherent setting), the center leaf space will be defined in \cite{BFFP-sequel}. 

\begin{proof}
 Let $\eta = \pi(\hat\eta)$ be the projection of $\hat\eta$ to $M$. Since $\hat\eta$ is invariant by $\gamma$ and $\hat f$, the curve $\eta$ is a circle on which $f$ acts.
 
 Suppose first that $f$ has a periodic point on $\eta$. Then, there exists a center leaf through that point that is periodic, as claimed. 
 
 Otherwise, there exists a point in $\eta$ that is inside its $\omega$-limit set for $f$.
 
 Lifting back to the universal cover, this means that there exists $x\in \hat\eta$ such that there exists integers $m,n$, with $m$ arbitrarily large, such that $x$ and $\gamma^n \hat{f}^m(x)$ can be made arbitrarily close.
 
 Let $\tau$ be a compact segment of the stable leaf through $x$. Since $\hat f$ contracts the length of stable segments, we can choose $m\in \mathbb{N}$ large enough so that every center leaf through $\gamma^n \hat{f}^m (\tau)$ intersects \emph{the interior} of $\tau$. (This is possible as $\gamma^n \hat{f}^m (\tau)$ can be chosen arbitrarily small and arbitrarily close to $x$, which is in the interior of $\tau$.)
 
 Let $\cL^c_L$ be the leaf space of the center foliation in $L$. 
 Let 
 \[
  \tau_c = \left\{ c\in \cL^c_L \mid c \cap \tau \neq \emptyset \right\}.
 \]
 Notice that $\tau_c$ is a compact interval in the $1$-manifold $\cL^c_L$.
 
 Then consider the function $h\colon \tau_c \to \cL^c_L$ defined by $h(c) = \gamma^n \hat{f}^m (c)$. The map $h$ is continuous, and, thanks to our choice of $m$, $h(\tau_c)$ is contained in the interior of $\tau_c$. Hence, there exists $c_0\in \tau_c$ that is fixed by $h$, as claimed.
\end{proof}

\end{appendix}
%%%%%%%%%%%%%%%%%%%%%%%%%%%%%%%%%%%%%%%%%%%%%
%%%%%%%%% End of article text %%%%%%%%%%%%%%%
%%%%%%%%%%%%%%%%%%%%%%%%%%%%%%%%%%%%%%%%%%%%%
% \backmatter
\bibliographystyle{amsalpha_for_Rafael}
\bibliography{biblio}

\end{document}